\documentclass{amsart}
\usepackage{amsmath,amssymb}
\usepackage{graphicx,subfigure}
\usepackage[latin1]{inputenc}
\newtheorem{theorem}{Theorem}[section]
\newtheorem{proposition}[theorem]{Proposition}
\newtheorem{corollary}[theorem]{Corollary}

\newtheorem{lemma}[theorem]{Lemma}

\theoremstyle{definition}
\newtheorem{definition}[theorem]{Definition}

\theoremstyle{remark}
\newtheorem{remark}[theorem]{Remark}

\numberwithin{equation}{section}

\newcommand{\al}{\alpha}
\newcommand{\be}{\beta}
\newcommand{\de}{\delta}
\newcommand{\ep}{\epsilon}

\newcommand{\ga}{\gamma}
\newcommand{\ka}{\kappa}
\newcommand{\la}{\lambda}
\newcommand{\om}{\omega}
\newcommand{\si}{\sigma}
\newcommand{\te}{\theta}
\newcommand{\vp}{\varphi}

\newcommand{\De}{\Delta}
\newcommand{\Ga}{\Gamma}
\newcommand{\La}{\Lambda}
\newcommand{\Si}{\Sigma}
\newcommand{\Om}{\Omega}
\newcommand{\bs}{\mathbf{s}}

\newcommand{\hu}{\hat{u}}

\newcommand{\hPi}{\hat{\Pi}}

\def\NN{\mathbb{N}}
\def\RR{\mathbb{R}}
\def\BB{\mathbb{B}}
\def\ZZ{\mathbb{Z}}

\renewcommand\SS{\mathbb{S}}
\newcommand{\cB}{{\mathcal B}}

\newcommand{\cD}{{\mathcal D}}

\newcommand{\cF}{{\mathcal F}}

\newcommand{\cH}{{\mathcal H}}

\newcommand{\cM}{{\mathcal M}}
\newcommand{\cN}{{\mathcal N}}
\newcommand{\cO}{{\mathcal O}}
\newcommand{\cP}{{\mathcal P}}
\newcommand{\cQ}{{\mathcal Q}}
\newcommand{\cU}{{\mathcal U}}
\newcommand{\cR}{{\mathcal R}}

\newcommand{\cT}{{\mathcal T}}

\newcommand{\cW}{{\mathcal W}}
\newcommand{\pd}{\partial}
\newcommand\minus\backslash
\newcommand{\id}{{\rm{id}}}

\newcommand{\ms}{\mspace{1mu}}
\newcommand\lan\langle
\newcommand\ran\rangle

\newcommand{\I}{{\mathrm i}}
\newcommand{\e}{{\mathrm e}}

\DeclareMathOperator\Div{div}

\DeclareMathOperator\dist{dist} 
 \DeclareMathOperator\curl{curl}

\renewcommand\leq\leqslant
\renewcommand\geq\geqslant
\newlength{\intwidth}

\newcommand\loc{_{\mathrm{loc}}}

\def\hv{\hat v}

\def\es{\mathbb S^1_\ell}
\def\DD{\mathbb D^2}
\def\Bpsi{\bar\psi}
\def\Brho{\bar\rho}
\def\A{\mathbb A}
\def\tube{\cT_{\ep}}

\def\bs{\bar s}

\def\teo{\te^{(0)}}

\def\tG{G_{\tilde\La}}
\DeclareMathOperator\ddiv{Div}
\DeclareMathOperator\Curl{Curl_\ep}
\def\th{$^{\text{th}}$}
\begin{document}

\title[Existence of knotted vortex tubes in steady
  Euler flows]{Existence of 
  knotted vortex tubes\\ in steady
  Euler flows}

\author{Alberto Enciso}
\address{Instituto de Ciencias Matemáticas, Consejo Superior de
  Investigaciones Científicas, 28049 Madrid, Spain}
\email{aenciso@icmat.es, dperalta@icmat.es}

\author{Daniel Peralta-Salas}

%
%
\begin{abstract}

  We prove the existence of knotted and linked thin vortex tubes for
  steady solutions to the incompressible Euler equation in
  $\RR^3$. More precisely, given a finite collection of (possibly
  linked and knotted) disjoint thin tubes in~$\RR^3$, we show that
  they can be transformed with a $C^m$-small diffeomorphism into a set
  of vortex tubes of a Beltrami field that tends to zero at
  infinity. The structure of the vortex lines in the tubes is
  extremely rich, presenting a positive-measure set of invariant tori
  and infinitely many periodic vortex lines.  The problem of the
  existence of steady knotted thin vortex tubes can be traced back to Lord
  Kelvin.\\


\end{abstract}
\maketitle

\tableofcontents

\section{Introduction}
\label{S.intro}

The motion of particles in an ideal fluid in $\RR^3$ is described
by its velocity field $u(x,t)$, which satisfies the Euler equation
\[
\pd_t u+(u\cdot\nabla)u=-\nabla P\,,\qquad \Div u=0
\]
for some pressure function $P(x,t)$. Equivalently, the field $u$ satisfies the equation
\[
-\pd_t u+u\times \om=\nabla B\,,\qquad \Div u=0\,,
\]
where the field $\om:=\curl u$ is the {\em vorticity}\/ and $B:=P+\frac12|u|^2$
is the Bernoulli function. The trajectories (or integral curves) of
the vorticity $\om(x,t)$ for fixed~$t$ are usually called {\em vortex
  lines}. A solution $u$ to the Euler equation is called {\em steady}
when it does not depend on time.

A domain in $\RR^3$ that is the union of vortex lines and whose
boundary is an embedded torus is a (closed) {\em vortex tube}\/. The
analysis of {\em thin}\/ vortex tubes, in a sense to be specified below,
for solutions to the Euler equation has attracted considerable
attention. A long-standing problem in this direction is Lord Kelvin's
conjecture~\cite{Kelvin} that knotted and linked thin vortex tubes can arise in
steady solutions to the Euler equation. This conjecture was motivated
by results due to Helmholtz on the time-dependent case, which hinge on
the mechanism of vorticity transport, and Maxwell's observations of
what he called ``water twists''.

Kelvin's conjecture is basically a question on the existence of
knotted invariant tori in steady solutions of the Euler
equation. There is a considerable body of literature devoted to the
analysis of topological and geometrical structures that appear in
fluid flows, which has led to significant results e.g.\ on
particle trajectories and vortex lines~\cite{Ar66,EG00,LS00,Annals,Nadirashvili},
on the relationship between the Euler equation and the group of
volume-preserving diffeomorphisms~\cite{Ar66,Ebin,KM03,Sverak},
and on the connection of the helicity with the energy functional and
the asymptotic linking number~\cite{Arnold,FH91,Ghys,Vogel}. However,
Kelvin's conjecture remains wide open, and indeed has been included as
a major open problem in topological fluid mechanics in the
surveys~\cite{Ri98,Mo00}.

There is strong numerical evidence of the existence of thin vortex
tubes, both in the case of steady and time-dependent fluid flows. As a
matter of fact, thin vortex tubes have long played a key role in the
construction and numerical exploration of possible blow-up scenarios
for the Euler equation, which in turn has led to rigorous results such
as~\cite{CF01,Deng}. A particularly influential scenario in this
direction is~\cite{Pelz}, which discusses how an initial condition
with a certain set of linked thin vortex tubes might lead to
singularity formation in finite time. As a side remark, let us point
out that thin vortex tubes of complicated knot topologies have been
recently constructed experimentally in the fluid mechanics laboratory at
the University of Chicago~\cite{Irvine}.

Our aim in this paper is to prove that there exist steady solutions to
the Euler equation in  $\RR^3$ having thin vortex
tubes of any link and knot type. The steady solutions we construct
are {\em Beltrami fields}\/, that is, they satisfy the equation
\[
\curl u=\la\ms u
\]
in $\RR^3$ for some nonzero real constant $\la$. As any Beltrami field
satisfies the equation $\De u=-\la^2 u$ in $\RR^3$, it is apparent
that the solutions we construct are real analytic but do not have finite energy (i.e., $u$
is not in $L^2(\RR^3)$). However, our construction yields solutions
with optimal decay at infinity in the class of Beltrami fields, which
fall off as $|u(x)|<C/|x|$. In particular, they are in $L^p(\RR^3)$
for all $p>3$.

The motivation to consider the class of Beltrami solutions to the
Euler equation to address the existence of linked vortex tubes comes
from Arnold's structure theorem~\cite[Theorem II.1.2]{AK99}.  Under
mild technical assumptions, this theorem ensures that the vortex lines
of a steady solution to the Euler equation whose velocity field is not
everywhere collinear with its vorticity are nicely stacked in a rigid
structure akin to those appearing in the study of integrable
Hamiltonian systems. Heuristically, this structure should somehow
restrict the way the vortex lines are arranged; partial results in
this direction have been shown in~\cite{EG99}, where it is proved that
under appropriate (strong) hypotheses the vortex lines of steady solutions with
non-collinear velocity and vorticity can only be of certain knot
types. In contrast, using Beltrami fields we have recently managed to produce
steady solutions of the Euler equation with a set of vortex lines
diffeomorphic to any locally finite link~\cite{Annals}.

Before stating the main result we need some definitions. We will say
that a bounded domain of $\RR^3$ is a (closed) {\em tube} if its
boundary is a smoothly embedded torus. Therefore, a vortex tube is a
tube whose boundary is the union of vortex lines (or, equivalently,
its boundary is an invariant torus of the
vorticity field). A convenient way of constructing {\em thin}\/ tubes is as metric neighborhoods
of curves. Indeed, if $\ga\subset\RR^3$ is a closed curve, we will
denote by
\begin{equation}\label{cTep}
\cT_\ep(\ga):=\big\{x\in\RR^3:\dist(x,\ga)<\ep\big\}
\end{equation}
the tube of core $\ga$ and thickness $\ep$. We are interested in the
case where $\ep$ is a small positive number, which corresponds to the
case of thin tubes. Obviously any finite collection of disjoint tubes
in $\RR^3$ can be isotoped to a collection of thin tubes of this form.

Let us consider a finite collection of (possibly knotted and linked)
disjoint thin tubes, constructed using metric neighborhoods of curves
as above. Our main result in this paper is that, if the thickness of the tubes
is small enough, this collection can be transformed by a $C^m$-small
diffeomorphism into a union of vortex tubes of a Beltrami field in
$\RR^3$. Furthermore, the structure of the vortex lines inside each vortex tube
is extremely rich: firstly, the boundary of each vortex tube is far
from being the only invariant torus of the flow, as the set of
invariant tori has positive Lebesgue measure. Secondly, between any
pair of these invariant tori there are infinitely many periodic vortex
lines. Thirdly, there is a periodic vortex line which
is close to the core of the initial tube and diffeomorphic to it. More
precisely, we have the following statement, where (as always
henceforth) all the diffeomorphisms are assumed smooth, and all curves
and surfaces are smoothly embedded in $\RR^3$. It should be emphasized that the
thinness of the tubes is crucially used in the proof of the theorem.

\begin{theorem}\label{T.main}
Let $\ga_1,\dots,\ga_N$ be $N$ pairwise disjoint (possibly knotted and
linked) closed curves in $\RR^3$. For small enough $\ep$, one can
transform the collection of pairwise disjoint thin
tubes $\cT_\ep(\ga_1),\dots, \cT_\ep(\ga_N)$ by a 
diffeomorphism $\Phi$ of $\RR^3$, arbitrarily close to the identity in
any $C^m$ norm, so that $\Phi[\cT_\ep(\ga_1)],\dots, \Phi[\cT_\ep(\ga_N)]$
are vortex tubes of a Beltrami field $u$, which satisfies the equation
$\curl u=\la u$ in $\RR^3$ for some nonzero constant $\la$.

Moreover, the field $u$ decays at infinity as $|D^ju(x)|<C_j/|x|$ for
all $j$ and has the following properties in each vortex tube
$\Phi[\cT_\ep(\ga_i)]$:
\begin{enumerate}
\item In the interior of $\Phi[\cT_\ep(\ga_i)]$ there are uncountably
  many nested tori invariant under the Beltrami field $u$. On each of these
  invariant tori, the field $u$ is ergodic.

\item The set of invariant tori has positive Lebesgue measure in
  a small neighborhood of the boundary $\pd\Phi[\cT_\ep(\ga_i)]$.  

\item In
  the region bounded by any pair of these invariant tori there are infinitely many closed
  vortex lines, not necessarily of the same knot type as the curve~$\ga_i$.

\item The image of the core curve $\ga_i$ under the
  diffeomorphism~$\Phi$ is a periodic vortex line  of $u$. 

\end{enumerate}
\end{theorem}

An important property of the structure of the
vortex lines inside each vortex tube, as described above, is that it
is {\em stable}\/ in the following sense: on the one hand, it is robust under
small perturbations of the field $u$, meaning that the trajectories of
any field which is close enough to $u$ in a sufficiently high $C^k$ norm have the same
structure. On the other hand, the boundary of each vortex tube
$\pd\Phi[\cT_\ep(\ga_i)]$ is Lyapunov-stable under the flow of the Beltrami
field~$u$. It should be noticed too that from Property~(iv) in Theorem~\ref{T.main} we
recover the main theorem of~\cite{Annals} for finite links and improve it
by ensuring that the solution falls off at infinity (while
in~\cite{Annals} we had no control at all on the growth of the solution
at infinity). Besides, the
proof of Theorem~\ref{T.main} shows that the vortex lines
$\Phi(\ga_i)$ are elliptic trajectories, and therefore linearly stable, 
while the vortex lines we constructed in~\cite{Annals} were
hyperbolic, and thus unstable.

After establishing his structure theorem, Arnold
conjectured~\cite{Ar66} that, contrary to what happens in the
non-collinear case, Beltrami fields could present vortex lines of the
same topological complexity as the trajectories of any divergence-free vector
field. By KAM theory, typically these trajectories give rise to a set of
invariant tori with positive measure and chaotic regions with
homoclinic tangles between these tori. Theorem~\ref{T.main}
is fully consistent with this picture, and proves the existence of
the aforementioned positive-measure set of invariant tori.


The paper is organized as follows. In Section~\ref{S.strategy} we will
discuss the strategy of the proof of Theorem~\ref{T.main}. This section, which
serves as a guide to the paper, also allows us to explain
the key difficulties that appear in the proof of this result but not
in that of~\cite{Annals}, which require the introduction of new
techniques and ideas, and make this paper
considerably more involved. In Section~\ref{S.tubes} we introduce some
objects associated with the geometry of a thin tube that will be used
throughout the paper. In Section~\ref{S.Laplacian} we prove some
estimates for the Laplacian in a thin tube with Neumann boundary
conditions. These estimates are used in Section~\ref{S.harmonic} to
study harmonic fields in thin tubes. In
Section~\ref{S.Beltrami} we construct Beltrami fields in thin tubes
with prescribed harmonic part. These fields are
analyzed further in Section~\ref{S.KAM}, where we prove a KAM
theorem for Beltrami fields in generic thin tubes. A Runge-type approximation theorem by global
Beltrami fields tending to zero at infinity is presented in
Section~\ref{S.approx}. With all these ingredients, in
Section~\ref{S.main} we prove Theorem~\ref{T.main}. The paper
concludes with an easy application to the Navier--Stokes equation,
which we present
in Section~\ref{S.remarks}.

\section{Strategy of the proof and guide to the paper}
\label{S.strategy}

To prove Theorem~\ref{T.main}, our basic goal is to establish the existence of a Beltrami field $u$,
satisfying the equation 
\[
\curl u=\la u
\]
in $\RR^3$ for some nonzero constant $\la$ and falling off at
infinity, such that the field $u$ has a set of $N$ invariant tori
diffeomorphic to the surfaces $\{\pd\cT_\ep(\ga_i)\}_{i=1}^N$, with
$\{\ga_i\}_{i=1}^N$ being a set of prescribed (possibly knotted and linked) closed curves. We recall
that $\cT_\ep(\ga_i)$ is a metric neighborhood of the curve
$\ga_i$ of small thickness~$\ep$, as defined in Eq.~\eqref{cTep}. By
deforming them a little if necessary, we can assume without loss
of generality that the curves $\ga_i$ are analytic.

The basic idea behind the proof of Theorem~\ref{T.main} is carried out
in three interrelated stages. Firstly, we construct a {\em local
Beltrami field $v$}, which satisfies the Beltrami equation in a
neighborhood of each closed tube $\overline{\cT_\ep(\ga_i)}$ and has a set of
invariant tori given by $\pd\cT_\ep(\ga_1),\dots,
\pd\cT_\ep(\ga_N)$. Secondly, we prove that these invariant tori are
{\em ``robust''}, meaning that suitably small perturbations of the
local Beltrami field $v$ still have a set of invariant tori
diffeomorphic to $\pd\cT_\ep(\ga_1),\dots, \pd\cT_\ep(\ga_N)$. To
conclude, we show that the local Beltrami field $v$ can be
approximated in any $C^k$ norm by a {\em global Beltrami field $u$}\/,
which satisfies the Beltrami equation in the whole space $\RR^3$ and
falls off at infinity in an optimal way. The robustness of the
invariant tori of the local Beltrami field ensure that $u$ has a set
of vortex tubes diffeomorphic to the initial configuration of thin
tubes, and that in fact this diffeomorphism can be taken close to the
identity in a $C^m$ norm.

However, the implementation of this basic idea turns out to be
extremely subtle. To understand why, one can start by noticing that
the robustness of the invariant tori of the local Beltrami field
relies on a KAM-type argument. To apply this KAM argument, we need a
small perturbation parameter and some control on the dynamics of the
local Beltrami field in a neighborhood of each invariant torus, which
is required in order to show that the local Beltrami field is equivalent to a Diophantine rotation on
the torus and satisfies a suitable nondegeneracy condition.

To construct local Beltrami fields in a neighborhood of the tori
$\pd\cT_\ep(\ga_i)$ with controlled behavior on these surfaces, it is
natural to use some variant of the Cauchy--Kowalewski theorem for the
curl operator with Cauchy data on the tori (e.g.~\cite[Theorem
3.1]{Annals}). This would lead to a local Beltrami field defined in a small
neighborhood~$\Om_i$ of each torus $\pd\cT_\ep(\ga_i)$. Unfortunately,
the Runge-type approximation theorem we prove in this paper does
not allow us to approximate a local Beltrami field defined in
$\Om_1\cup\cdots\cup \Om_N$ by a global Beltrami field, since the
complement $\RR^3\minus (\Om_1\cup\cdots\cup \Om_N)$ has compact connected
components. As is well known, this is not just a technical issue, but
a fundamental obstruction in any Runge-type theorem.

Therefore, to construct the local Beltrami field $v$ we will not
consider a Cauchy problem but a
boundary value problem for the curl operator in each tube. This leads
to a vector field that satisfies the Beltrami equation
\[
\curl v=\la v
\]
in a neighborhood of the tubes (because the boundaries are analytic),
which allows us to apply our Runge-type approximation theorem. In this
boundary value problem one can prescribe the normal component of $v$
at each boundary $\pd\cT_\ep(\ga_i)$, so that by
setting it to zero we can ensure that each boundary is an invariant
torus of the local Beltrami field. However, we have no control on the
tangential component of $v$ on each torus, so in principle the local
Beltrami field does not necessarily satisfy the dynamical conditions
our KAM-type theorem requires for the invariant tori to be preserved
under small perturbations.

To overcome this difficulty we resort to a careful analysis of harmonic
fields in thin tubes. Indeed, we show that in the above boundary value
problem for the local Beltrami field, one can also prescribe its
harmonic part (that is, the $L^2$ projection of $v$ into the space of
harmonic fields with tangency boundary conditions), and for small
$\la$ one expects the local Beltrami field to behave essentially as if
it were harmonic. Hence, we exploit the small parameter $\ep$ (that
is, the thickness of the tubes) to extract detailed analytic
information about harmonic fields through energy estimates, and then
utilize this knowledge to compute the dynamical properties of the
local Beltrami field that are required in the KAM-type theorem under
the assumption that the Beltrami parameter $\la$ is suitably small (in
fact, of order $\ep^3$). A key ingredient in the computation of these
dynamical properties is a set of energy estimates for the local
Beltrami fields and all its derivatives that are optimal with respect
to the geometry of the tubes, i.e., the parameter $\ep$.

The local Beltrami field can now be approximated in any $C^k$ norm by
a global Beltrami field $u$ that falls off at infinity as $C/|x|$
using a Runge-type theorem. When using the KAM argument to guarantee
that $u$ still has a set of invariant tori diffeomorphic to
$\{\pd\cT_\ep(\ga_i)\}$, one has to face the problem that the local
Beltrami field is in fact a small perturbation of a field that does
{\em not}\/ satisfy the nondegeneracy condition (equivalent to a
rotation of the disk with constant frequency), which requires a fine
analysis of the terms controlled by each small parameter:
the thickness~$\ep$, the Beltrami constant~$\la$ and the error in the
Runge-type approximation.

We shall next present a sketch of the proof of Theorem~\ref{T.main},
where we explain the different intermediate results that are needed
in the demonstration and their interrelations. This short sketch is
also intended to serve as a guide to the paper. For the sake of
clarity, we will divide the proof in three stages:

\vspace{2.5ex}

\noindent{\bf Stage 1: Construction of the local Beltrami field.}
The local Beltrami field $v$ is obtained as the unique solution to
certain boundary value problem for the Beltrami equation. Our goal is
to estimate various analytic properties of this field, and for this
it is natural to introduce coordinates adapted to a Frenet frame in
each tube $\cT_\ep\equiv \cT_\ep(\ga_i)$, which essentially
correspond to an arc-length parametrization of the curve $\ga_i$ and
to rectangular coordinates in a transverse section of the tube. Thus we
consider an angular coordinate $\al$, taking values in
$\es:=\RR/\ell\ZZ$ (with $\ell$ the length of the curve $\ga_i$), and
rectangular coordinates $y=(y_1,y_2)$ taking values in the unit
$2$-disk $\DD$. Details are given in {\em Section~\ref{S.tubes}}\/.

To extract information about the local Beltrami field, a useful tool
is the boundary value problem for the Laplacian on scalar functions
with zero mean and zero Neumann
boundary conditions in the thin tube $\cT_\ep$:
\[
\De\psi=\rho\quad \text{in }\cT_\ep\,,\qquad
\pd_\nu\psi|_{\pd\cT_\ep}=0\,, \qquad \int_{\cT_\ep}\psi\, dx=0\,.
\]
When written in the natural coordinates $(\al,y)$, we obtain a
boundary value problem in the domain $\es\times\DD$, the
coefficients of the Laplacian in these coordinates depending on the
geometry of the tube strongly through its thickness $\ep$ and the
curvature and torsion of the core curve $\ga_i$.

It is clear that we have the $H^k$ estimate
\[
\|\psi\|_{H^k(\cT_\ep)}\leq C_{\ep,k}\|\rho\|_{H^{k-2}(\cT_\ep)}\,.
\]
However this estimate, where the constant $C_{\ep,k}$ depends on $\ep$
in an undetermined way, is far from being enough to compute the
dynamical quantities for the local Beltrami field that are needed in
the KAM-type theorem. Therefore, in {\em Section~\ref{S.Laplacian}}\/ we prove
energy estimates for the function $\psi$ and its derivatives that are
optimal with respect to the parameter~$\ep$. In particular, in order
to be able to compute the desired dynamical quantities for the local
Beltrami field later on, it is crucial to distinguish between
estimates for derivatives of $\psi$ with respect to the ``slow''
variable $\al$ and the ``fast'' variable $y$. The estimates for
the function $\psi$ and its derivatives that we will need are stated
in Theorem~\ref{T.Ck}.

These estimates are immediately put to work in {\em
  Section~\ref{S.harmonic}}\/ to derive a
perturbative expression for the harmonic field in each thin tube
$\cT_\ep$ and small~$\ep$. (Of course, the problem degenerates for
$\ep=0$ and the asymptotic results we prove do not correspond to a
Taylor expansion.) We will use the notation~$h$ for the harmonic field in $\cT_\ep$, which
is unique up to a multiplicative constant. In fact, we need to compute $h$ up to corrections that are
suitably small for small $\ep$; as before, different powers of $\ep$ are
required for the slow and fast components of the field. To first
order, the harmonic field can be written in the coordinates $(\al,y)$
as
\begin{equation}\label{perturbh}
h=\pd_\al+\tau(\al)\,(y_1\,\pd_2-y_2\,\pd_1)+\cO(\ep)\,,
\end{equation}
where $\tau$ stands for the torsion of each curve $\ga_i$ and
$\cO(\ep)$ denotes some vector field whose components in these
coordinates are bounded by a multiple of $\ep$ in the
$C^k(\es\times\DD)$ norm. By Eq.~\eqref{perturbh}, the field $h$ is an $\cO(\ep)$~perturbation of the rotation of constant frequency given by the total
torsion $\int_0^\ell\tau(\al)\, d\al$. From the point of view of KAM
theory, this is a very degenerate case, so it is not hard to guess that
one needs to compute $h$ (at least) up to order $\cO(\ep^3)$, as we do in
Theorem~\ref{T.psi0}. As a matter of fact, we will see later on that
the nondegeneracy condition that appears in the KAM theorem is that a
certain quantity, called the normal torsion, must be nonzero, and that
it actually vanishes modulo terms of order $\cO(\ep^2)$, which justifies
why we need good estimates for $h$.

In {\em Section~\ref{S.Beltrami}}\/ we show that, for any nonzero
constant $\la$ that is small enough (say, smaller than some fixed,
$\ep$-independent constant), there is a unique vector field that is
tangent to the boundary of each tube $\cT_\ep$, satisfies the equation
\begin{equation}\label{curlvrollo}
\curl v=\la v
\end{equation}
in the tube and whose harmonic part (i.e., the $L^2$ projection on the
space of harmonic fields) is $h$. (This is, of course, different to
what happens in the case of compact manifolds without boundary, where
Beltrami fields are always orthogonal to harmonic fields.)
We also prove estimates that measure how the field $v$ becomes close
to $h$ in the $C^k$ norm for small $\ep$ and $\la$. The key result is
Theorem~\ref{T.estimBeltrami}, which will be crucial in verifying the
conditions in the KAM argument  for the preservation of invariant tori.

\vspace{2.5ex}

\noindent{\bf Stage 2: Preservation of the invariant tori.} 
In Stage~1, for small enough~$\ep$ and $\la$, we have constructed a local Beltrami field $v$, which
satisfies Eq.~\eqref{curlvrollo} in a neighborhood of
$\bigcup_{i=1}^N\overline{\cT_\ep(\ga_i)}$ and has a set of invariant
tori given by $\pd\cT_\ep(\ga_i)$. Furthermore, the estimates we have
proved provide a very convenient expression for the local Beltrami
field, up to terms that are of order $\ep^3$ (when $\la=\ep^3$) in a
suitable sense.

In {\em Section~\ref{S.KAM}} we analyze the robustness of the
invariant tori using the Poincar\'e map~$\Pi$ defined by the local Beltrami field at
a transverse section of each tube. In each tube, $\Pi$ is
then a diffeomorphism of the disk that preserves certain measure,
while the intersection of each invariant torus with the transverse
section is an invariant circle of the Poincar\'e map. The persistence
of the invariant tori will rely on a KAM theorem for the Poincar\'e
map (Theorem~\ref{T.Llave}) that applies to individual invariant
circles.

To apply the KAM theorem, one has to verify that the rotation number
of $\Pi$ on the invariant circle is Diophantine and that $\Pi$ satisfies
a nondegeneracy condition (namely, that its normal torsion is
nonzero). Computing these dynamical quantities using the estimates for the local
Beltrami field is nontrivial. On the one hand, the rotation number
depends on the behavior of the trajectories for arbitrarily large
times, so a delicate treatment is required in order to control uniformly the
effect of the $\cO(\ep^2)$ terms in the vector field. On the
other hand, the normal torsion, which is required to be nonzero, turns
out to be zero modulo $\cO(\ep^2)$, since the Poincar\'e map is a
small perturbation of a constant frequency rotation of the disk, which is a highly
degenerate case from the point of view of KAM theory.

The estimates proved at Stage~1 are tailored to permit us to overcome
these difficulties, yielding expressions for the rotation number
$\om_\Pi$ and the normal torsion~$\cN_\Pi$ that depend on the geometry
of each curve $\ga_i$ through its curvature~$\ka$ and torsion~$\tau$
(see Theorems~\ref{T.rotation} and~\ref{T.Herman}):
\begin{align*}
\om_\Pi&=\int_0^\ell \tau(\al)\,d\al+\cO(\ep^2)\,,\\
\cN_\Pi&=-\frac{5\pi\ep^2}8\int_0^\ell\ka(\al)^2\,\tau(\al)\, d\al
+\cO(\ep^3)\,.
\end{align*}

These expressions allow us to prove the main result in this Stage~2,
which is that for ``generic'' curves $\ga_i$ the hypotheses of the
aforementioned KAM theorem are satisfied, so that the invariant tori
of the local Beltrami field $v$ are robust: if $u$ is a
divergence-free vector field in a neighborhood of the tubes that is
close enough to~$v$ in a suitable sense (e.g., in a high enough $C^k$
norm), then $u$ also has an invariant tube diffeomorphic to each
$\cT_\ep(\ga_i)$, and moreover the corresponding diffeomorphisms can
be taken close to the identity. This is proved in Theorem~\ref{T.KAM}.

\vspace{2.5ex}

\noindent{\bf Stage 3: Approximation by a global Beltrami field.}
In {\em Section~\ref{S.approx}}\/ we prove a Runge-type approximation
theorem for Beltrami fields that decay at infinity (Theorem~\ref{T.approx}). More precisely, we
will show that the local Beltrami field $v$, considered in the
previous stages and defined in a neighborhood of the thin tubes $\cT_\ep(\ga_i)$, can be
approximated in any $C^k$ norm by a Beltrami field $u$ that falls off
at infinity as $|D^ju(x)|<C_j/|x|$. This decay is optimal for the
class of Beltrami fields.

The proof of this theorem consists of two steps. In the first step we
use functional-analytic methods to approximate the field field $v$ by
an auxiliary vector field $w$ that satisfies the elliptic equation
$\De w=-\la^2 w$ in a large ball of $\RR^3$ that contains all the
tubes. In the second step, we define the approximating global Beltrami
field $u$ in terms of a truncation of a suitable series representation
of the field $w$, ensuring that $u$ has the desired fall-off at
infinity.

\vspace{2.5ex}

\noindent{\bf Completion of the proof of the main theorem.} 
In {\em Section~\ref{S.main}}\/ we use all previous results to
complete the proof of Theorem~\ref{T.main}. Indeed, from the
robustness of the invariant tori of the local Beltrami field, it
immediately follows that the global Beltrami field $u$ has a set of
thin vortex tubes equivalent through a $C^m$-small diffeomorphism
$\Phi$ to $\{\cT_\ep(\ga_i)\}$, provided $\ep$ is small enough. More
precisely, what is proved is that there is some constant $\ep_0$,
which depends only (albeit in a rather nontrivial way) on the geometry
of the curves $\ga_i$ and on the allowed smallness for the
diffeomorphism~$\Phi$, so that the statement of the theorem holds for
any thickness~$\ep\leq\ep_0$. 

The remaining properties of the vortex lines stated in
Theorem~\ref{T.main} are established in this section too, using
results derived throughout the paper. In particular, we show that near
the curve $\ga_i$ there is an elliptic periodic trajectory of the
field $u$. As a side remark, notice that this elliptic periodic
trajectory is obviously linearly stable, but is {\em not}\/ granted a priori
to be Lyapunov stable. Therefore, as shown by the counterexample of
Anosov and Katok (cf.\ e.g.~\cite{FK09}), it is not guaranteed that
there are invariant tori of the field in a neighborhood of the
elliptic trajectory. A careful (and nontrivial) analysis of the
dynamics near the elliptic trajectory for small $\ep$ and $\la$ would
be required to prove the existence of these tori.


\vspace{2.5ex}

Before beginning with the technical part of the paper, it is
convenient to provide a short comparison between the proof of
Theorem~\ref{T.main} and that of the main theorem in
Ref.~\cite{Annals}, where we showed that there are steady solutions to
the Euler equation with a set of vortex lines diffeomorphic to any
given link. In this reference, the proof was also based on the
construction of a local Beltrami field with a ``robust'' set of
periodic trajectories, which was then approximated by a global
Beltrami field. However, the implementation of this basic principle is
totally different.

To begin with, the robustness of periodic trajectories
in~\cite{Annals} relies on the hyperbolic permanence theorem (which is
essentially an application of the implicit function theorem) instead
of a considerably more sophisticated KAM argument. To construct a
local Beltrami field with prescribed periodic trajectories $\ga_i$ that are
hyperbolic, as required by the permanence theorem, it is enough to
prove a suitable analog of the Cauchy--Kowalewski theorem for the
curl operator: indeed, since the field is divergence-free, it is
enough to assume that the field used as Cauchy datum is exponentially
contracting into the curve $\ga_i$ on the Cauchy surface to obtain a local Beltrami field
that has $\ga_i$ as a hyperbolic periodic trajectory. All this is in
strong contrast with the case of vortex tubes, where the construction
of the local solution with a robust set of prescribed invariant tori
requires the analysis of the boundary value problem and the KAM
argument described in Stages~1 and~2 (Sections~\ref{S.Laplacian}
to~\ref{S.KAM}). 

The approximation theorem we use in this paper is also different from
the one employed in~\cite{Annals}. The reason for this is that
in~\cite{Annals} we had no control on the growth of the global
Beltrami field at infinity, even in the case of connected links. On
the contrary, the approximation theorem we prove in this paper
(Theorem~\ref{T.approx}) yields Beltrami fields with optimal fall-off
at infinity. The proof of these approximation theorems is considerably
different: the old theorem is based on an iterative scheme that
uses a theorem by Lax and Malgrange and works for the $\curl$ operator
in any Riemannian $3$-manifold, while the new one is based on
different principles and takes advantage of the geometry of Euclidean
$3$-space to ensure that the global Beltrami field falls off at
infinity. 

Notice that the main theorem in~\cite{Annals} applies to
locally finite sets of curves, while in Theorem~\ref{T.main} in this
paper we can only take finite sets of curves $\ga_1,\dots, \ga_N$. Some
comments in this regard are made in Remark~\ref{R.locallyfinite}.

\section{Geometry of thin tubes}
\label{S.tubes}

In this section we introduce some notation, including a coordinate
system, that will be used throughout the paper to describe functions
and vector fields defined on thin tubes. These tubes will be
characterized in terms of the curve that sits on its core and its
thickness $\ep$, which is a parameter that will be everywhere assumed
to be suitably small.

Let us start with a closed analytic curve with an arc-length
parametrization $\ga:\es\to \RR^3$, with $\es:=\RR/\ell\ms
\ZZ$ (throughout the paper, when the period is $2\pi$ we will simply
write $\SS^1\equiv \SS^1_{2\pi}$). This amounts to saying that the tangent field~$\dot\ga$ has unit norm and $\ell$
is the length of the curve. We will abuse the
notation and denote also by $\ga$ the curve parametrized by the above
map (i.e., the image set $\ga(\es)$).

Let us denote
by $\cT_\ep\equiv \cT_\ep(\ga)$ a metric neighborhood with
thickness~$\ep$ of the curve~$\ga$, that is,
\[
\cT_\ep:=\big\{x\in\RR^3:\dist(x,\ga)<\ep\big\}\,.
\]
This is a thin tube having the curve $\ga$ as its core.  It is
standard that, for small $\ep$, the boundary $\pd \cT_\ep$ is
analytic. The normal bundle of the curve $\ga$ being
trivial~\cite{Ma59}, one can associate to each $\al\in \es$ two
orthogonal unit vectors $e_j(\al)$ in $\RR^3$ perpendicular to the
curve at the point $\ga(\al)$. For convenience, we will make the
assumption that the curvature of the curve $\ga$ does not vanish, which allows us to take $e_1(\al),e_2(\al)$ as the normal
and binormal vectors at the point $\ga(\al)$. It is well known that
the assumption that the curvature does not vanish is satisfied for
generic curves~\cite[p.~184]{Bruce} (roughly speaking, ``generic''
refers to an open and dense set, with respect to a reasonable $C^k$
topology, in the space of smooth curves in $\RR^3$).



Using the vector fields $e_j(\al)$ and denoting by $\DD$ the two-dimensional
unit disk, we can introduce analytic coordinates
$(\al,y)\in\es\times\DD$ in the tube $\cT_\ep$ via the diffeomorphism
\[
(\al,y)\mapsto \ga(\al)+\ep\ms y_1\ms e_1(\al)+\ep \ms y_2\ms e_2(\al)\,.
\]
In the coordinates $(\al,y)$, a short computation using the Frenet
formulas shows that the Euclidean metric in the tube reads as
\begin{equation}\label{ds}
ds^2=A\,d\al^2+2\ep^2\tau(y_2\,dy_1-y_1\,dy_2)\,d\al+\ep^2\big(dy_1^2+dy_2^2\big)\,,
\end{equation}
where $\ka\equiv\ka(\al)$ and $\tau\equiv\tau(\al)$ respectively
denote the curvature and torsion of the curve,
\begin{equation}\label{A1}
A:=(1-\ep\ka y_1)^2+(\ep\tau)^2|y|^2\,,
\end{equation}
and $|y|$ stands for the Euclidean norm of $y=(y_1,y_2)$. As is customary, we will denote by $g_{ij}$
and $g^{ij}$ the components of the metric tensor and its inverse,
respectively.

We will sometimes take polar coordinates $r\in(0,1)$, $\te\in\SS^1:=\RR/2\pi\ms\ZZ$  in the disk~$\DD$, which are
defined so that
\[
y_1=r\cos\te\,,\quad y_2=r\sin\te\,.
\]
The metric then reads
\begin{equation}\label{dspolar}
ds^2=A\,d\al^2-2\ep^2\tau
r^2d\te\,d\al+\ep^2dr^2+\ep^2r^2d\te^2\,,
\end{equation}
where we, with a slight abuse of notation, still call $A$ the
expression of~\eqref{A1} in these coordinates, i.e.,
\begin{equation}\label{A2}
A:=(1-\ep\ka r\cos\te)^2+(\ep\tau r)^2\,.
\end{equation}
Notice that the coordinate $r$ is simply the distance to the curve
$\ga$. For future reference, we record that the volume measure is written in these coordinates as
\begin{equation}\label{dV}
dV:=B\,d\al\,dy=B r\,d\al\,dr\,d\te
\end{equation}
up to a factor of $\ep^2$, where $dy:=dy_1\,dy_2$ and
\begin{equation}\label{B}
B:=1-\ep\ka y_1=1-\ep\ka r\cos\te\,.
\end{equation}

\section{Estimates for the Neumann Laplacian in thin tubes}
\label{S.Laplacian}

In this section we will derive some estimates for the Laplace equation $\De \psi=\rho$
in the thin tube $\cT_\ep$ with zero Neumann boundary
conditions. As we will see, to gain control on the function $\psi$ it is convenient to attack this
equation in the coordinates $(\al,y)\in\es\times\DD$ that we introduced in
Section~\ref{S.tubes}, in terms of which the Laplace equation can be
written as
\begin{subequations}\label{Depsi}
\begin{equation}
\De\psi=\rho \quad\text{in } \es\times\DD\,,\qquad \pd_\nu\psi=0 \quad\text{on } \es\times\pd\DD\,,
\end{equation}
where the Laplacian $\De$ is now interpreted as an $\ep$-dependent
differential operator in the variables $(\al,y)$ or $(\al,r,\te)$. In
order to ensure the existence of solutions to this equation, we
suppose that $\int\rho\,dV=0$, which allows us to uniquely determine
the solution $\psi$ by demanding that it also has zero mean:
\begin{equation}
\int \psi\,d\al\,dy=0\,.
\end{equation}
\end{subequations}
(Here and in what follows, we omit the domain of integration when it
is the whole domain $\es\times\DD$. We could have used the measure
$dV$ in the above formula too, but this choice is slightly more convenient.) For future reference, we record
here the expression of $\De$ in the variables $(\al,r,\te)$:
\begin{align}
\De\psi&=\frac1{\ep^2}\Big(\psi_{rr}+\frac1r\psi_r+\frac
A{r^2B^2}\psi_{\te\te}\Big)+\frac1{B^2}\psi_{\al\al} +
\frac{2\tau}{B^2}\psi_{\al\te} +\frac{\tau'-\ep
  r(\ka\tau'-\ka'\tau)\cos\te}{B^3}\psi_\te \notag
\\&\qquad\qquad+\frac1\ep\Big(
\frac{\ka\sin\te(B^2-(\ep\tau r)^2)}{r B^3}\psi_\te-\frac{\ka \cos\te}{B}\psi_r \Big)
+\frac{\ep r(\ka'\cos\te-\tau\ka\sin\te)}{B^3}\psi_\al \,.\label{De}
\end{align}
As usual, we denote partial derivatives by subscripts when
there is not risk of confusion.

Given a subset $\Om\subset \es\times\DD$, we will use the notation 
\[
\|\psi\|_\Om:=\bigg(\int_\Om \psi^2\,d\al\,dy\bigg)^{1/2}
\]
for the $L^2(\Om)$-norm of the function $\psi$, omitting the subscript
when $\Om$ is the whole domain $\es\times\DD$. In this section we will
use the notation
\begin{equation}\label{dH1ep}
\|\psi\|_{\dot H^1_\ep}^2:=\|\pd_\al\psi\|^2+ \frac{\|D_y\psi\|^2}{\ep^2}
\end{equation}
for a homogeneous Sobolev norm in which the derivatives associated
with the ``small directions'' of the thin tube are weighted with an
appropriate $\ep$-dependent factor. The usual $H^k$ norm of the function $\psi(\al,y)$ will be denoted by $\|\psi\|_{H^k}$.

In what follows, we will assume that $\psi$ is a solution of the Laplace
equation~\eqref{Depsi} with zero mean and zero Neumann boundary
conditions, which ensures that
\begin{equation}\label{weak}
\int g^{ij}\,\pd_i\psi\,\pd_j\vp\, dV=-\int \rho\vp\, dV
\end{equation}
for any $\vp\in H^1(\es\times\DD)$. As usual, the scripts $i,j$ range over the set of coordinates
$\{\al,y_1,y_2\}$, and summation over repeated indices is understood. It is important to notice that, for any function
$\vp$,
\begin{equation}\label{H1epgij}
\int g^{ij}\,\pd_i\vp\,\pd_j\vp\,
dV=\big(1+\cO(\ep)\big)\,\|\vp\|_{\dot H^1_\ep}^2\,,
\end{equation}
so that the $\dot H^1_\ep$~norm is essentially a more convenient way of dealing with
the natural $\dot H^1$~norm associated with the metric. We will use this identity many times in this section without further comment.

The structure of this section is the following. We will start by
estimating the $L^2$ norm of the derivatives of the function $\psi$
with respect to the ``slow'' variable $\al$
(Subsection~\ref{SS.basic}). The proof of these estimates is
standard. We recall that the reason why $\al$ is called the slow
variable is that it parametrizes the ``large'' direction of the thin
tube $\cT_\ep$, as opposed to the ``fast'' variable $y$, which is
obtained by rescaling the small section of the tube. Estimates for the
derivatives of $\psi$ with respect to the ``fast'' variable $y$ with
optimal dependence in the small parameter $\ep$ are presented in
Subsection~\ref{SS.improved}. They are more complicated to obtain,
basically because one has to consider an auxiliary function in order
to get rid of the contributions to $\psi$ that only depend on the slow
variable. The resulting estimates for $\psi$, which we will often use
in forthcoming sections, will be stated in Subsection~\ref{SS.Ck}.

\subsection{Estimates for derivatives with respect to the ``slow'' variable}
\label{SS.basic}

In this subsection we will prove $H^k$ estimates for the derivatives of
the function $\psi$ with respect to the ``slow'' variable $\al$ (cf.\ Proposition~\ref{L.psial}).
We will begin with the following proposition, where we estimate the $L^2$
and $\dot H^1_\ep$ norms of the function $\psi$. As is customary, throughout this article we will use the 
letter $C$ to denote $\ep$-independent constants that may vary
from line to line.

\begin{proposition}\label{L.H1}
  The function $\psi$ satisfies the estimate
\[
\|\psi\|+\|\psi\|_{\dot H^1_\ep}\leq C\|\rho\|\,.
\]
\end{proposition}
\begin{proof}
By the expression of the metric in the coordinates $(\al,y)$, it is
clear that for small
enough $\ep$ one has
\[
\int g^{ij}\pd_i\psi\,\pd_j\psi\, dV\geq \big(1+\cO(\ep)\big)\int
\Big(\psi_\al^2+\frac{|D_y\psi|^2}{\ep^2}\Big)\,d\al \,dy\geq \frac12\int
\big(\psi_\al^2+|D_y\psi|^2\big)\,d\al \,dy\,.
\]
The rightmost term in the
inequality is bounded from below by $\frac12\mu_1\|\psi\|^2$, where $\mu_1$ stands for
the first nonzero Neumann eigenvalue of the flat solid torus:
\[
\mu_1:=\inf\bigg\{\int
\big(\vp_\al^2+|D_y\vp|^2\big)\,d\al\,dy: \vp\in
C^\infty(\es\times\DD),\;\int\vp\,d\al\,dy=0,\; \|\vp\|=1\bigg\}>0\,.
\]
Hence we infer that
\begin{equation}\label{Neigenv}
\|\psi\|\leq C \|\psi\|_{\dot H^1_\ep}
\end{equation}
for some $C>0$.

To conclude, we can now use the weak formulation of
the equation~\eqref{weak} with $\vp=\psi$ and the Cauchy--Schwartz inequality to derive that
\[
\|\psi\|_{\dot H^1_\ep}^2\leq \big(1+\cO(\ep)\big)\int g^{ij} \pd_i\psi\,\pd_j\psi\, dV\leq
\big(1+\cO(\ep)\big)\|\rho\|\|\psi\|\leq C\|\rho\|\|\psi\|_{\dot H^1_\ep}\,.
\]
The proposition then follows from this inequality
and the estimate~\eqref{Neigenv}.
\end{proof}

In the following  lemma we record an elementary inequality that will
be of use several times in this section.

\begin{lemma}\label{L.npi}
Let $D^l_{\al,y}$ denote the tensor of $l$\th\ order derivatives with respect to
the variables $(\al,y)$. For any well-behaved (e.g., smooth) functions
$\vp,\tilde\vp,\chi,\tilde\chi$ on $\es\times\DD$, 
one has
\[
\int \big| D_{\al,y}^l(B\ms g^{ij})\,\pd_i\vp\,\pd_j\tilde\vp\,\chi\tilde\chi \big|\,d\al\,dy 
\leq \ep C_l\bigg(\!\int g^{ij}\,\pd_i\vp\,\pd_j\vp\ms\chi\,dV \!\bigg)^{\frac12}\bigg(\!\int  g^{ij}\pd_i\tilde\vp\,\pd_j\tilde\vp\,\tilde\chi\,dV \!\bigg)^{\frac12}.
\]
In particular,
\[
\int \big| D_{\al,y}^l(B\ms g^{ij})\,\pd_i\vp\,\pd_j\tilde\vp
\big|\,d\al\,dy\leq C_l\ep\|\vp\|_{\dot H^1_\ep}\|\tilde\vp\|_{\dot H^1_\ep}\,.
\]
\end{lemma}
\begin{proof}
It is an immediate consequence of the expression for the metric and the
function $B$ in the coordinates $(\al,y)$ (see Eq.~\eqref{ds}) and the
Cauchy--Schwartz inequality.
\end{proof}

We are now ready to prove the estimates for the derivatives of $\psi$ with respect to the slow
variable $\al$ that we will need in this paper:

\begin{proposition}\label{L.psial}
The $k$\th\ partial derivative of the function $\psi$ with respect to
the angle $\al$ satisfies
\[
\|\pd_\al^{k+1}\psi\|_{\dot H^1_\ep}\leq C_k\|\rho\|_{H^k}\,, 
\]
where $k$ is any nonnegative integer. 
\end{proposition}
\begin{proof}
Let us take $\vp=\psi_{\al\al}$ in Eq.~\eqref{weak} and integrate by
parts to get
\[
\int g^{ij}\,\pd_i\psi_\al\,\pd_j\psi_\al\, dV=  \int
\rho\psi_{\al\al}\,d V- \int \pd_\al(B g^{ij}) \,\pd_i
\psi\,\pd_j\psi_\al\, d\al\, dy\,.
\]
The LHS is bounded from below by $(1+\cO(\ep))\ms \|\psi_\al\|_{\dot
  H^1_\ep}^2$, while from the definition of the norm $\dot H^1_\ep$ and
Lemma~\ref{L.npi} it stems that
\begin{align*}
  \bigg|\int \rho\psi_{\al\al}\,d V\bigg|&\leq
  (1+\cO(\ep))\,\|\rho\|\,\|\psi_\al\|_{\dot H^1_\ep}\,,\\
\bigg|\int \pd_\al(B g^{ij}) \,\pd_i
\psi\,\pd_j\psi_\al\, d\al \, dy\bigg|&\leq C\ep\|\psi\|_{\dot H^1_\ep}
\|\psi_\al\|_{\dot H^1_\ep}\,.
\end{align*}
Using Proposition~\ref{L.H1} to estimate $\|\psi\|_{\dot H^1_\ep}$ in terms
of $\|\rho\|$, we infer that
\[
\|\psi_\al\|_{\dot H^1_\ep}^2\leq C \|\rho\|\|\psi_\al\|_{\dot
  H^1_\ep}\,,
\]
which readily implies the desired bound for $k=0$. When $k$ is a
positive integer, the proof is totally analogous and can be obtained
by induction on $k$ using $\vp=\pd_\al^{2k+2}\psi$, the only difference being that one needs to estimate the
term $\int \rho\vp$ as
\[
\bigg|\int \rho\, \pd_\al^{2k+2}\psi\, dV\bigg|\leq
C\|\rho\|_{H^k}\sum_{j=1}^{k+1}\|\pd_\al^j\psi\|_{\dot H^1_\ep}\leq
C\|\rho\|_{H^k}^2+C\|\rho\|_{H^k}\|\pd_\al^{k+1}\psi\|_{\dot H^1_\ep}
\]
by the induction hypothesis.
\end{proof}


\subsection{Estimates for the ``fast'' variables}
\label{SS.improved}

To estimate the derivates with respect to the ``fast'' variable $y$ in
an optimal way, it is crucial to ensure that the terms that only depend on
$\al$ are not considered when estimating the norms of the function. A
convenient way of doing this is by considering the auxiliary function
\begin{equation}\label{Bpsi}
\Bpsi(\al,y):=\psi(\al,y)-\frac1\pi\int_{\DD}\psi(\al,y')\,dy'\,,
\end{equation}
which is obtained from $\psi$ by subtracting its average in the fast
variable. It should be emphasized that, if we were not to subtract
this average, the estimates we would obtain would not be strong
enough for our needs in later sections. The key estimate in this
subsection is Theorem~\ref{T.Hk}.

An immediate observation is that, of course
\begin{equation}\label{psiBpsi}
D_y^j\pd_\al^k\psi=D_y^j\pd_\al^k\Bpsi
\end{equation}
whenever the number $j$ of derivatives we take with respect to $y$ is
greater than zero, and that $\Bpsi$ has zero mean: 
\[
\int\Bpsi\, d\al\, dy=0\,.
\]
Moreover, the function $\Bpsi$ satisfies the equation
\[
\De\Bpsi=\Brho\quad\text{in } \es\times\DD\,,\qquad \pd_\nu\Bpsi=0
\quad\text{on } \es\times\pd\DD\,,
\]
with
\[
\Brho(\al,y):=\rho(\al,y)-\frac1\pi\,\De\int_{\DD}\psi(\al,y')\, dy'\,.
\]

The estimates for $\psi$ we derive in the previous section guarantee
that the norm of $\Brho$ can be bounded by a multiple of the norm of
the initial source term $\rho$:

\begin{proposition}\label{L.Brho}
The $H^k$ norm of $\Brho$ is bounded by
\[
\|\Brho\|_{H^k}\leq C_k \|\rho\|_{H^k}\,.
\]
\end{proposition}
\begin{proof}
Observe that the action of the
Laplacian (which we compute in the coordinates $(\al,y)$) on the function
$\psi(\al,y')$ is
\[
\De\psi(\al,y')=\big(1+\cO(\ep)\big)\,\pd_\al^2\psi(\al,y')+\cO(\ep)\,\pd_\al\psi(\al,y')\,,
\]
where $\cO(\ep^n)$ here stands for an $\ep$-dependent quantity $Q(\al,y)$ bounded as
\[
\|Q\|_{H^k}\leq C_k\ep^n
\]
for all $k$. Therefore, one finds that
\begin{align*}
\|D_y^j\pd_\al^k\Brho\|&=\bigg\|D_y^j\pd_\al^k\rho-\frac1\pi\,D_y^j\pd_\al^k\De\int_{\DD}\psi(\al,y')\,
dy'\bigg\|\\
&=\bigg\|D_y^j\pd_\al^k\rho-\frac1\pi\int_{\DD}\pd_\al^{k+2}\psi(\al,y')\,
dy' + \sum_{l=1}^{k+2}\int_{\DD} \cO(\ep)\,\pd_\al^l\psi(\al,y')\,
dy'\bigg\|\\
&\leq
\|D_y^j\pd_\al^k\rho\|+(1+C\ep)\|\pd_\al^{k+2}\psi\|+C\ep\sum_{l=1}^{k+1}\|\pd_\al^l\psi\|\\
&\leq C\|\rho\|_{H^k}\,,
\end{align*}
where in the last step we have used Propositions~\ref{L.H1} and~\ref{L.psial}. The
claim then follows.
\end{proof}

To derive energy estimates for $\psi$, we will use that one obviously has
\begin{equation}\label{weakBpsi}
\int g^{ij}\,\pd_i\Bpsi\,\pd_j\vp\, dV=-\int \Brho\ms \vp\, dV
\end{equation}
for all $\vp\in H^1(\es\times\DD)$. Our first result will be an
estimate for the $L^2$ norm of $\Bpsi$ and $D_y\Bpsi$. While we can
readily derive a bound for these quantities using Proposition~\ref{L.H1},
the estimates we prove here are much sharper for small $\ep$. This
will be crucial in the derivation of optimal estimates for $\psi$.

\begin{proposition}\label{L.H1psi}
The function $\Bpsi$ satisfies the $H^1$ estimates
\[
\|\Bpsi\|+\|D_y\Bpsi\|\leq C\ep^2\|\rho\|\,,\qquad\|\pd_\al\Bpsi\|\leq C\ep\|\rho\|\,.
\]
\end{proposition}
\begin{proof}
Choosing $\vp=\Bpsi$ in Eq.~\eqref{weakBpsi} and in view of the
expression of the coefficients of the metric~\eqref{ds}, one immediately obtains
that
\begin{equation}\label{eq2}
\|\Bpsi\|_{\dot H^1_\ep}^2\leq \big(1+\cO(\ep)\big)\|\Brho\|\|\Bpsi\|\,.
\end{equation}
Since $\Bpsi(\al,\cdot)$ has zero mean in the disk $\DD$ for any fixed
$\al$, by Poincar\'e's inequality there is a positive constant $C$,
independent of $\ep$ and $\al$ (namely, the first nonzero Neumann
eigenvalue of the disk), such that
\[
\int_{\DD} \big|D_y\Bpsi(\al,y)\big|^2\,dy \geq C\int_{\DD}\Bpsi(\al,y)^2\, dy\,.
\]
Integrating this inequality in $\al$, the $H^1$ norm of $\Bpsi$ can be estimated as
\[
\|\Bpsi\|_{\dot H^1_\ep}^2\geq \frac1{\ep^2}\int
\big|D_y\Bpsi(\al,y)\big|^2\,dy\, d\al\geq \frac C{\ep^2}\|\Bpsi\|^2\,.
\]
Together with Eq.~\eqref{eq2}, this yields $\|\Bpsi\|\leq
C\ep^2\|\Brho\|$ and $\|\Bpsi\|_{\dot H^1_\ep}\leq C\ep\|\Brho\|$, so
the claim follows directly from Proposition~\ref{L.Brho}.
\end{proof}

It is particularly easy to derive preliminary estimates (which will be
instrumental in the proof of Proposition~\ref{L.Bpsiy}) for the derivatives of $\Bpsi$ with respect to $\al$ thanks to Proposition~\ref{L.psial}. Notice these 
bounds will be substantially improved later on.

\begin{proposition}\label{L.Bpsial}
The derivative of $\Bpsi$ with respect to $\al$ satisfies
\[
\|D_y\pd_\al^{k+1}\Bpsi\|\leq C\ep\|\rho\|_{H^k}\,,\qquad \|\pd_\al^{k+2}\Bpsi\|\leq
C_k\|\rho\|_{H^k}
\]
for any nonnegative integer $k$.
\end{proposition}
\begin{proof}
The claim is an immediate consequence of the definition of $\Bpsi$ (cf.\ Eq.~\eqref{Bpsi}) and
Proposition~\ref{L.psial}, since obviously the $L^2(\es\times\DD)$ norm of
the function
\[
\pd_\al^j\int_{\DD}\psi(\al,y')\,dy'
\]
is bounded from above by $C\|\pd_\al^j\psi\|$, which was in turn
estimated in the aforementioned proposition.
\end{proof}

We are ready to show that the second derivatives of $\Bpsi$ with
respect to the fast variable $y$ are bounded by a factor of order
$\ep^2$. The proof of these estimates makes essential use of Propositions~\ref{L.H1psi}
and~\ref{L.Bpsial}.

\begin{proposition}\label{L.Bpsiy}
The second derivatives of the function $\Bpsi$ with respect to the fast
coordinates are bounded by
\[
\|D_y^2\Bpsi\|\leq C\ep^2\|\rho\|\,.
\]
\end{proposition}
\begin{proof}
  We will denote by $\DD_R$ the two-dimensional disk of radius $R$,
  $R$~being a fixed real smaller than $1$. Let us start by proving
  interior estimates. For this, we will denote by $\pd_a$ the
  derivative along a $y$-direction (that is, $\pd_1$ or $\pd_2$), and
  consider a smooth function $\chi(|y|)$ equal to $1$ for $|y|<R$ and
  equal to zero in a neighborhood of $\pd\DD$. Taking
  $\vp=\pd_a(\chi^2\pd_a\Bpsi)$ in Eq.~\eqref{weakBpsi} (here and in
  what follows we will {\em not}\/ sum over the index $a$) and
  integrating by parts, one readily obtains
\begin{equation}\label{eq3}
I^2:=\int g^{ij} \pd_{i}\pd_a\Bpsi\,\pd_{j}\pd_a\Bpsi\, \chi^2\, dV= I_1-I_2-I_3\,,
\end{equation}
where
\begin{gather*}
I_1:=\int\Brho\, \pd_a(\chi^2\pd_a\Bpsi)\, dV\,, \qquad I_2:=\int \pd_a(B
g^{ij})\,\pd_i\Bpsi \,\pd_j(\chi^2\pd_a\Bpsi)\, d\al\, dy\\
I_3:=\int \pd_a\Bpsi\,g^{ij}\pd_j(\chi^2)\,\pd_i\pd_a\Bpsi\,d V\,.
\end{gather*}
Let us estimate these integrals. The first one can be easily
controlled using the Cauchy--Schwartz inequality and Propositions~\ref{L.Brho} and~\ref{L.H1psi}:
\begin{align*}
|I_1|&\leq \bigg|\int\Brho\chi^2\pd_{a}^2\Bpsi\, dV\bigg|+\bigg|\int
\Brho\,\pd_a(\chi^2)\pd_a\Bpsi\, dV\bigg|\leq
C\|\Brho\|\big(\|\chi\, \pd_{a}^2\Bpsi\|+\|\pd_a\Bpsi\|)\\
&\leq C\ep\|\rho\|I+ C\ep^2\|\rho\|^2\,.
\end{align*}
The integral $I_2$ can be controlled using an analogous argument,
Lemma~\ref{L.npi} and Jensen's inequality. This leads to the estimate
\begin{align*}
|I_2|&\leq \bigg|\int \pd_a(Bg^{ij})\chi^2\pd_i\Bpsi\,\pd_j\pd_a\Bpsi
\, d\al\,dy\bigg| + \bigg|\int \pd_a(Bg^{ij})\pd_i\Bpsi\,\pd_j(\chi^2)\,\pd_a\Bpsi
\, d\al\,dy\bigg| \\
&\leq C\ep \|\Bpsi\|_{\dot H^1_\ep} I+ C\|\pd_a\Bpsi\|\bigg(\!\int
\Big( \pd_a(Bg^{ij})\,\pd_i\Bpsi\,\pd_j\chi\Big)^2\,
d\al\,dy\bigg)^{\frac12}\\
&\leq C\ep^2\|\rho\|I + C\ep^2\|\rho\|\int
\Big|\pd_a(Bg^{ij})\,\pd_i\Bpsi\,\pd_j\chi\Big|\, d\al \, dy\\
&\leq C\ep^2\|\rho\|I + C\ep^3\|\rho\|\|\Bpsi\|_{\dot
  H^1_\ep}\|\chi\|_{\dot H^1_\ep} \leq C\ep^2\|\rho\|I + C\ep^3\|\rho\|^2\,,
\end{align*}
where in the fourth line we have used that $\|\chi\|_{\dot
  H^1_\ep}=\|D_y\chi\|/\ep=C/\ep$. A similar argument shows that
\[
|I_3|\leq C\|\pd_a\Bpsi\|\|\chi\|_{\dot H^1_\ep}I\leq C\ep\|\rho\| I\,.
\]
Feeding these bounds into Eq.~\eqref{eq3}, we obtain 
\[
I^2\leq C\ep\|\rho\| I+ C\ep^2\|\rho\|^2\,,
\]
so that $I\leq C\ep\|\rho\|$. From the definition of the function
$\chi$ and the identity~\eqref{H1epgij} it then follows that
\begin{equation}\label{interior}
\|D_y^2\Bpsi\|_{\es\times\DD_{R}} \leq C\ep^2\|\rho\|\,,
\end{equation}
which is the desired interior estimate.

To prove the estimates up to the boundary, we begin by showing that
the $\dot H^1_\ep$~norm of 
\[
\pd_\te \Bpsi\equiv
y_1\,\pd_2\Bpsi-y_2\,\pd_1\Bpsi
\]
is bounded in terms of
$\|\rho\|$. For this, we will find it convenient to take a smooth
function $\chi(|y|)$, equal to $1$ for $|y|>R$ and vanishing in a
neighborhood of the origin, and use polar coordinates throughout
without further notice. If we now take $\vp=\chi^2\Bpsi_{\te\te}$ in
Eq.~\eqref{weak} and integrate by parts, we readily find that
\begin{align*}
  \int g^{ij}\,\pd_i\Bpsi_\te\,\pd_j\Bpsi_\te\,\chi^2\, dV&=\int
  \Brho\ms \Bpsi_{\te\te}\,\chi^2\, dV
  -\int \pd_\te(B g^{ij})\,\pd_i\Bpsi\,\pd_j (\chi^2\Bpsi_\te)\,
  d\al\,dy\\
&\qquad\qquad\qquad\qquad\qquad\qquad\qquad -\int g^{ij}\,\pd_i\Bpsi_\te\,\Bpsi_\te\,\pd_j(\chi^2)\,dV\,,
\end{align*}
where the indices $i,j$ now range over the set
$\{r,\te,\al\}$. Notice that the reason we are now using polar
coordinates is that $\pd_\te$ does not commute with the derivatives
with respect to $(y_1,y_2)$. Arguing as
above, one finds that
\[
\int g^{ij}\,\pd_i\Bpsi_\te\,\pd_j\Bpsi_\te\, \chi^2\, dV\leq C\ep\|\rho\|\,,
\]
which ensures that
\begin{equation}\label{bd1}
\|\Bpsi_{\te\te}\|_{\es\times\A_{R}}+\|\Bpsi_{r\te}\|_{\es\times\A_{R}}\leq C\ep^2\|\rho\|\,,
\end{equation}
where $\A_R:=\DD\minus\overline{\DD_R}$ is the annulus of inner radius
$R$.

To estimate the derivative $\Bpsi_{rr}$, it now suffices to isolate
this quantity in the equation $\De \Bpsi=\Brho$. From the expression of
the Laplacian in these coordinates~\eqref{De} it is apparent that for
$r>R$ one can write $\Bpsi_{rr}$ as
\[
\ep^{-2}\Bpsi_{rr}-\Brho = \cO(\ep^{-2})\,
\big(|\Bpsi_{\te\te}|+|\Bpsi_r|\big)+\cO(\ep^{-1})\, \Bpsi_\te+\cO(1)\,\big(|\Bpsi_{\al
  \te}|+|\Bpsi_{\al \al}|+|\Bpsi_\al|\big)\,.
\]
From Propositions~\ref{L.H1psi}--\ref{L.Bpsiy}
and the estimates~\eqref{bd1}, it
then follows that
\[
\|\Bpsi_{rr}\|_{\es\times\A_R}\leq C\ep^2\|\rho\|\,.
\]
This yields the desired boundary estimates, thus completing the proof
of the proposition.
\end{proof}

The results we have established so far show that the derivatives of $\Bpsi$ can be
bounded in terms of the source $\rho$ as
\begin{subequations}\label{muchasders}
\begin{align}
\|\Bpsi\|+\|D_y\Bpsi\|+\|D_y^2\Bpsi\|&\leq C\ep^2\|\rho\|\,,\\
\|\pd_\al\Bpsi\|+\|D_y\pd_\al\Bpsi\|&\leq C\ep\|\rho\|\,,\\
\|\pd_\al^2\Bpsi\|&\leq C\|\rho\|\,.
\end{align}
\end{subequations}
However, having in mind applications in forthcoming sections,we would
rather have estimates where the RHS always has a factor of $\ep^2$. 

In the following theorem we show that this can be
achieved by replacing the $L^2$~norm of $\rho$ by its $H^1$ norm (thus
using estimates that are weaker in terms of the gain of derivatives), and
provide a generalization for higher derivatives. It
is worth emphasizing that both the estimates~\eqref{muchasders} and
those in the following theorem are optimal with respect to
its dependence on the small parameter $\ep$, as can be checked easily.

\begin{theorem}\label{T.Hk}
For any nonnegative integers $j$ and $k$ we have the bound
\begin{gather*}
\|D_y^{j}\pd_\al^k\Bpsi\|\leq C_{jk}\ep^2\|\rho\|_{H^{J+k}}\,,
\end{gather*}
where $J:=\max\{j-2,0\}$.
\end{theorem}
\begin{proof}
  The proof proceeds by induction on $J+k$, that is, on the number of
  derivatives in the RHS of the inequality. The case $J+k=0$ follows from
  Propositions~\ref{L.H1psi} and~\ref{L.Bpsiy}. To avoid cumbersome
  notations that might obscure the argument, we will sketch the procedure for the case $J+k=1$, where
  one has to estimate $\pd_\al\Bpsi$, $D_y\pd_\al\Bpsi$ (improving the bounds in
  Propositions~\ref{L.H1psi} and~\ref{L.Bpsial}), $D_y^2\pd_\al\Bpsi$ and
  $D_y^3\Bpsi$. Once this case has been worked out in detail, it is
  straightforward to prove the general result using an induction argument.

  Let us begin by estimating the quantities having derivatives with
  respect to $\al$, that is, $\pd_\al\Bpsi$, $D_y\pd_\al\Bpsi$ and
  $D_y^2\pd_\al\Bpsi$. An easy computation using
  the expression of the Laplacian~\eqref{De} shows that the commutator
\[
\varrho:=\De(\pd_\al\Bpsi)-\pd_\al(\De\Bpsi)
\]
can be written as
\begin{align*}
\varrho&=\cO(\ep^{-1})\,
D_y^2\Bpsi+\cO(\ep)\,\Bpsi_{\al\al} + \cO(1)\,D_y\pd_\al\Bpsi +\cO(\ep^{-1})D_y\Bpsi+\cO(\ep)\Bpsi_\al\,.
\end{align*}
Now from the $H^2$ estimates proved in
Propositions~\ref{L.H1psi}--\ref{L.Bpsiy}, one obtains that the $L^2$ norm of
the commutator is bounded by
\[
\|\varrho\|\leq C\ep\|\rho\|\,.
\]
The function $\pd_\al\Bpsi$ obviously satisfies the zero-mean
condition
\begin{subequations}\label{syst}
\begin{equation}
\int_{\DD} \pd_\al\Bpsi(\al,y')\, dy'=0\,,
\end{equation}
the boundary condition
\begin{equation}
\pd_\nu(\pd_\al\Bpsi)=0
\end{equation}
on $\es\times\pd\DD$ and the equation 
\begin{equation}
\De(\pd_\al\Bpsi)=\pd_\al\Brho+\varrho\,.
\end{equation}
\end{subequations}
As the $L^2$ norm of the RHS is bounded by
$C\|\rho\|_{H^1}$, an immediate application of Propositions~\ref{L.H1psi} and~\ref{L.Bpsiy}
(applied to the boundary problem~\eqref{syst} rather than
to~\eqref{Depsi}) shows that
\[
\|\Bpsi_\al\|+ \|D_y\Bpsi_\al\|+\|D_y^2\Bpsi_\al\|\leq C\ep^2\|\rho\|_{H^1}\,.
\]

To estimate $D_y^3\Bpsi$ we will argue essentially as in the proof of
Proposition~\ref{L.Bpsiy}. One starts by proving interior estimates, which
are obtained by feeding the test function
$\vp:=\pd_a^2(\chi^2\pd_a^2\Bpsi)$ in the identity~\eqref{weakBpsi}. As
before, $\chi(y)$ is a smooth function which vanishes in a
neighborhood of $\pd\DD$ and is identically equal to $1$ in a disk of
radius~$R$ and the script $a$ (which is not summed) denotes any $y$
direction. In order to get boundary estimates, one can start by
noticing that the same argument we have used above to control
derivatives with respect to $\al$ also works for
$\pd_\theta\Bpsi\equiv y_1\,\pd_2\Bpsi-y_2\,\pd_1\Bpsi$. Indeed, the $L^2$
norm of the commutator
\[
\tilde\varrho:=\De(\pd_\te\Bpsi)-\pd_\te(\De\Bpsi)
\]
is also bounded
by $C\ep\|\rho\|$ as a consequence of
Propositions~\ref{L.H1psi}--\ref{L.Bpsiy}, and besides $\pd_\te\Bpsi$ satisfies the
boundary value problem
\begin{align*}
&\De(\pd_\te\Bpsi)=\pd_\te\Brho+\tilde\varrho\quad\text{in }
\es\times\DD\,,\\
&\pd_\nu(\pd_\te\Bpsi)=0\quad \text{on }\es\times\pd\DD,\qquad \int_{\DD} \pd_\te\Bpsi(\al,y')\, dy'=0\,.
\end{align*}
It then follows that the $L^2$ norm of the derivatives
$D_y^2\pd_\te\Bpsi$ are bounded by $C\ep^2\|\rho\|_{H^1}$. Hence one
can now differentiate the equation $\De\Bpsi=\Brho$ with respect to
$r$ and isolate~$\pd_{r}^3\Bpsi$ to show that the norm
$\|\pd_r^3\Bpsi\|_{\es\times\A_R}$ in a neighborhood of the boundary
is bounded by $C\ep^2\|\rho\|_{H^1}$, as claimed.

This proves the induction hypotheses for $J+k=1$. For higher values of
$J+k$, the idea is exactly the same. Derivatives with respect to $\al$
(or $\te$) are dealt with by taking derivatives directly in the
equation and invoking the induction hypotheses. To estimates the
highest derivative with respect to $r$, one combines interior
estimates with test function 
\[
\vp=\pd_a^{J+k+1}(\chi^2
\pd_a^{J+k+1}\Bpsi)
\]
with the direct isolation of $\pd_r^{J+k+2}\Bpsi$ in the equation
$\pd_r^{J+k}(\De\Bpsi-\Brho)=0$.
\end{proof}

\subsection{Pointwise estimates}
\label{SS.Ck}

Taking into account that $D_y\Bpsi=D_y\psi$, we can combine
the results in the previous two subsections to obtain $H^k$
estimates for $\psi$ in which the derivatives with respect to the fast
and slow variables are controlled in terms of different (and optimal)
powers of $\ep$. To
begin with, we can put together Propositions~\ref{L.psial} and~\ref{L.Bpsial} and
Theorem~\ref{T.Hk} to arrive at the following bounds:

\begin{theorem}\label{C.Hk}
The functions $\psi$ and $\Bpsi$ satisfy the $H^k$ estimate
\[
\|\psi\|_{H^{k+2}}\leq C_k\|\rho\|_{H^k}\,,\qquad \|\Bpsi\|_{H^k}\leq C_k\ep^2\|\rho\|_{H^k}\,.
\]
for any nonnegative integer $k$.
\end{theorem}
\begin{remark}
It is not hard to prove that these estimates are optimal with respect
to the dependence on $\ep$. To have a rough idea of why this is true,
the reader might want to consider the problem
\begin{equation}\label{model}
\psi_{\al\al}+\frac{\De_y \psi}{\ep^2}=\rho\quad \text{in
}\es\times\DD\,,\qquad \pd_\nu\psi=0 \quad \text{on
}\es\times\pd\DD\,,\qquad \int\psi\, d\al\, dy=0\,,
\end{equation}
which can be understood as a simplified version of the
problem~\eqref{Depsi}. Here and in what follows, $\De_y\psi:=\pd_1^2\psi+\pd_2^2\psi$ stands
for the standard Laplacian in the variables~$y$.
\end{remark}

For future reference, it is convenient to invoke the
Sobolev embedding theorem and record the following estimate for the $C^k$ norm of the function $\psi$:

\begin{theorem}\label{T.Ck}
The function $\psi$ satisfies the $C^k$ estimates
\[
\|\psi\|_{C^k}\leq C_k\|\rho\|_{H^{k}}\,,\qquad \|D_y\psi\|_{C^k}\leq C_k\ep^2\|\rho\|_{H^{k+3}}
\]
for any nonnegative integer $k$.
\end{theorem}
\begin{proof}
It follows immediately from Theorem~\ref{C.Hk} upon noticing that $D_y\Bpsi=D_y\psi$
and using the Sobolev inequality $\|\vp\|_{C^k}\leq C\|\vp\|_{H^{k+2}}$.
\end{proof}

\begin{remark}\label{R.modelLaplacian}
It is clear that the same bounds we have proved for the function $\psi$ hold true if we assume that $\psi$ solves the model problem~\eqref{model}
instead of $\De\psi=\rho$. In particular,
for future reference we record here that this function also satisfies the
estimates 
\begin{equation}\label{remarkeqm}
\|\psi\|_{H^{k+2}}+ \frac{\|D_y^2\psi\|_{H^{k}}}{\ep^2} \leq
C_k\|\rho\|_{H^{k}}\,,\qquad \|D_y\psi\|_{H^k}\leq C_k\ep^2\|\rho\|_{H^k}\,.
\end{equation}
\end{remark}

\section{Harmonic fields in thin tubes}
\label{S.harmonic}

In this section we use the estimates proved in Section~\ref{S.Laplacian} to
compute the harmonic field in a thin tube up to terms that are
suitably bounded for small~$\ep$.

We recall that a vector field $h$ in the tube $\cT_\ep$ is {\em harmonic}\/ if
it is divergence-free, irrotational, and tangent to the boundary. The
vector space of all harmonic fields in the tube will be denoted by
\begin{equation}\label{cH}
\cH(\tube):=\big\{ h\in C^\infty(\Om,\RR^3): \Div h=0,\; \curl h=0,\; h\cdot \nu=0\big\}\,.
\end{equation}
It is
standard that the space $\cH(\cT_\ep)$ is
one-dimensional, as it is isomorphic to the first cohomology group of
the tube with real coefficients.

Let us consider the vector field in $\es\times\DD$ defined by
\begin{equation}\label{h0}
h_0:=B^{-2}\big(\pd_\al+\tau\,\pd_\te\big)\,.
\end{equation}
It can be readily checked that $h_0$ is irrotational (with respect to
the metric~\eqref{dspolar}) and tangent to the boundary.

By the Hodge decomposition~\cite{FT78}, there is a function $\psi$
such that
\begin{equation}\label{h}
h:=h_0+\nabla \psi
\end{equation}
is harmonic. Clearly this is the only harmonic vector field in
$\es\times\DD$ up to a multiplicative constant. By $\nabla\psi$ we are denoting the gradient of the
function $\psi$ with respect to the metric~\eqref{dspolar}, that is,
\begin{equation}\label{nabla}
\nabla\psi= \frac{\psi_\al+\tau \psi_\te}{B^2}\,\pd_\al
+\frac{\psi_r}{\ep^2}\,\pd_r+
\frac{A\psi_\te+\ep^2r^2\tau\psi_\al}{(\ep rB)^2}\,\pd_\te\,.
\end{equation}
The fact that the field $h$ is divergence-free implies that $\psi$
solves the Neumann boundary value problem
\[
\De \psi=\rho\quad \text{in
}\es\times\DD\,,\qquad \pd_\nu\psi=0 \quad \text{on
}\es\times\pd\DD\,,
\]
where
\begin{equation}\label{eqrho}
\rho:=\ep B^{-3} r(\tau\ka\sin\te-\ka'\cos\te)
\end{equation}
is minus the divergence of $h_0$ (and, as such, satisfies $\int \rho\,
dV=0$). We will also assume that $\int
\psi\,d\al\, dy=0$ in order to determine $\psi$ uniquely. 

In the following two sections we will need some estimates for the
harmonic field~$h$ (or, equivalently, for the function $\psi$) that
depend on the particular form of the source term $\rho$. These
estimates are obtained in the following theorem, where we
calculate, up to some controllable error, some derivatives of the
function $\psi$ that we will use later on. To simplify the notation,
we will write $\cO(\ep^n)$ for any function
$\chi$ satisfying the bound $\|\chi\|_{C^k}\leq C_k\ep^n$ for all $k$
(and also for numbers whose absolute value is smaller than $C\ep^n$,
but the meaning should be clear from the context).

\begin{theorem}\label{T.psi0}
Consider the functions
\begin{align}\label{vp0}
\vp_0&:=\frac{\ep^3(r^3-3r)}8\,(\tau\ka\sin\te-\ka'\cos\te)\,,\\
\vp_1&:=\frac{13\ep^4(r^4-2r^2)}{96}\,(\tau\ka^2\sin2\te-\ka\ka'\cos2\te)\,,\label{vp1}
\end{align}
which are obviously of order $\cO(\ep^3)$ and $\cO(\ep^4)$, respectively. Then
$\psi$ is related to these functions through the estimates
\begin{align*}
\psi&=\cO(\ep^2) \,,\\
D_y\psi&=D_y\vp_0+\cO(\ep^4)\,,\\
\pd_\te\psi&=\pd_\te\vp_0+\pd_\te\vp_1+ \cO(\ep^5)\,.
\end{align*}
\end{theorem}

\begin{proof}
Let $\rho_0:=\ep r(\tau\ka\sin\te-\ka'\cos\te)$. It is easy to see that the function $\vp_0(\al,y)$
is the only solution to the problem
\[
{\De_y\vp_0}={\ep^2}\rho_0\quad\text{in
}\es\times\DD\,,\qquad \pd_r\vp_0|_{r=1}=0\,,\qquad \int_{\DD}\vp_0\, dy=0
\]
for any value of the angle $\al$.

Consider now the function $\psi_1:=\psi-\vp_0$, which obviously
has zero normal derivative on $\es\times\pd\DD$ and has zero mean
because so do $\psi$ and $\vp_0$. The Laplacian of~$\psi_1$ is given by
\begin{align*}
\De\psi_1&=\rho-\De\vp_0=(\rho-\rho_0)+\bigg(\frac{\De_y\vp_0}{\ep^2}-\De\vp_0\bigg)\,.
\end{align*}
Using Eq.~\eqref{eqrho}, the first term in brackets can be easily shown to be
\[
\rho-\rho_0=3\ep^2\ka r^2\cos\te\big(\tau\ka\sin\te-\ka'\cos\te\big)+\cO(\ep^3)\,,
\]
while the second can be easily dealt with using the formula~\eqref{De}
for the Laplacian:
\[
\frac{\De_y\vp_0}{\ep^2}-\De\vp_0=\frac\ka\ep\bigg(\cos\te\,\pd_r\vp_0-\frac{\sin\te\,\pd_\te\vp_0}r\bigg)+\cO(\ep^3)\,.
\]
Using the definition of $\vp_0$, this yields
\[
\De\psi_1=\rho_1+\rho_2+\cO(\ep^3)\,,
\]
with
\begin{align*}
\rho_1&:=\frac{13\ep^2\ka
  r^2}8\big[\ka\tau\sin2\te-\ka'\cos2\te\big]\,,\\
\rho_2&:= \frac{\ep^2 (3-14r^2)\ka\ka'}8
\end{align*}
of order $\cO(\ep^2)$, so the estimates we proved in Theorem~\ref{T.Ck} then ensure that
\[
\psi_1=\cO(\ep^2)\quad\text{and}\quad D_y\psi_1=\cO(\ep^4)\,.
\]
This shows that $\psi=\cO(\ep^2)$ and $D_y\psi= D_y\vp_0+\cO(\ep^4)$.

To calculate $D_y\psi$ up to $\cO(\ep^5)$ we will consider two
auxiliary functions $\vp_1$ and~$\vp_2$. The function $\vp_1(\al,y)$ is the
solution to the problem
\begin{align*}
{\De_y\vp_1}=\ep^2\rho_1\quad\text{in
}\es\times\DD\,,\qquad \pd_r\vp_1|_{r=1}=0\,,\qquad \int_{\DD}\vp_1\, dy=0\,.
\end{align*}
The existence and uniqueness of the solution are standard given that
the RHS satisfies
\[
\int_{\DD}\rho_1\, dy=0
\] 
for all $\al$. In fact, the solution can be computed in closed form
using separation of variables, which readily yields the formula for
$\vp_1$ given in the statement (Eq.~\eqref{vp1}). 

The function $\vp_2$ is the solution to the problem
\begin{align*}
\pd_\al^2\vp_2+\frac{\De_y\vp_2}{\ep^2}=\rho_2\quad\text{in
}\es\times\DD\,,\qquad \pd_r\vp_2|_{r=1}=0\,,\qquad
\int\vp_2\, d\al\, dy=0\,.
\end{align*}
Again, the existence and uniqueness of solutions is standard because
the RHS satisfies the zero-mean condition
\[
\int \rho_2\, d\al\,d y=0\,.
\]
(The reason why we are including derivatives with respect to $\al$ in
the definition of~$\vp_2$ but not in that of $\vp_1$ is that the
source term $\rho_2$ has zero mean when averaged with respect to $\al$
and $y$ but not when averaged in $y$ only.) Obviously $\vp_2$ does not
depend on $\te$.

We can now estimate $\psi_2:=\psi_1-\vp_1-\vp_2$ using the same
argument we used with $\psi_1$. We start by noticing that,  by construction,
$\pd_r\psi_2=0$ when $r=1$ and the integral
$\int_{\es\times\DD}\psi_2\,d\al\, dy$ is zero. The Laplacian of
$\psi_2$ is
\begin{align*}
\De\psi_2&=\rho_1+\rho_2-\De\vp_1-\De\vp_2+\cO(\ep^3) \\
&=\bigg(\frac{\De_y\vp_1}{\ep^2}-\De\vp_1\bigg)+\bigg(\pd_\al^2\vp_2+\frac{\De_y\vp_2}{\ep^2}-\De\vp_2\bigg)+\cO(\ep^3)\\
&=\cO(\ep^3)\,.
\end{align*}
To pass to the third line we have used the expression of the Laplacian
in polar coordinates (Eq.~\eqref{De}) and of $\vp_1$, as well as the estimates for $\vp_2$
that stem from Theorem~\ref{T.Ck} and
Remark~\ref{R.modelLaplacian}. Another application of
Theorem~\ref{T.Ck} shows that $D_y\psi_2=\cO(\ep^5)$. Since
$\pd_\te\vp_2=0$ and 
\[
\psi=\vp_0+\vp_1+\vp_2+\psi_2\,,
\]
the theorem follows.
\end{proof}

\section{Beltrami fields with prescribed harmonic part}
\label{S.Beltrami}

Our goal in this section is to construct a Beltrami field $v$ in the thin
tube $\cT_\ep$, tangent to the boundary, whose harmonic part is a fixed harmonic
field $h$. We will be
particularly interested in the way the Beltrami field $v$ is related
to the harmonic field $h$ as the parameter $\la$ tends to zero, since
in the next section it will be crucial to have good estimates for 
this relation in order to compute some dynamical quantities of the
field $v$.

This section is divided in three parts. In the first subsection,
we prove an existence result for a boundary value problem for the
curl operator in tubes in which we can prescribe the harmonic part of
the solution (Corollary~\ref{C.existence}). In the second subsection
we will provide estimates for an auxiliary vector equation with
constant coefficients (Proposition~\ref{P.Hklocal}), which are used in
the third subsection to prove the desired estimates, optimal in~$\ep$, for the boundary
value problem under consideration (Theorem~\ref{T.estimBeltrami}).


\subsection{An existence result for the curl operator}
\label{SS.existence}

Let us begin by making precise what we understand by the harmonic part
of a vector field $w$ in the thin tube $\cT_\ep$ that is tangent to
the boundary. We will define its {\em harmonic part} to be the vector field 
\[
\cP w:= \frac{h}{\|h\|_{L^2(\cT_\ep)}^2}\int_{\cT_\ep} h\cdot w \, dx\,,
\]
that is, its projection to the space of harmonic vector fields
$\cH(\cT_\ep)$, as introduced in Eq.~\eqref{cH}. {In this subsection}, we will denote by
\[
\|v\|_{H^k(\cT_\ep)}^2:=\sum_{j=0}^k\int_{\cT_\ep} |D^jv|^2\, dx
\]
the usual~$H^k$ norm of a vector field in the tube
$\cT_\ep\subset\RR^3$ and write $L^2(\cT_\ep)$ for~$H^0(\cT_\ep)$. Throughout, we
will use the notation $C_\ep$ for positive constants, possibly not
uniformly bounded in the small parameter~$\ep$, that may vary from
line to line.

In the following proposition we present the basic existence result
that we will use to show the existence of Beltrami fields with
prescribed harmonic part. The result is probably known to some experts
but we have not found it in the literature. The proof relies on a duality argument for a suitable
energy functional and the Fredholm alternative theorem. 

\begin{proposition}\label{P.existence}
Let $f$ be an $L^2$ vector field in $\cT_\ep$ which is
divergence-free. There is a countable subset of the real line without
accumulation points such that, if the constant $\la$ does not belong
to it, then the equation 
\begin{align}\label{eqwf}
\curl w-\la w=f\,,\qquad \Div w=0
\end{align}
has a unique $H^1$ solution $w$ that is tangent to the boundary~$\pd\cT_\ep$ and
has zero harmonic part.
\end{proposition}
\begin{proof}
Let us consider the Hilbert spaces of vector fields
\begin{align*}\label{cW}
\cF&:=\big\{ F\in L^2(\cT_\ep,\RR^3):\Div F=0\big\}\,,\\
\cW&:=\big\{ w\in H^1(\cT_\ep,\RR^3): w \text{ is tangent to }
\pd\cT_\ep\big\}\,,
\end{align*}
where the tangency condition is to be understood in terms of
traces. It is well known~\cite{FT78} that, because of the tangency
condition imposed on~$\cW$, the $H^1$ norm is equivalent
to the norm
\[
\|w\|_{\cW}^2:=\|\curl w\|_{L^2(\cT_\ep)}^2+\|\Div w\|_{L^2(\cT_\ep)}^2+\|\cP w\|_{L^2(\cT_\ep)}^2
\]
in the sense that, for all $w\in\cW$,
\[
\frac{\|w\|_{\cW}}{C_\ep}\leq \|w\|_{H^1(\cT_\ep)}\leq C_\ep\|w\|_{\cW}\,.
\]

Consider the scalar product $E:\cW\times \cW\to\RR$ associated to the
norm $\|\cdot\|_{\cW}$, given by
\[
E[w,u]:=\int_{\cT_\ep}\Big(\curl w\cdot \curl u + \Div w\,
\Div u+ \cP w\cdot\cP u\Big)\, dx\,.
\]
By the Riesz representation theorem, for any $L^2$ vector field $F$ there is a unique $w_F\in\cW$
such that
\begin{equation}\label{Ew}
E[w_F,u]=\int_{\cT_\ep} F\cdot \curl u\, dx
\end{equation}
for all $u\in\cW$.

We claim that, for any $F\in \cF$, the condition~\eqref{Ew} is equivalent to demanding that
$w_F$ be an $H^1$ solution to the equation 
\begin{align}\label{probwF}
\curl w_F=F\,,\qquad \Div w_F=0
\end{align}
in the tube $\cT_\ep$ that is tangent to the boundary and has zero
harmonic part. One side of the implication is obvious, so we only need
to prove that if $w_F$ satisfies the condition~\eqref{Ew} with
$F\in\cF$, then it solves Eq.~\eqref{probwF}.

For this, let us take suitable choices of the field $u$ in
Eq.~\eqref{Ew}. Letting $u=\cP w_F$ be the harmonic part of $w_F$, we
obtain that $\cP w_F$ is zero. To see that $\Div w_F$ is zero too, it
suffices to take $u=\nabla \psi$, with $\psi$ being the solution to the boundary
value problem
\[
\De\psi=\Div w_F\quad\text{in }\cT_\ep\,,\qquad
\pd_\nu\psi|_{\pd\cT_\ep}=0\,,\qquad \int_{\cT_\ep}\psi\, dx=0\,.
\]
Notice that, in order to show that the solution exists and $\nabla\psi$ belongs to the
space~$\cW$, we need to use that $w_F$ is in $H^1$ and tangent to the
boundary. 

Since $\cP w_F$ and $\Div w_F$ are zero, the condition~\eqref{Ew} can then
be written as
\begin{equation}\label{intcT}
\int_{\cT_\ep}\curl G\cdot u\, dx+\int_{\pd\cT_\ep} (G\times \nu)\cdot
u\, d\si=0
\end{equation}
for all $u\in\cW$, where 
\[
G:=\curl w_F-F
\]
and $d\si$ denotes the
induced surface measure on the boundary. Letting $u$ vary over the space
of smooth vector fields of compact support in $\cT_\ep$, we immediately
infer that $G$ must be irrotational. Now that we know the first term
in~\eqref{intcT} is always zero, we can take an arbitrary field $u$
tangent to the boundary to derive that $G\times\nu $ must be
zero, so that $G$ is divergence-free,
irrotational and orthogonal to the boundary. By the Hodge
decomposition theorem, this ensures the field $G$ is identically
zero. Hence we have proved that $w_F$ is the only solution to the
problem~\eqref{probwF}. Notice that, as an immediate consequence of
the equivalence of the $H^1$ norm to $\|\cdot\|_{\cW}$ on $\cW$, 
\begin{equation}\label{estimwF}
\|w_F\|_{H^1(\cT_\ep)}\leq C_\ep\|F\|_{L^2(\cT_\ep)}
\end{equation}
for some constant that depends on $\ep$ but not on $F$.

Let us consider the operator $K$ mapping a field $F\in\cF$ to its solution
$w_F$, which we regard as a linear map $K:\cF\to\cF$ whose image lies in $\cW$. It is standard in view
of the estimate~\eqref{estimwF} that the operator $K$ is compact, so
by the Fredholm alternative the equation
\begin{equation}\label{eqK}
Kw-\frac1\la w=\tilde f
\end{equation}
has a unique solution for each $\tilde f\in\cF$ and any constant
$1/\la$ that does not belong to the spectrum of the adjoint operator
$K^*$. This spectrum is a bounded countable set that only accumulates
at zero. 

Let $w$ be the only solution to Eq.~\eqref{eqK} with $\tilde
f=-Kf/\la$, for an arbitrary field $f$ in~$\cF$. Since $w= K(\la w+f)$ belongs
to the image of $K$, from the above discussion it stems that the field $w$ is tangent to the boundary,
divergence-free and has zero harmonic part, and by the definition of
$K$ satisfies $\curl w= \la w+f$. The proposition then follows.
\end{proof}
\begin{remark}\label{R.Fredholm}
  The proof of the proposition shows that the values of $\la$ for
  which there is not a unique solution to the Eq.~\eqref{eqwf} are
  given by the reciprocal of the eigenvalues of the compact operator
  $K^*$ that we defined in the proof.
\end{remark}

As a direct application of the previous proposition, we derive the
following corollary, which gives the existence result for Beltrami
fields with prescribed harmonic part that we will use later on. Of
course, it is apparent that this corollary is totally different from
the results that one has on compact manifolds without boundary, where
the Beltrami equation has no solutions but for $\la$ belonging to a
countable set and these solutions (i.e., the eigenfunctions of
$\curl$) are necessarily orthogonal to harmonic fields.

\begin{corollary}\label{C.existence}
For any constant $\la$ not belonging to certain $\ep$-dependent
countable set without accumulation points, there is a unique solution
to the equation
\begin{align*}
\curl v=\la v\,,\qquad \Div v=0
\end{align*} 
in the tube $\cT_\ep$ that is tangent to the boundary and whose
harmonic part is $\cP v=h$.
\end{corollary}
\begin{proof}
It follows immediately from Proposition~\ref{P.existence} by taking
$v=h+w$, with $w$ being the only solution to Eq.~\eqref{eqwf}
with $f=\la h$.
\end{proof}

\subsection{Estimates for an equation with constant coefficients}
\label{SS.constant}

In this section we will prove estimates for a curl-type equation with
constant coefficients on the domain $\es\times\DD$. This equation with
constant coefficients is closely related to the Beltrami
equation on the thin tube $\cT_\ep$ and will be used subsequently to
estimate the difference between the Beltrami field $v$ and its
harmonic part $h$. As before, the natural coordinates on the domain
$\es\times\DD$ will be denoted by $(\al, y)$, and we will consider
polar coordinates $(\al,r,\te)$ where convenient.

Vector fields on $\es\times\DD$ are regarded as functions
$W:\es\times\DD\to\RR^3$, whose components are denoted by 
\[
W=(W_\al,W_1,W_2)\,.
\]
We will sometimes write $W_y:=(W_1,W_2)$. On these
vector fields we will consider the action of the differential operators
\begin{align*}
\ddiv W&:=\pd_\al W_\al+\pd_1 W_1+\pd_2 W_2\,,\\
\Curl W&:=\bigg(\pd_1W_2-\pd_2 W_1, \frac{\pd_2 W_\al}{\ep^2}-\pd_\al
W_2,\pd_\al W_1-\frac{\pd_1 W_\al}{\ep^2}\bigg)\,,
\end{align*}
and we will say that a vector field $W$ is tangent to the boundary if
$y\cdot W_y=0$ (in the sense of traces) on the torus $|y|=1$. Obviously, $\ddiv$ and $\Curl$
are related to the standard divergence and curl operators through a
rescaling, but for our purposes this form of the operators is more convenient.

We will also consider the
functional $\cQ$ that maps each field $W$ to the real number
\[
\cQ [W]:=\int W_\al\, d\al\, dy\,.
\]
This is obviously related with the projection onto the space of
``harmonic fields'' associated to the operators $\ddiv$ and $\Curl$, which is spanned by the constant
field $(1,0,0)$, so when $\cQ[W]=0$ we will simply say that the 
harmonic part of $\cQ[v]$ is zero. All integrals are taken over $\es\times\DD$ unless
otherwise stated. To control the behavior of the vector fields, in
addition to the usual $H^k$ norms we will consider the norms
\begin{equation}\label{defHkep}
\|W\|_{H^k_\ep}^2:=\|W_\al\|^2_{H^k}+\ep^2\|W_y\|^2_{H^k}\,.
\end{equation}
These norms should not be confused with the norm $\|\cdot\|_{\dot
  H^1_\ep}$ that we considered for scalar functions in
Section~\ref{S.Laplacian}. We will simply write $\|\cdot\|_\ep$ for
the corresponding $L^2_\ep$~norm ($k=0$), and reserve the notation
$\|\cdot\|$ for the usual $L^2$ norm, where factors of $\ep$ do not appear.

In the following proposition we compute the lowest eigenvalue of a
self-adjoint operator associated with $\Curl$ to show how the
$L^2_\ep$ norm of a vector field $W$ can be controlled using that of
$\Curl W$:

\begin{proposition}\label{P.eigenW}
Let $W$ be an $H^1$ vector field on $\es\times\DD$ with $\ddiv W=0$ that is tangent to
the boundary and has zero harmonic part. Then
\[
\|\Curl W\|_\ep\geq \frac {C\ms \|W\|_\ep}\ep\,.
\]
\end{proposition}
\begin{proof}
It is an easy consequence of~\cite{Giga} that $\Curl$ defines an unbounded
self-adjoint operator on the Hilbert space 
\[
\cH:=\big\{ W\in L^2(\es\times\DD,\RR^3): \ddiv W=0,\; \cQ[W]=0,\;
W \text{ tangent to the boundary}\big\}\,,
\]
endowed with the scalar product associated with the norm $\|\cdot\|_\ep$. The domain of this operator
consists of the $H^1$ vector fields in $\cH$ such that $\Curl W$ is
also in $\cH$. Hence, to prove the proposition it is enough to see that the
eigenvalues of this operator, in absolute value,
satisfy $|\mu|> C/\ep$.

Exploiting the symmetry of the equations (that is, rotation of the
angles~$\al$ and~$\te$), it is not hard to see that
the eigenvalue equation
\[
\Curl W =\mu W\,,
\]
with $W$ divergence-free, tangent to the boundary and with zero
harmonic part, can be solved in closed form. Indeed, the symmetry
ensures that the eigenfunctions can be chosen of the form
\[
W=\e^{\I n\al+\I m\te}\big( v_1(r)\,e_\al +v_2(r)\, e_r+ v_3(r)\,e_\te\big)\,,
\]
where $n,m$ are integers, $v_j(r)$ are functions of  the radial
variable and the unit vectors $e_\al,e_r,e_\te$ are defined in the
obvious way:
\[
e_\al:=(1,0,0)\,,\quad e_r:=(0,\cos\te,\sin\te)\,,\quad e_\te:=(0,-\sin\te,\cos\te)\,.
\]
Using this expression, a tedious but straightforward computation shows that the
eigenvalues with the smallest absolute value are $\pm{j^{(1)}_1}/\ep$, 
where $j^{(1)}_1$ denotes the first positive zero of the Bessel
function $J_1$.
\end{proof}

In the proof of the main result of this subsection we will need the
following identity: 

\begin{lemma}\label{L.identity}
Let $W$ be a vector field tangent to the boundary of
$\es\times\DD$. Then
\begin{multline*}
\ep^2\|\pd_\al W_y\|^2+\|D_yW_y\|^2+ \|\pd_\al W_\al\|^2+\ep^{-2}\|D_y
W_\al\|^2=
\|\ddiv W\|^2\\+\|\Curl W\|_\ep^2-\int_{\es\times\pd\DD} |W_y|^2\,
d\al\, d\te\,,
\end{multline*}
where we are writing $d\al\, d\te$
for the induced surface measure on $\es\times\pd\DD$.
\end{lemma}
\begin{proof}
One can easily check that
\begin{multline*}
\|\ddiv W\|^2+\|\Curl W\|^2_\ep=
\|\pd_\al W_\al\|^2 \\
+ \sum_{i=1}^2\bigg(\ep^2\|\pd_\al W_i\|^2+ \frac{\|\pd_i W_\al\|^2}{\ep^{2}}\bigg)+\sum_{i,j=1}^2\|\pd_i W_j\|^2+S\,,
\end{multline*}
where
\[
S:=\sum_{i,j\in\{\al,1,2\}}\int\big(\pd_i
W_i\,\pd_j W_j- \pd_jW_i\, \pd_iW_j\big)\, d\al\, dy\,.
\]
We can now integrate by parts to write
\begin{align*}
S&=-\sum_{i,j=1}^2\int_{\es\times\pd\DD}(y_j W_i\,\pd_i W_j+y_jW_\al\,
\pd_\al W_j)\, d\al\, d\te\\
&=\sum_{i,j=1}^2\int_{\es\times\pd\DD}W_j\,\pd_i(y_j W_i)\,d\al\,
d\te\\
&=\sum_{i=1}^2\int_{\es\times\pd\DD}W_i^2\,d\al\, d\te\,.
\end{align*}
To pass to the second identity we have used that $y\cdot W_y$ is zero
on the boundary, so the second summand in the first integrand
vanishes, and taken advantage of the fact that $W$ is tangent to the
boundary to integrate by parts a second time.
\end{proof}

In the following proposition, which is the main result in this subsection, we prove an estimate for the
operators $\ddiv$ and $\Curl$ that is ``optimal'' with respect to the
small parameter $\ep$. In the proof of this version of the inequalities
we will also derive another in which the RHS has one derivative less
than the LHS, as is customary. However, the estimate we will need
later on is the former, which is why it is the one that appears in
the statement. It is not hard to check that the dependence on~$\ep$ of both estimates is sharp.

\begin{proposition}\label{P.Hklocal}
Let $W$ be any vector field in $\es\times\DD$ that is tangent to the
boundary and satisfies the equation
\[
\Curl W=F\,,\qquad \ddiv W=\rho
\]
for a scalar function $\rho$ and a vector field $F$. Then 
\begin{align*}
\|W_\al\|_{H^k}+ \|D_y W_\al\|_{H^k}&\leq C_k\big(\ep\|F\|_{H^{k}_\ep}
+ \|\rho\|_{H^{k}}+\big|\cQ[W]\big|\big)\,,\\
\|\pd_\al W_\al\|_{H^k} &\leq C_k\big(\|F\|_{H^k_\ep}+\|\rho\|_{H^k}\big)\,,\\
\|W_y\|_{H^k}+ \ep\|\pd_\al W_y\|_{H^k} + \|D_y W_y\|_{H^k}&\leq C_k\big(\|F\|_{H^{k}_\ep} + \|\rho\|_{H^{k}}+\big|\cQ[W]\big|\big)
\end{align*}
for constants that depend on $k$ but not on $\ep$.
\end{proposition}

\begin{proof}
We can write the field $W$ as the sum of three fields:
\begin{equation}\label{confuso}
W=V+\bigg(\pd_\al\psi,\frac{\pd_1\psi}{\ep^2},
\frac{\pd_2\psi}{\ep^2}\bigg)+\bigg(\frac{\cQ[W]}{|\es\times\DD|},0,0\bigg)\,.
\end{equation}
The scalar function $\psi$ that appears in the second vector field is defined as
the only solution to the Neumann boundary value problem
\begin{equation}\label{decompositionW}
\pd_\al^2\psi+\frac{\De_y\psi}{\ep^2}=\rho\quad\text{in
}\es\times\DD\,,\qquad \pd_\nu\psi|_{\es\times\pd\DD}=0\,,\qquad \int_{\es\times\DD}\psi\,d\al\, dy=0\,,
\end{equation}
and it should be noticed that the third vector field (which corresponds to
the harmonic part of $W$) is constant. As a consequence of these
definitions and the properties of $W$, the field $V$ is
tangent to the boundary, has zero harmonic part ($\cQ[V]=0$) and
satisfies the equation
\begin{equation}\label{eqV}
\Curl V=F\,,\qquad \ddiv V=0\,.
\end{equation}

The $H^k$ estimates stated in Remark~\ref{R.modelLaplacian} (Eq.~\eqref{remarkeqm}), applied to
the boundary problem~\eqref{decompositionW}, provide suitable control
of the second field that appears in Eq.~\eqref{confuso} (that
is, the ``gradient'' part) and its derivatives, as they show that
\begin{align*}
\|\pd_\al\psi\|_{H^k}+ \|\pd_\al^2 \psi\|_{H^k} + \frac{\|D_y \pd_\al\psi\|_{H^k}}\ep&\leq C \|\rho\|_{H^{k}}\,,\\
\frac{\|D_y\psi\|_{H^k}}{\ep^2}+ \frac{\|D_y\pd_\al\psi\|_{H^k}}\ep + \frac{\|D_y^2 \psi\|_{H^k}}{\ep^2}&\leq C \|\rho\|_{H^{k}}\,.
\end{align*}
The third field in Eq.~\eqref{confuso} is trivial to control as
it is constant. Therefore, to prove the proposition it is enough to
derive suitable estimates for the field $V$.

Hence, our goal is to show that the vector field $V$ satisfies
\begin{subequations}\label{HkV}
\begin{align}
\|V_\al\|_{H^k}+ \|D_y V_\al\|_{H^k} &\leq C\ep\|F\|_{H^k_\ep}\,,\\
\|\pd_\al V_\al\|_{H^k} &\leq C\|F\|_{H^k_\ep}\,,\\
 \|V_y\|_{H^k}+ \ep\|\pd_\al V_y\|_{H^k} + \|D_y V_y\|_{H^k} &\leq C\|F\|_{H^k_\ep}\,.
\end{align}
\end{subequations}
For this we start by noticing that, when applied to Eq.~\eqref{eqV}, the $L^2$~estimate proved in
Proposition~\ref{P.eigenW} yields $\|V\|_\ep\leq C\ep\|F\|_\ep$, or equivalently
\begin{equation}\label{L2V}
\|V_\al\|\leq C\ep\|F\|_\ep\,,\qquad \|V_y\|\leq C\|F\|_\ep\,.
\end{equation}
To prove the inequalities~\eqref{HkV}, we start by using estimates for
boundary traces and interpolation to control the term
\begin{align*}
S&:=\int_{\es\times\pd\DD}
|V_y|^2=\|V_y\|_{L^2(\es\times\pd\DD)}^2 \\
&\leq
C\|V_y\|^2_{L^2(\es)\times H^{1/2}(\DD)}\leq C\|V_y\|\big(\|V_y\|+\|D_yV_y\|\big)
\end{align*}
that appears in Lemma~\ref{L.identity}. Together with the $L^2$~estimate~\eqref{L2V}, we can then apply Lemma~\ref{L.identity} to the
field $V$ to infer that
\begin{align*}
&\|\pd_\al V_\al\|\leq C\|F\|_\ep\,,\qquad &\|D_y V_\al\|\leq C\ep\|F\|_\ep\,,\\
&\|\pd_\al V_y\|\leq \frac C\ep\|F\|_\ep\,,\qquad &\|D_yV_y\|\leq
C\|F\|_\ep\,.
\end{align*}
This proves the estimate~\eqref{HkV} for $k=0$.

The derivation of higher-order estimates from these inequalities is
standard. To show the basic ideas, let us sketch the proof of the
$k=1$ estimates. Interior estimates are obtained by considering the
identity shown in Lemma~\ref{L.identity} for the vector 
field $\pd_i(\chi V)$, where $i=1,2$ and $\chi$ is a smooth
function, compactly supported in $\es\times\DD$ and equal to one in
the region $|y|<1/2$. Besides, it is easy to
estimate derivatives of $V$ with respect to $\al$ because the field
$\pd_\al^jV$ is still tangent to the boundary, divergence-free, has
zero harmonic part and satisfies the equation
\[
\Curl(\pd_\al^j V)= \pd_\al^j F\,.
\]
Analogous properties can be shown too for the
globally defined field
\[
\pd_\te V+(0,V_2,-V_1)\,,
\]
with $\pd_\te:=y_1\,\pd_2-y_2\,\pd_1$ the angular derivative, so the
same argument can be applied for this field. Since the second summand is obviously
controlled by the previous estimates, this yields good $H^1$~estimates
for $\pd_\te V$. To get estimates for the field~$V$ up to the boundary, it now
suffices to control the radial derivative $\pd_rV$ in the region
$|y|>1/2$, and this can be readily done by writing the
equation~\eqref{eqV} in polar coordinates, taking its derivative with
respect to $r$ and isolating the terms having two radial
derivatives. This process can be readily iterated to get the desired $H^k$
estimates. The details, which can be easily filled using these comments, are omitted.
\end{proof}

\subsection{Estimates for Beltrami fields in thin tubes}
\label{SS.Beltrami}

In this subsection we establish some estimates for Beltrami fields
defined on the tube $\cT_\ep$. The point of these estimates will be to
control the difference between a Beltrami field $v$ and its harmonic
part $h$ in terms of the Beltrami parameter $\la$, while ensuring that
the dependence on~$\ep$ of these estimates is optimal. As we did in Section~\ref{S.Laplacian}
(but not in Subsection~\ref{SS.existence}), we will identify $\cT_\ep$
with $\es\times\DD$  through the coordinates $(\al,y)$. 

Given a vector field $v$ in the tube, we will consider its components $(v_\al,v_1,v_2)$,
defined through the relation
\[
v=v_\al\, \pd_\al+v_1\, \pd_1+ v_2\,\pd_2\,.
\] 
We will use these coordinates to define the norms $H^k_\ep$,
just as in Eq.~\eqref{defHkep}, of vector fields in $\cT_\ep$,
using the obvious formula:
\[
\|v\|_{H^k_\ep}^2:=\|v_\al\|_{H^k}^2+\ep^2\|v_y\|_{H^k}^2\,.
\]
Here the Sobolev norms in the RHS denote the usual scalar norms and we
are using the notation $v_y:=(v_1,v_2)$. It should be noticed that
these norms reflect to some extent the effect of the metric on the way
vectors are measured; in particular, the
$L^2$ norm defined by the above formula is obviously equivalent
through constants that do not depend on~$\ep$ to the
perhaps more appealing expression
\[
\int g(v,v)\, dV\,,
\]
where $g$ denotes the expression of the Euclidean metric in these
coordinates and $dV$ is the normalized volume (cf.\ Eq.~\eqref{ds} and
Eq.~\eqref{dV}). 

We will begin with the following result, whose proof hinges on the
analogous estimates proved for an equation with constant coefficients
that we established in Proposition~\ref{P.Hklocal}:

\begin{proposition}\label{P.estimv}
Suppose that the vector field $w$ is 
tangent to the boundary, has zero harmonic part ($\cP w=0$) and satisfies the equation
\begin{equation}\label{Eqwf}
\curl w=f\,,\qquad \Div w=0
\end{equation}
in $\cT_\ep$. Then 
\[
\|w\|_{H^k_\ep}\leq C_k\ep\|f\|_{H^k_\ep}\,.
\]
\end{proposition}
\begin{proof}
In the notation of Subsection~\ref{SS.constant}, let us consider the
vector field $W:\es\times\DD\to\RR^3$ given by
\[
W:=\big(Aw_\al+\ep^2\tau(y_2w_1-y_1w_2), w_1+\tau y_2w_\al, w_2-\tau y_1w_\al\big)\,,
\]
where we recall that $A$ is the function defined in Eq.~\eqref{A1}. It
can be checked that~$W$ is given by the components (in the coordinates
$(\al,y)$) of the $1$-form
dual to $w$ up to a rescaling of the $y$~components by a factor of
$\ep^{-2}$. It is worth noticing that, as the field $w$ is tangent to
the boundary of the tube $\cT_\ep$, a simple computation using the
definition of $W$ shows that the field $W$ is tangent to the boundary
$\es\times\pd\DD$, that is, $y\cdot W_y=0$ on $|y|=1$. Moreover, the components of the field $w$
on the tube can be recovered from those of $W$ through the relations
\begin{subequations}\label{Wtow}
\begin{align}
w_\al&=\frac{W_\al+\ep^2\tau(y_1W_2-y_2W_1)}{B^2}\,,\\
w_1&=W_1-\frac{\tau y_2[W_\al+\ep^2\tau(y_1W_2-y_2W_1)]}{B^2}\,,\\
w_2&=W_2+\frac{\tau y_1[W_\al+\ep^2\tau(y_1W_2-y_2W_1)]}{B^2}\,,
\end{align}
\end{subequations}
where the function $B$ was defined in Eq.~\eqref{B}.

Using the notation
\[
B_1:=\frac B{B^2+\ep^2\tau^2y_2^2}\,,\qquad B_2:=\frac B{B^2+\ep^2\tau^2y_1^2}\,,
\]
we will also consider the field $F:\es\times\DD\to\RR^3$ and the
function $\rho_W:\es\times\DD\to\RR$ defined by
\begin{align*}
F&:=\bigg(\frac{Af_\al+\ep^2\tau(y_2f_1-y_1f_2)}B, \frac{f_1+\tau
  y_2f_\al}{B_1}, \frac{f_2-\tau y_1f_\al}{B_2}\bigg)\,,\\
\rho_W&:= \pd_\al\bigg[W_\al\bigg(1-\frac{1}B\bigg)\bigg]+\pd_1\bigg[W_1\bigg(1-\frac{1}{B_1}\bigg)\bigg]+ \pd_2\bigg[W_2\bigg(1-\frac{1}{B_2}\bigg)\bigg]\,.
\end{align*}
Notice that, by the definition of the functions $A$ and $B$, we
trivially have
\begin{align}\label{rhoF}
\|F\|_{H^k_\ep}&\leq C\|f\|_{H^k_\ep}\,,\\
\|\rho_W\|_{H^k}&\leq C\ep\big(\|W\|_{H^k}+\|\pd_\al W_\al\|_{H^k}
+\|D_y W_y\|_{H^k}\big)\,.\label{rhoW}
\end{align}

A straightforward computation shows that Eq.~\eqref{Eqwf} can
be written in terms of the field $W$ as
\begin{equation}\label{WFrhoW}
\Curl W=F\,,\qquad \ddiv W=\rho_W\,.
\end{equation}
To derive estimates for $W$, we start by analyzing $\cQ[W]$. To this
end, we recall (cf.~e.g.~\cite{Giga}) that if $d\si$ denotes the
induced surface measure on
the disk 
\[
\{\al=\al_0\}\subset\cT_\ep\,,
\]
the fact that the field $w$ has zero
harmonic part ($\cP w=0$) is equivalent to the assertion that
\[
0=\int_{\{\al=\al_0\}} g(w,\nu)\, d\si=\int_{\DD} B
w_\al\big|_{\al=\al_0}\, dy 
\]
for each angle $\al_0$. Since $A-B$ is of order $\ep$, using the above equality one can estimate
\begin{align}
|\cQ[W]|&= \bigg|\int W_\al\, d\al\, dy\bigg|=\bigg|\int
\big(Aw_\al+\ep^2\tau (y_2w_1-y_1w_2)\big)\, d\al\, dy\bigg|\notag\\
&=\bigg|\int
\big((A-B)w_\al+\ep^2\tau (y_2w_1-y_1w_2)\big)\, d\al\, dy\bigg|\notag\\
&\leq C\ep\|w_\al\|+C\ep^2\|w_y\|\notag\\
&\leq C\ep\|W_\al\|+C\ep^2\|W_y\|\,,\label{QW}
\end{align}
where to derive the last inequality we have used the expression of the
components of $w$ in terms of those of  $W$ given in Eq.~\eqref{Wtow}.

We are now ready to derive some estimates for the vector field
$W$. Since $W$ is tangent to $\es\times\pd\DD$, we can apply Proposition~\ref{P.Hklocal} to Eq.~\eqref{WFrhoW} to
obtain
\begin{subequations}\label{rollofin}
\begin{align}
\|W_\al\|_{H^k}&\leq C\big(\ep\|F\|_{H^{k}_\ep}
+ \|\rho_W\|_{H^{k}}+\big|\cQ[W]\big|\big)\,,\\
\|\pd_\al W_\al\|_{H^k} &\leq C\big(\|F\|_{H^k_\ep}+\|\rho_W\|_{H^k}\big)\,,\\
\|W_y\|_{H^k}+ \|D_y W_y\|_{H^k}&\leq C\big(\|F\|_{H^{k}_\ep} + \|\rho_W\|_{H^{k}}+\big|\cQ[W]\big|\big)\,.
\end{align}
\end{subequations}
Going back to Eq.~\eqref{rhoW}, this yields
\[
\|\rho_W\|_{H^k}\leq C\ep\big(\|F\|_{H^{k}_\ep}+ \|W_\al\|_{H^k}+\|W_y\|_{H^k}\big)\,.
\]
Plugging this inequality and the bound~\eqref{QW} for $\cQ[W]$ in Eq.~\eqref{rollofin}, we immediately get
\begin{align*}
\|W_\al\|_{H^k}&\leq C\ep\big(\|F\|_{H^{k}_\ep} +\phantom{\ep}\|W_\al\|_{H^k}+\phantom{\ep}\|W_y\|_{H^k}\big)\,,\\
\|W_y\|_{H^k}&\leq C \phantom{\ep}\big(\|F\|_{H^{k}_\ep} +\ep \|W_\al\|_{H^k}+\ep\|W_y\|_{H^k}\big)\,.
\end{align*}
From these estimates one readily infers that, for small $\ep$,
\[
\|W_\al\|_{H^k}+\ep\|W_y\|_{H^k}\leq C\ep\|F\|_{H^{k}_\ep}\,.
\]
In view of the formulas~\eqref{Wtow}, $\|w\|_{H^k_\ep}\leq
C\|W\|_{H^k_\ep}$, so the proposition follows from the
bound~\eqref{rhoF} for the $H^k_\ep$ norm of $F$.
\end{proof}


We are now ready to prove the main result in this section, which
estimates the difference between a Beltrami field in the tube and its
harmonic part in terms of the Beltrami parameter $\la$:

\begin{theorem}\label{T.estimBeltrami}
Let $\la$ be any nonzero real constant that is smaller in absolute value than
some fixed positive constant $\La$. For small enough $\ep$, the
problem
\[
\curl v=\la v
\]
has a unique solution in the tube $\cT_\ep$ that is tangent to the
boundary and whose harmonic part $\cP v$ is the harmonic
field $h$. Moreover, the difference between the field $v$ and the
harmonic field is bounded pointwise by
\[
\|v_\al-h_\al\|_{C^k(\es\times\DD)}+ \ep\|v_y-h_y\|_{C^k(\es\times\DD)}\leq C_{k,\La}\ep\,|\la|\,,
\]
where the constant only depends on $k$ and $\La$.
\end{theorem}
\begin{proof}
Let us write $w:=v-h$, so that the field $w$ is tangent to the
boundary, has zero harmonic part and satisfies the equation
\begin{equation}\label{eqwh}
\curl w=\la w+\la h\,,\qquad \Div w=0\,.
\end{equation}
We proved in Proposition~\ref{P.existence} and Remark~\ref{R.Fredholm}
that this equation has a unique solution if and only if $1/\la$ is not
an eigenvalue of the operator $K^*$ introduced in the proof of the
aforementioned proposition. As $\la$ is real, this is equivalent to $1/\la$
being eigenvalue of the operator $K$, which means that there is a
divergence-free $L^2$~field $u$ such that
\[
Ku=\frac u\la\,.
\]

Suppose that $1/\la$ is an eigenvalue of $K$. By the properties of $K$ proved in Proposition~\ref{P.existence}, this
means that the nonzero vector field $Ku$ is tangent to the
boundary, has zero harmonic part and satisfies
\[
\curl(Ku)=\la\ms {Ku}\,,\qquad \Div (Ku)=0\,.
\]
Applying the estimate proved in Proposition~\ref{P.estimv} to this
equation, we get
\[
\|Ku\|_{L^2_\ep}\leq C\ep|\la|\|Ku\|_{L^2_\ep}\,,
\]
which means that 
\[
|\la|> C/\ep
\]
for some positive constant that does
not depend on $\ep$.

Therefore the problem~\eqref{eqwh} has a unique solution for all
$|\la|\leq \La$, provided $\ep$ is small enough. Moreover, when
applied to Eq.~\eqref{eqwh}, Proposition~\ref{P.estimv} ensures that
\[
\|w\|_{H^k_\ep}\leq C_k\ep|\la|\big(\|w\|_{H^k_\ep}+\|h\|_{H^k_\ep}\big)\,,
\]
which yields
\[
\|w\|_{H^k_\ep}\leq C_k\ep|\la|\|h\|_{H^k_\ep}
\]
provided that $\ep$ is smaller than some constant of the form $C/\La$. Since the $H^k_\ep$ norm of the
harmonic field is bounded uniformly in $\ep$ as
\[
\|h\|_{H^k_\ep}\leq C_k
\] 
due to Eq.~\eqref{h} and the estimates in Theorem~\ref{T.psi0},
we obtain
\[
\|w\|_{H^k_\ep}\leq C_k\ep|\la|\,.
\]
$C^s$ pointwise estimates for $w$ are immediately obtained
from this bound by taking $k=s+2$ and using Sobolev embeddings.
\end{proof}

\section{A KAM theorem for Beltrami fields with small $\la$}
\label{S.KAM}

Let us consider the harmonic field $h$ in the tube $\cT_\ep$, as
introduced in Eq.~\eqref{h}. As before, we will assume that the
thickness $\ep$ is small. By Theorem~\ref{T.estimBeltrami}, for any
$\la$ smaller in absolute value than some fixed, $\ep$-independent
constant, there is a unique solution $v$ to the Beltrami equation
\[
\curl v=\la v
\]
in $\cT_\ep$ that is tangent to the boundary and whose harmonic part
is $h$. We are interested in the case where the Beltrami constant
$\la$ is suitably small.  {\em For simplicity of notation, throughout this section we
will take $\la:=\ep^3$ (although we could have taken any nonzero constant
$\la=\cO(\ep^3)$) and refer to the vector field $v$ corresponding to this choice of $\la$ as the local Beltrami field}.

Our objective in this section is to study some fine dynamical
properties of the local Beltrami field $v$. More precisely, we will
show that for small values of $\ep$ and ``most'' core curves $\ga$,
the boundary of the tube $\cT_\ep$ is an invariant torus of the field
$v$ that is preserved (i.e., there is a small perturbation of
$\pd\cT_\ep$ that is still invariant) under suitably small
perturbations of $v$. For this, we will see that the key point is the
analysis of the harmonic field $h$, which is close to the local
Beltrami field $v$ as a consequence of the estimates proved in
Theorem~\ref{T.estimBeltrami} and the fact that the Beltrami parameter
$\la=\ep^3$ is small. At this point, it is worth emphasizing that, to some extent, the
proofs of the dynamical properties that we study in this section ultimately
depend on the estimates derived in Section~\ref{S.Laplacian}. In
particular, if the dependence on $\ep$ of these estimates were worse,
we would not be able to check the nondegeneracy conditions of the KAM
theorem we prove in this section.

This section is divided into four parts. In
Subsection~\ref{SS.trajectories} we consider the trajectories of the
local Beltrami field, after a rescaling of the field (which
does not alter their geometric structure), and calculate these
trajectories perturbatively in a suitable range of time
(Proposition~\ref{L.trajectories}). In Subsections~\ref{SS.flow}
and~\ref{SS.nondegeneracy} we use these perturbative expressions to
compute the rotation number and normal torsion of the Poincaré map of
the local Beltrami field (Theorems~\ref{T.rotation}
and~\ref{T.Herman}), which are the quantities that control the
stability of the torus in the KAM theorem that we establish in
Subsection~\ref{SS.KAM} (Theorem~\ref{T.KAM}).

\subsection{Trajectories of the local Beltrami field}
\label{SS.trajectories}

In this subsection we aim to compute the trajectories of the local Beltrami
field $v$
perturbatively in the small parameter $\ep$. From
Theorem~\ref{T.estimBeltrami} and the fact that the Beltrami parameter
is $\la=\ep^3$, it follows that the local Beltrami field
$v$ is close to the harmonic field in the sense that
\begin{equation}\label{vhep3}
\|v_\al-h_\al\|_{C^k(\es\times\DD)}<C_k\ep^4\,,\qquad \|v_y-h_y\|_{C^k(\es\times\DD)}<C_k\ep^3\,.
\end{equation}
Since the harmonic field can be written as $h=h_0+\nabla\psi$
(with $h_0$ and $\nabla\psi$ respectively given by~\eqref{h0}
and~\eqref{nabla}), from the estimates for $\psi$ proved in
Theorem~\ref{T.psi0} we infer that
\begin{align*}
v_\al&=B^{-2}(1+\psi_\al+\tau\psi_\te)+\cO(\ep^4)\\
&=1+\cO(\ep)\,.
\end{align*}
In particular, as $v_\al$ does not vanish for small $\ep$, we can consider the analytic vector field
\begin{equation}\label{defX}
X:=\frac v{v_\al}\,.
\end{equation}
We will use this vector field to study the geometric structure of the
trajectories of the local Beltrami field, since both fields have the
same unparametrized trajectories and the vector field $X$ presents certain computational
advantages, as we shall see in the next
subsection. Before we go on, and identifying the tube
$\cT_\ep$ with $\es\times\DD$ through the coordinates $(\al,y)$, let
us note that the field $X$ is well defined in a small neighborhood of
$\es\times\overline{\DD}$ because so is the local Beltrami field $v$.

The trajectories of $X$ are given by the parametrization
$(\al(s),r(s),\te(s))$, where these functions satisfy the system of
ODEs
\begin{subequations}\label{ODEs}
\begin{align}
\dot\al&= 1\,,\label{ODEal}\\
\dot r & =\frac{B^2\psi_r}{\ep^2(1+\psi_\al+\tau \psi_\te)}+\cO(\ep^3)\,,\label{ODEr}\\
\dot \te&= \frac{\tau+(\ep r)^{-2}A\psi_\te + \tau\psi_\al}{1+\psi_\al
  + \tau \psi_\te}+\cO(\ep^3)\,.\label{ODEte}
\end{align}
\end{subequations}
These equations can be read off from the definition of the field $X$ and its connection with the harmonic field (Eq.~\eqref{vhep3})
and the formulas~\eqref{h} and~\eqref{nabla} for the harmonic field
and the gradient of $\psi$. The point of the trajectory not only
depends on the ``flow parameter'' $s$, but also on the initial
conditions $(\al_0,r_0,\te_0)$ at $s=0$. Without loss of generality, in this section we will always take
$\al_0=0$, and make the dependence of the trajectory on $(r_0,\te_0)$
explicit by writing
\[
\big(\al(s;r_0,\te_0),r(s;r_0,\te_0),\te(s;r_0,\te_0)\big)
\]  
when appropriate.

Throughout, it will be convenient to denote by $\te$ not only the
angular coordinate in $\SS^1$, but also its lift to the real
line. It should be noticed that the formulas we will give below are actually valid for the
lifted coordinate too, which will be of use in the following subsection.

In the following lemma we will compute the trajectory of the field $X$
at time $s\in[0,\ell]$ up to a controllable error. We will assume that
$r_0$ is bounded away from~$0$ so that the trajectory cannot reach the
coordinate singularity $\{r=0\}$ at any time $s\in[0,\ell]$. This is
convenient in view of the terms $1/r^2$ that appear in the equations
and is not a restriction for the applications that we have in mind, as
we will be only concerned with initial conditions near the invariant
torus $\{r=1\}$. For simplicity, in this lemma we will
abuse the notation and denote by $\cO(\ep^j)$ a quantity $Q(r_0,\te_0,s)$ that is uniformly
bounded as
\[
\big|\pd_{r_0}^k\pd_{\te_0}^l Q(r_0,\te_0,s)\big|< C_{kl}\ep^j
\]
for $r_0$ in any fixed compact set of the interval $(0,1]$
(which is the domain where polar coordinates define a diffeomorphism),
$\te_0\in\SS^1$ and $s\in[0,\ell]$, provided $\ep$ is small enough.

\begin{proposition}\label{L.trajectories}
Consider the solution to the system~\eqref{ODEs} with initial condition
$(0,r_0,\te_0)$ and $r_0>0$. At time $s\in[0,\ell]$, this solution is given by
\begin{align*}
\al(s;r_0,\te_0)&=s\,,\\
r(s;r_0,\te_0)&= r_0+ \cO(\ep)\,,\\
\te(s;r_0,\te_0)&= \teo(s)+ \ep\ms \te^{(1)}(s) + \ep^2\ms \te^{(2)}(s)+\cO(\ep^3)\,,
\end{align*}
where each quantity $\te^{(j)}(s)\equiv \te^{(j)}(s;r_0,\te_0)$ is of order $\cO(1)$ and given by
\begin{align*}
\teo(s)&:=\te_0+\int_0^s \tau(\bs)\,d\bs\,,\\
\te^{(1)}(s)&:=\frac{r_0^2-3}{8r_0}\big[\ka(s)\sin\teo(s)-\ka(0)\sin\te_0\big]\,,\\
\te^{(2)}(s)&:=\frac{12-5r_0^2}{32}\int_0^s\ka(\bs)^2\,\tau(\bs)\,
d\bs +\frac{3(r_0^4+2r_0^2-3)\ka(s)\ka(0)}{64r_0^2}\cos\te_0\,\sin\teo(s)\\
&\;-\frac{(3-r_0^2)^2\ka(s)\ka(0)}{64r_0^2}\sin\te_0\cos\teo(s) +\frac{(27-50r_0^2+25r_0^4)\ka(s)^2}{384r_0^2}\sin2\teo(s)\\
&\; +\frac{(27+14r_0^2-31r_0^4)\ka(0)^2}{384r_0^2}\sin2\te_0\,.
\end{align*}
\end{proposition}

\begin{proof}
Converting the ODEs~\eqref{ODEs} into integral equations, it is
clear that, for $s\in[0,\ell]$,
\begin{subequations}\label{ODEs2}
\begin{align}
\al(s;r_0,\te_0)&=s\,,\\
r(s;r_0,\te_0)&=r_0+\int_0^{s} \frac{B^2\psi_r}{\ep^2(1+\psi_\al+\tau
  \psi_\te)} d\bs\label{ODEr2}+\cO(\ep^3)\,,\\
\te(s;r_0,\te_0)&= \te_0+ \int_0^s \frac{\tau+(\ep r)^{-2}A\psi_\te + \tau\psi_\al}{1+\psi_\al
  + \tau \psi_\te} d\bs +\cO(\ep^3)\,.\label{ODEte2}
\end{align}
\end{subequations}
In these equations all the functions under the integral signs are
evaluated along the trajectories, i.e., at the
point
\begin{equation}\label{traj-rollo}
\al=\bs\,,\quad r=r(\bs;r_0,\te_0)\,,\quad \te=\te(\bs;r_0,\te_0)\,.
\end{equation}

Let us solve the equations perturbatively. We start by noticing
that, as a consequence of the bounds for $\psi$ derived in
Theorem~\ref{T.psi0} (and its connection with the functions $\vp_0,\vp_1$
introduced in this theorem), the integrands can be expanded in $\ep$ as
\begin{align}\notag
\frac{\tau+(\ep r)^{-2}A\psi_\te + \tau\psi_\al}{1+\psi_\al
  + \tau \psi_\te}&=\bigg(\tau+\frac{A\psi_\te}{\ep^2 r^{2}} + \tau\psi_\al\bigg)\big(1-\psi_\al
  +\cO(\ep^3)\big)\\
&= \tau+\frac{A\psi_\te}{\ep^2 r^{2}} +\cO(\ep^3)\,.\label{fracte}\\
\frac{B^2\psi_r}{\ep^2(1+\psi_\al+\tau
  \psi_\te)}&= \frac{\pd_r\vp_0}{\ep^2}+\cO(\ep^2)\,,\label{fracr}
\end{align}
Since $r_0>0$ and $s\in[0,\ell]$ (which allows us to control the
effect of the denominator $1/r^2$ for small enough $\ep$), we immediately infer from
Eqs.~\eqref{ODEr2} and~\eqref{ODEte2} that
\begin{equation}\label{orden0}
r(s;r_0,\te_0)=r_0+\cO(\ep)\,,\qquad \te(s;r_0,\te_0)=\teo(s)+\cO(\ep)\,.
\end{equation}
Of course, we define each function $\te^{(j)}(s)$ as in
the statement of the proposition.

Now that we have the zeroth-order expression of the trajectories, we
will next compute them up to second-order corrections. For
convenience we will use the notation
\[
\cR_0(r):=\frac{r^3-3r}8\,,\qquad \cR_1(r):=\frac{13(r^4-2r^2)}{96}
\]
for the dependence of $\vp_0$ and $\vp_1$ on $r$, respectively.
We start with the analysis of 
the radial coordinate of the trajectories. Using again Eq.~\eqref{fracr} and the
zeroth-order estimates for the trajectories~\eqref{orden0}, we derive that
\begin{align}
r(s;r_0,\te_0)&=r_0+\int_0^{s}
\bigg[\frac{\pd_r\vp_0(\bs,r_0+
  \cO(\ep),\teo(\bs)+ \cO(\ep))}{\ep^2}+\cO(\ep^2)\bigg] d\bs +\cO(\ep^3) \notag\\
&=r_0+\int_0^s \frac{\pd_r\vp_0(\bs,r_0,\teo(\bs))}{\ep^2}d\bs+
\cO(\ep^2) \notag\\
&=
r_0+\ep\ms\cR_0'(r_0)\int_0^s\big[\tau(\bs)\ka(\bs)\sin\teo(\bs)-\ka'(\bs)\cos\teo(\bs)\big]d\bs+
\cO(\ep^2) \notag\\
&=r_0-\ep\ms\cR_0'(r_0)\int_0^s\frac
d{d\bs}\Big(\ka(\bs)\cos\teo(\bs)\Big)\, d\bs+
\cO(\ep^2) \notag\\
&=r_0+\ep\ms r^{(1)}(s)+\cO(\ep^2)\,,\label{orden1r}
\end{align}
where
\[
r^{(1)}(s):=\frac{3(1-r_0^2)}8\big[ \ka(s)\cos\teo(s)-\ka(0)\cos\te_0\big]\,.
\]
 To pass to the second line we have used the mean value theorem and the
obvious $C^k$ bound $\vp_0=\cO(\ep^3)$, and to complete the
calculation we have just plugged in the formulas for $\vp_0$
(Eq.~\eqref{vp0}) and the definition of $\teo(s)$. In a totally
analogous manner we can compute $\te(s;r_0,\te_0)$ up to $\cO(\ep^2)$:
\begin{align}
\te(s;r_0,\te_0)&=\teo(s)+\int_0^s \bigg[\frac{A\psi_\te}{(\ep
  r)^2}+\cO(\ep^3)\bigg]\bigg|_{(s,r_0+\cO(\ep),\teo(s)+\cO(\ep))}
d\bs +\cO(\ep^3)\notag\\
&=\teo(s)+\int_0^s \frac{\pd_\te\vp_0(\bs,r_0,\teo(s))}{(\ep
  r_0)^2}d\bs+\cO(\ep^2)\notag\\
&=\teo(s)+\frac{\ep\ms\cR_0(r_0)}{r_0^2}\int_0^s\big[\tau(\bs)\ka(\bs)\cos\teo(\bs)+\ka'(\bs)\sin\teo(\bs)\big]d\bs\notag\\
&=\teo(s)+\ep\ms\te^{(1)}(s)+\cO(\ep^2)\,.\label{orden1te}
\end{align}

To complete the proof we need to calculate $\te(s;r_0,\te_0)$ up to
$\cO(\ep^3)$. The procedure is as above but the computations are more
tedious. We start by noticing that the term $A\psi_\te$ that appears
in the integrand~\eqref{fracte} can be written as
\begin{align}\label{fracte2}
A\psi_\te=\pd_\te\vp_0+\big(\pd_\te\vp_1-2\ep\ka r\cos\te\,\pd_\te\vp_0\big)+\cO(\ep^5)
\end{align}
where we have used the estimates
we established in Theorem~\ref{T.psi0} (recall that $\vp_1$ was
introduced in Eq.~\eqref{vp1}). Notice that the first
summand is of order $\cO(\ep^3)$, while the term in brackets is of
order $\cO(\ep^4)$. 

We can now use the integral equation~\eqref{ODEte2} for the
trajectories, the expansion~\eqref{fracte} for the integrand (together
with~\eqref{fracte2}), and the expression for the trajectories up to
second order terms (Eqs.~\eqref{orden1r} and~\eqref{orden1te}) to get
\begin{align*}
\te(s;& \,r_0,\te_0)=\teo(s)+\int_0^s\frac{\pd_\te\vp_0}{\ep^2r^2}\bigg|_{(\bs,r_0+\ep
  r^{(1)}+\cO(\ep^2),\teo+\ep\te^{(1)}+\cO(\ep^2))} d\bs\\
&  +\int_0^s\frac{\pd_\te\vp_1-2\ep\ka
  r\cos\te\,\pd_\te\vp_0}{\ep^2r^2}\bigg|_{(\bs,r_0+\ep r^{(1)}+\cO(\ep^2),\teo+\ep\te^{(1)}+\cO(\ep^2))}d\bs+\cO(\ep^3)\,.
\end{align*}
For simplicity of notation, we omit the argument $\bs$ when there is no
risk of confusion. An elementary Taylor expansion and the mean value theorem show the
first integral (let us call it $I_1$) is given by
\begin{align*}
I_1&= \int_0^s\bigg[\frac{\pd_\te\vp_0}{\ep^2r^2}+ \pd_r\bigg(\frac{\pd_\te\vp_0}{\ep
  r^2}\bigg) r^{(1)}(\bs)+  \frac{\pd_\te^2\vp_0}{\ep
  r^2}\te^{(1)}(\bs)\bigg]\bigg|_{(\bs,r_0,\teo)} d\bs+\cO(\ep^3)\notag\\
&=\ep\te^{(1)}(s)+\int_0^s\bigg[\pd_r\bigg(\frac{\pd_\te\vp_0}{\ep
  r^2}\bigg) r^{(1)}(\bs)+  \frac{\pd_\te^2\vp_0}{\ep
  r^2}\te^{(1)}(\bs)\bigg]\bigg|_{(\bs,r_0,\teo)} d\bs+\cO(\ep^3)\,.
\end{align*}
The second integral, which we call $I_2$, can be immediately simplified using the
mean value theorem, finding that
\begin{align*}
I_2=\int_0^s\frac{\pd_\te\vp_1-2\ep\ka
  r\cos\te\,\pd_\te\vp_0}{\ep^2r^2}\bigg|_{(\bs,r_0,\teo)}d\bs+\cO(\ep^3)\,.
\end{align*}

The integrals $I_1$ and $I_2$ can be computed in closed form after replacing the functions
$\vp_0,\vp_1$ by their expressions, given in
Eqs.~\eqref{vp0} and \eqref{vp1}. For example, 
\[
I_2=J_1+J_2+\cO(\ep^3)\,,
\]
where
\begin{align*}
J_1&:=\int_0^s\frac{\pd_\te\vp_1}{\ep^2r^2}\bigg|_{(\bs,r_0,\teo)}d\bs\\
&=\frac{\ep^2\cR_1(r_0)}{r_0^2}\int_0^s\big[\tau(\bs)\ka(\bs)^2\cos
2\teo(s)+\ka(\bs)\ka'(\bs)\sin2\teo(\bs)\big]\,d\bs\\
&=\frac{\ep^2\cR_1(r_0)}{2r_0^2}\int_0^s\frac
d{d\bs}\Big(\ka(\bs)^2\sin 2\teo(\bs)\Big)\, d\bs\\
&=\frac{13\ep^2(r_0^2-2)}{96}\big[\ka(s)^2\sin 2\teo(s)-\ka(0)^2\sin 2\te_0\big]
\end{align*}
and
\begin{align*}
J_2&:=-2\int_0^s\frac{\ka
  \cos\te\,\pd_\te\vp_0}{\ep r}\bigg|_{(\bs,r_0,\teo)}d\bs\\
&=-\frac{2\ep^2\cR_0(r_0)}{r_0}\int_0^s \ka\cos\teo \big[\tau\ka
\cos\teo+\ka'\sin\teo\big]d\bs\\
&=-\frac{\ep^2\cR_0(r_0)}{r_0}\int_0^s\bigg[\tau\ka^2+\frac12\frac
d{d\bs}\Big(\ka^2\sin2\teo\Big)\bigg]d\bs\\
&=\frac{\ep^2(3-r_0^2)}{16}\bigg[2\int_0^s\tau\ka^2\,
d\bs+\ka(s)^2\sin 2\teo(s)-\ka(0)^2\sin 2\te_0\bigg]\,.
\end{align*}
The other terms can be dealt with using analogous arguments,
arriving at the formula for $\te^{(2)}(s)$ that appears in the statement.
\end{proof}

To conclude this subsection, we will show that the trajectories of the
field $X$ on the invariant torus $r=1$ satisfy certain functional
equation up to some controllable errors. The reason why we need to
consider this way of describing trajectories on the invariant torus is
that, in order to compute the rotation number later on, we will need
to understand the trajectories on the torus for arbitrarily large
times. The expression for the trajectories we obtained in
Proposition~\ref{L.trajectories} is not well suited for this purpose,
while the functional equation below turns out to be much more
convenient. In order to describe the errors that appear in the
functional equation, in the following proposition we will use the
notation $s\,\cO(\ep^n)$ for any quantity $Q(\te_0,s)$ that is bounded as
\[
\big|\pd_{\te_0}^jQ(\te_0,s)\big|\leq C_j(1+|s|)\ep^n\quad
\text{and}\quad 
\big|\pd_s\pd_{\te_0}^jQ(\te_0,s)\big|\leq C_j\ep^n
\]
for nonnegative integer $j$ (we could have considered higher
derivatives with respect to $s$ too, but we will not need this feature).

\begin{proposition}\label{L.trajectories2}
Consider the trajectories of the system of ODEs~\eqref{ODEs} with
initial condition $(\al_0,r_0,\te_0)=(0,1,\te_0)$. The function $\te(s)\equiv \te(s;1,\te_0)$ satisfies the approximate
functional equation
\begin{equation*}
\te(s)= \te_0+\int_0^s\tau(\bs)\,d\bs
-\frac\ep4\Big[\ka(s)\sin\te(s)-\ka(0)\sin\te_0\big] +s\, \cO(\ep^2)\,.
\end{equation*}
\end{proposition}

\begin{proof}
The starting point is the differential equations for the
trajectories~\eqref{ODEs} with initial radius $r_0=1$. It is obvious that
the radial component of the trajectory is 
\[
r(s;1,\te_0)=1\,,
\]
which simply shows that the set $\{r=1\}$ is an invariant
torus. Therefore, Eq.~\eqref{ODEte} for $\te(s)$ becomes
\begin{equation}\label{ODEte3}
\dot\te(s)=\frac{\tau+\ep ^{-2}A\psi_\te + \tau\psi_\al}{1+\psi_\al
  + \tau \psi_\te}+\cO(\ep^3)=\tau + \frac{\ep ^{-2}A\psi_\te - \tau^2\psi_\te}{1+\psi_\al
  + \tau \psi_\te}+\cO(\ep^3)\,,
\end{equation}
where the functions in the RHS are evaluated along the trajectories
$(s,1,\te(s))$. Expanding the fraction
that appears in the second identity using the estimates
for~$\psi$ proved in Theorem~\ref{T.psi0} we therefore arrive at
\begin{equation}\label{dotte}
\dot\te(s)=\tau(s)+ \frac{\pd_\te\vp_0(s,1,\te(s))}{\ep^2}+\cO(\ep^2)\,.
\end{equation}
Notice that the fraction is of order $\cO(\ep)$. Converting~\eqref{ODEte3} into an integral
equation, an immediate consequence of this estimate is that
\begin{equation*}
\te(s)=\teo(s)+\int_0^s\frac{\pd_\te\vp_0(\bs,1,\te(\bs))}{\ep^2}d\bs+s\,\cO(\ep^2)\,.
\end{equation*}

This integral can be evaluated, modulo $s\,\cO(\ep^2)$,  using the formula~\eqref{dotte}
for the derivative $\dot\te$, the expression for $\vp_0$,  and integration by parts: 
\begin{align*}
\int_0^s\frac{\pd_\te\vp_0}{\ep^2}d\bs&=
-\frac\ep4\int_0^s
\big[\ka\tau\cos\te+\ka'\sin\te\big]\, d\bs\\
&=-\frac\ep4\int_0^s\bigg[\ka(\dot\te+\cO(\ep))\cos\te
+\ka'\sin\te\bigg]d\bs\\
&=-\frac\ep4\int_0^s \frac
d{d\bs}\Big(\ka\sin\te\Big)d\bs
+s\,\cO(\ep^2)\\
&=-\frac\ep4\big[\ka(s)\sin\te(s)-\ka(0)\sin\te_0\big]+s\,\cO(\ep^2)\,.
\end{align*}
Here, of course, all the integrands are evaluated at the point
$(\bs,1,\te(\bs))$.
\end{proof}

\subsection{Rotation number of the Poincaré map of the local Beltrami field}
\label{SS.flow}

We will denote by $\phi_{s}$ the time-$s$ flow of the field $X$,
which maps each point $(\al_0,r_0,\te_0)$ to the trajectory of the ODEs~\eqref{ODEs}
at time $s$ that has the latter values as initial
conditions. Since the field $X$, introduced in
Eq.~\eqref{defX}, is tangent to the boundary of the
domain~$\es\times\DD$, it is standard that the flow $\phi_s$ is a well
defined diffeomorphism of $\es\times\DD$ for all values of $s$.

Let us now consider the Poincar\'e map of the field $X$, which is the
tool we will use to analyze the dynamical properties of the flow (and
which coincides with that of the local Beltrami field $v$). For this, we
start by considering the section $\{\al=\al_0\}$, which is
clearly transverse to the vector field $X$. The Poincar\'e map of this section,
$\Pi_{\al_0}:\overline{\DD}\to \overline{\DD}$, sends each point $(r_0,\te_0)\in \overline{\DD}$ to the first
point at which the trajectory $\phi_s(\al_0,r_0,\te_0)$ intersects the
section $\{\al=\al_0\}$
(with $s>0$). The reason why we are considering the field $X$ is that
it is {\em isochronous}\/ in the sense that this
first return point is given by the time-$\ell$ flow of $X$,
that is,
\begin{equation}\label{Poincare}
\Pi_{\al_0}(r_0,\te_0)=\phi_{\ell}(\al_0,r_0,\te_0)\,.
\end{equation}
We will omit the subscript when $\al_0=0$, and use Cartesian
coordinates $y$ in the disk when convenient. 

One should notice that, since we are assuming that
the curve $\ga$ is analytic, the boundary of the tube $\cT_\ep$ is
also an analytic surface, so it is standard~\cite{Mo58} that the field
$v$ is analytic in a neighborhood of the closure
$\overline{\cT_\ep}$. This ensures that the Poincaré map is also a well-defined
analytic map in a neighborhood of the closed disk $\overline{\DD}$.

In the following proposition we will show that the
Poincar\'e map of the Beltrami field preserves a measure on the
disk. For later convenience, we will state this result in terms of the
associated $2$-form rather than the measure:

\begin{proposition}\label{P.area}
The Poincaré map $\Pi$ preserves the positive measure on the disk corresponding
to the $2$-form
\[
\La:=G_\La(r,\te)\,r\, dr\wedge d\te
\]
on the disk $\DD$, with
\begin{equation}\label{defF}
G_\La(r,\te):=B v_\al\big|_{\al=0}=1+\ep\ms \ka(0)\, r\cos\te+\cO(\ep^2)\,.
\end{equation}
\end{proposition}
\begin{proof}
That the function $G_\La(y):=Bv_\al|_{\al=0}$  has
indeed the form given by the RHS of~\eqref{defF} is an immediate
consequence of the estimates for the function $\psi$ proved in
Theorem~\ref{T.psi0} and Eq.~\eqref{vhep3}.
Given a Borel set $\cB\subset\overline{\DD}$ and a small positive
$\de$, let us denote by 
\[
\mu(\cB):=\int_{\cB} \La
\]
its area and let
\[
\cB_\de:=(-\de,\de)\times\cB
\]
be a small thickening of the set $\{0\}\times\cB$ in the closed domain $\es\times\overline{\DD}$.

Since the divergence of $v$ is zero, from the definition of $X$ it
stems that its flow preserves the volume
\[
d\tilde V:=v_\al\, dV\,.
\]
Clearly the $\tilde V$-volume of the set $\cB_\de$ is 
\begin{align}
\tilde V(\cB_\de)&:=\int_{\cB_\de}\,d\tilde V=\int_{-\de}^\de \int_{\cB} B v_\al
\,d y \, d\al\notag\\
&=2\de\ms\big(1+\cO(\de)\big)\int_{\cB}B v_\al\big|_{\al=0}\,
dy=2\de\ms\mu(\cB)+\cO(\de^2)\,. \label{Vol1}
\end{align}
Let us now observe that the image of the set $\cB_\de$ under the
time-$\ell$ flow $\phi_\ell$ is given by
\begin{equation}\label{phiell}
\phi_\ell(\cB_\de)=\bigcup_{-\de<\al<\de} \{\al\}\times \Pi_{\al}(\cB)\,.
\end{equation}
By the continuous dependence of the flow on the initial conditions and
Eq.~\eqref{Poincare}, the Poincaré maps corresponding to different
values of the angle $\al$ satisfy 
\[
\|\Pi_{\al}-\Pi\|_{C^0(\DD)}\leq C\de
\]
for 
$|\al|\leq \de$, so we can use the decomposition~\eqref{phiell} to
show that the $\tilde V$-volume of $\phi_\ell(\cB_\de)$ is
\begin{align}
\tilde V(\phi_\ell(\cB_\de))&=\int_{-\de}^\de \int_{\Pi_{\al}(\cB)} B v_\al
\,dy \, d\al\notag\\
&=\int_{-\de}^\de \bigg(\int_{\Pi(\cB)} B v_\al\big|_{\al=0}
\,dy+\cO(\de)\bigg) \, d\al\notag\\
&=2\de\ms\mu(\Pi(\cB))+\cO(\de^2)\,.\label{Vol2}
\end{align}
Equating the $\tilde V$-volumes of $\cB_\de$ and $\phi_\ell(\cB_\de)$,
given by Eqs.~\eqref{Vol1} and~\eqref{Vol2}, and considering
small values of $\de$ we then obtain that 
\[
\mu(\cB)=\mu(\Pi(\cB))\,,
\]
as claimed.
\end{proof}

Since the local Beltrami field $v$ is tangent to the boundary  of the domain $\es\times
\DD$, the image of $\pd\DD$ under the Poincaré map $\Pi$ is also contained
in $\pd\DD$. Hence, the restriction of $\Pi$ to $\pd\DD$ defines an
analytic diffeomorphism of the circle, which will be denoted by
\[
\Pi|_{\pd\DD}:\pd\DD\to\pd\DD\,.
\]
Using the coordinate $\te$ to identify
the circle $\pd\DD$ with $\RR/2\pi\ZZ$, the latter circle diffeomorphism can be
naturally lifted to a diffeomorphism of the real line that we will
denote by $\bar\Pi:\RR\to\RR$. As is well known, a basic tool in the
study of circle diffeomorphisms is the {\em rotation number} (or {\em
  frequency}) of the
map, which is defined as
\begin{equation}\label{def-rot}
\om_\Pi:=\lim_{n\to\infty}\frac{\bar\Pi^n(\te_0)-\te_0}{ n}\,.
\end{equation}
Here $\bar\Pi^n$ denotes the $n$\th\ iterate of $\bar\Pi$ and $\te_0$ is
any real number. Since the Poincaré map of a flow is homotopic to the
identity, it is standard that the above limit exists and is independent of
the choice of $\te_0$~\cite{Yoccoz}. 

We shall next compute the rotation number of the circle diffeomorphism
$\Pi|_{\pd\DD}$ using the functional equation satisfied (up to
controllable errors) by the trajectories of the field $X$ on the
invariant torus. The reason is that, in order to compute the rotation
number to order $\cO(\ep^2)$, we need to iterate the Poincar\'e map an arbitrarily large
number of times, which requires fine control of the growth of the
errors for large times. 

The following theorem asserts that
the rotation number is given by the total torsion not only modulo
$\cO(\ep)$, as can be shown without relying on the functional equation, but also modulo~$\cO(\ep^2)$. The fact that the $\cO(\ep)$
correction is zero will be important later on.

\begin{theorem}\label{T.rotation}
The rotation number of the circle diffeomorphism $\Pi|_{\pd\DD}$ is
\[
\om_\Pi=\int_0^\ell \tau(\al)\,d\al+\cO(\ep^2)\,. 
\]
\end{theorem}
\begin{proof}
By the definition of the flow, Eq.~\eqref{def-rot} simply asserts that 
\[
\om_\Pi:=\lim_{n\to\infty}\frac{\te(n\ell)-\te_0}{ n}\,,
\]
where $\te(n\ell)\equiv \te(n\ell;1,\te_0)$ denotes the angular component of the
trajectory solving the system~\eqref{ODEs} with initial condition
$(0,1,\te_0)$, evaluated at time $n\ell$. Since the curvature $\ka(\al)$
and torsion $\tau(\al)$ are $\ell$-periodic, Proposition~\ref{L.trajectories2}
then ensures that
\begin{equation*}
\om_\Pi=\lim_{n\to\infty}\frac{\int_0^{n\ell}
  \tau\,ds -\frac{\ep\ka(0)}4[\sin\te(n\ell)-\sin\te_0\big]+n\ell\,
\cO(\ep^2)}{n} =\int_0^{\ell}
  \tau\,ds+\cO(\ep^2)\,.
\end{equation*}
\end{proof}

\subsection{The nondegeneracy condition for the Poincaré map}
\label{SS.nondegeneracy}

In this subsection we will compute a quantity associated with the
Poincaré map $\Pi$ (sometimes called the normal torsion of the
map) that was introduced to analyze the stability of individual invariant tori of
symplectic diffeomorphisms~\cite{Herman,Llave}. As we shall see, the assumption that
the normal torsion is nonzero plays a role that is analogous to the twist
condition in the classical theorem by Arnold and Moser on
perturbations of integrable symplectic maps. As the name can be
misleading, it is worth emphasizing that, in principle, the normal
torsion has nothing to do with the torsion of a curve.

Let us begin by introducing some notation. We will consider a domain
$\cD$ in the plane that contains the closed unit disk $\overline{\DD}$
and a map $\hPi: \cD\to\RR^2$. (Eventually, we will be interested in
taking as $\hPi$ the Poincar\'e map $\Pi$ introduced in the previous
subsection.) A closed curve
$\Ga\subset\cD$ is {\em invariant}\/ if its image $\hPi(\Ga)$ is
contained in $\Ga$. If $\Ga$ is an invariant curve of $\hPi$, one says
that $\hPi|_{\Ga}$ is {\em conjugate to a rotation} of frequency $\om$
through the diffeomorphism $\Theta:\SS^1\to \Ga$ if
\[
\Theta^{-1}\circ \hPi|_{\Ga}\circ\Theta(\vartheta)=\vartheta+\om\,,
\]
for all $\vartheta$ in $\SS^1$. When $\Ga=\pd\DD$, we will
abuse the notation and also denote by $\Theta$ the diffeomorphism
$\SS^1\to \SS^1$ corresponding to the angular component
of the above diffeomorphism $\SS^1\to\pd\DD$. (Therefore, in
the case of $\pd\DD$ the
above diffeomorphism will read as
$y=(\cos\Theta(\vartheta),\sin\Theta(\vartheta))$ in Cartesian
coordinates and $(r,\te)=(1,\Theta(\vartheta))$ in polar coordinates.)
From the context it will be clear which interpretation of $\Theta$
must be considered in each case.

\begin{definition}\label{D.cN}
  Let $\hPi:\overline{\DD}\to \overline{\DD}$ be a diffeomorphism of
  the disk that preserves the measure defined by the $2$-form
  $G(r,\te)\, r\, dr\wedge d\te$. We will denote the radial and
  angular components of $\hPi$ by $(\hPi_r,\hPi_\te)$,
  respectively. Assume that $\hPi|_{\pd\DD}$ is a diffeomorphism of the
  circle $\pd\DD\to\pd\DD$ that is conjugate to a rotation of
  frequency $\om$ through a diffeomorphism $\Theta$, which we regard
  here as a map $\SS^1\to\SS^1$. The {\em normal torsion}\/ of the map $\hPi$ on the
  invariant circle $\pd\DD$ is the real number
\[
\cN_{\hPi}:=\int_0^{2\pi}\frac1{\Theta'(\vartheta+\om)\ms \Theta'(\vartheta)}\frac{\pd_r\hPi_\te(1,\Theta(\vartheta))}{G(1,\Theta(\vartheta))}\,d\vartheta\,.
\]
\end{definition}

The reason why we consider the above quantity is that it appears in a
nondegeneracy condition of a theorem by de la Llave et
al.~\cite{Llave}. In fact, the result~\cite[Theorem 1]{Llave} is much
more general, and we will only need a concrete application that we
state next in a form that is particularly well suited for our
purposes. The normal torsion was also considered for the same purpose
by Herman in~\cite{Herman} when $G(r,\te)=1$ and
$\Theta(\vartheta)=\vartheta$. Before stating the theorem, let us
recall that a number $\om$ is {\em Diophantine}\/
if there exist a positive constant $C$ and $\nu>1$ such that
\begin{equation}\label{Dioph}
\Big|\frac\om{2\pi}-\frac pk\Big|\geq \frac C{k^{1+\nu}}
\end{equation}
for any integers $p,k$ with $k\geq1$.

\begin{theorem}[De la Llave et al.~\cite{Llave}]\label{T.Llave}
Consider a small neighborhood $\cD$ of the closed unit disk
$\overline{\DD}$ in $\RR^2$. Take an analytic map $\hPi: \cD\to\RR^2$
that is a diffeomorphism onto its image preserving the measure
$G(r,\te)\,r\, dr\, d\te$, with $G$ analytic in ${\cD}$. Suppose that the following two conditions hold:
\begin{enumerate}
\item The circle $\pd\DD$ is invariant, and $\hPi|_{\pd\DD}$ is
  conjugate through an analytic diffeomorphism $\Theta:
\SS^1\to \pd\DD$ to a rotation whose frequency $\om$ satisfies a Diophantine condition.

\item The normal torsion $\cN_{\hPi}$ of the map $\hPi$ on the invariant
  circle $\pd\DD$ is nonzero.
\end{enumerate}
Then for each $\de>0$ and positive integer $m$ there are $\de'>0$ and an
integer $k$ such that, if an analytic map $\tilde\Pi:\cD\to\RR^2$
preserving the same measure $G(r,\te)\,r\, dr\, d\te$  satisfies
\[
\|\hPi-\tilde\Pi\|_{C^k(\cD)}<\de'\,,
\]
one can transform the circle
$\pd\DD$ by a diffeomorphism $\Psi$ of $\RR^2$ so that $\Psi(\pd\DD)$
is an invariant curve for the map $\tilde\Pi$. Moreover, the map
$\tilde\Pi|_{\Psi(\pd\DD)}$ is also conjugate to a rotation of frequency
$\om$ and the difference $\Psi-\id$ can be assumed to be supported in
a small neighborhood of $\pd\DD$ and such that 
\[
\|\Psi-\id\|_{C^m}<\de\,.
\]
\end{theorem}
\begin{proof}
The statement is simply a rewording of~\cite[Theorem~47]{Llave}, in the
particular case of planar maps and omitting some quantitative
estimates that will not be needed in the rest of the paper. The only point
that requires more elaboration is to check what the degeneracy
condition looks like in the situation we are considering in this
section.

For the benefit of the reader, let us give some details about how the
statement is derived from~\cite[Theorem~47]{Llave}, borrowing some
notation from this reference without further mention. The map $\hPi$
is obviously symplectic with respect to the analytic $2$-form $\La:=
G(r,\te)\,r\, dr\wedge d\te$. This $2$-form is obviously exact, as $\La= d\be$ with
  the smooth $1$-form on $\cD$
\[
\be:=\bigg(\int_0^r G_\La(\bar r,\te)\, \bar r\, d\bar r\bigg) \, d\te\,.
\]
Moreover, the $1$-forms $\hPi^*\be-\be$ and $\tilde\Pi^*\be-\be$ are
exact because
\[
\int_{\pd\DD}(\hPi^*\be-\be)=\int_{\DD}(d\ms \hPi^*\be-d\be)=\int(\hPi^*\La-\La)=0\,,
\]
and analogously for $\tilde\Pi$.

As before, let us now regard $\Theta$ as a diffeomorphism of $\SS^1$.
Consider the embedding $K:\SS^1\to\cD$ given by 
\[
K(\vartheta):=\big(\!\cos\Theta(\vartheta),\sin\Theta(\vartheta)\big)\,,
\]
which is analytic by hypothesis.

Therefore, the only hypothesis of~\cite[Theorem~47]{Llave} that is not
immediate is the non-degeneracy condition. Let us take Cartesian components $(y_1,y_2)$ in $\cD$ and call
$(\hPi_1,\hPi_2)$ the Cartesian components of the map $\hPi$. In the aforementioned
reference, the condition is that the average of the function
\[
S(\vartheta):= P(\vartheta+\om)\cdot \big[ D\hPi(\vartheta)\ms J(\vartheta)^{-1}\, P(\vartheta)\big]
\]
be nonzero, where the dot denotes the Euclidean
scalar product. Here $P(\vartheta):=\Theta'(\vartheta)^{-1} (-\sin\Theta(\vartheta),\cos\Theta(\vartheta))$,
\[
(D\hPi)_{ij}(\vartheta):=\frac{\pd\hPi_i}{\pd y_j}(K(\vartheta))
\]
is the Jacobian matrix of $\hPi$ and
\[
J(\vartheta):=G(K(\vartheta)) \left(\begin{matrix}
  0 &-1\\
  1 &0
\end{matrix}\right)\,.
\]
Expressing the Cartesian components in polar coordinates,
\[
\hPi_1=\hPi_r\cos\hPi_\te\,,\qquad \hPi_2=\hPi_r\sin\hPi_\te\,,
\]
we immediately obtain that
\[
S(\vartheta)=\frac{\pd_r\hPi_r\sin(\hPi_\te-\Theta(\vartheta+\om))+
  \hPi_r\pd_r\hPi_\te\cos(\hPi_\te-\Theta(\vartheta+\om))}{ G\,\Theta'(\vartheta)\Theta'(\vartheta+\om)}\,,
\]
where the functions whose argument has not been specified are evaluated at
$K(\vartheta)$. Since $\hPi_r(K(\vartheta))=1$ because the
circle $\pd\DD$ is invariant and, by hypothesis, $\hPi|_{\pd\DD}$ is
conjugate to the rotation of frequency $\om$ through the
diffeomorphism $\Theta$ (i.e.,
\[
\hPi_\te(K(\vartheta))=\Theta(\vartheta+\om)
\]
for all $\vartheta$), we get
\[
S(\vartheta)=\frac1{\Theta'(\vartheta+\om)\ms \Theta'(\vartheta)}\frac{\pd_r\hPi_\te(K(\vartheta))}{G(K(\vartheta))}\,.
\]
In view of the way we defined the normal torsion (cf.\ Definition~\ref{D.cN}), the statement then
follows immediately from~\cite[Theorem~47]{Llave} after realizing that 
\[
\hPi(K(\vartheta))=K(\vartheta+\om)\,.
\]
\end{proof}

In order to calculate the normal torsion of the Poincar\'e map $\Pi$, let us begin by computing the diffeomorphism that
conjugates $\Pi|_{\pd\DD}$ to a rotation when its rotation number satisfies a
Diophantine condition:

\begin{proposition}
\label{P.Theta}
Suppose that the rotation number $\om_\Pi$ of the Poincar\'e map of
the local Beltrami field is Diophantine. Then the
circle diffeomorphism 
$\Pi|_{\pd\DD}$ is conjugate to a rotation of frequency $\om_\Pi$
through an analytic diffeomorphism $\Theta:\SS^1\to\pd\DD$ that
satisfies
\[
\Theta(\vartheta)=\vartheta-\frac{\ep\ka(0)}{4}\sin\vartheta+\cO(\ep^2)\,.
\]
\end{proposition}
\begin{proof}
As the rotation number $\om_\Pi$ satisfies a
  Diophantine condition, the map $\Pi|_{\pd\DD}$ is conjugate to
  a rotation of frequency $\om_\Pi$ through an analytic diffeomorphism
  $\Theta$~\cite[Theorem 1.3]{Yoccoz}. This diffeomorphism can be
  understood as a change of coordinates $\te_0=\Theta(\vartheta)$.

Let us now compute the diffeomorphism~$\Theta$. Writing $\te_0=\Theta(\vartheta)$, the fact
that $\hPi|_{\pd\DD}$ is conjugate to a rotation of frequency $\om_\Pi$
through $\Theta$ means that the trajectory
$\te(\ell;1,\Theta(\vartheta))$ at time $\ell$ corresponds to
$\Theta(\vartheta+\om_\Pi)$, for any choice of $\vartheta$. Using the
equation for the trajectory proved in Proposition~\ref{L.trajectories2} at
time $\ell$,
this means that $\Theta$ must satisfy the equation
\begin{equation}\label{eqTheta}
\Theta(\vartheta+\om_\Pi)=\Theta(\vartheta)+\om_\Pi+\frac{\ep\ka(0)}4\big[\sin\Theta(\vartheta)-\sin\Theta(\vartheta+\om_\Pi)\big]+\cO(\ep^2)\,.
\end{equation}
Here we have used the expression for $\om_\Pi$ proved in Theorem~\ref{T.rotation}
(which allows us to replace $\int_0^\ell\tau\,
d\bs=\om_\Pi+\cO(\ep^2)$) and that the function $\ka$ is
$\ell$-periodic.

To analyze this equation, let us write the $\cO(\ep^2)$ term in
the RHS of~\eqref{eqTheta} as~$R(\vartheta)$. Let us identify $\pd\DD$ with $\SS^1=\RR/2\pi\ZZ$
through the angular coordinate $\te_0$, so that $\Theta$ is regarded
as a diffeomorphism $\SS^1\to\SS^1$. With a slight abuse
of notation, let us still denote by $\Theta$ its lift $\RR\to\RR$. It is standard
that this lift can be written as
\begin{equation}\label{eqTe1}
\Theta(\vartheta)=\vartheta+H(\vartheta)\,,
\end{equation}
where $H(\vartheta)$ is an analytic $2\pi$-periodic function.

Let  us consider the $2\pi$-periodic function
\begin{equation}\label{eqTe2}
F(\vartheta):=H(\vartheta)+\frac{\ep\ka(0)}4 \sin\Theta(\vartheta)\,.
\end{equation}
Eq.~\eqref{eqTheta} then reads as
\begin{equation}\label{eqF}
F(\vartheta+\om_\Pi)= F(\vartheta)+ R(\vartheta)\,,
\end{equation}
with $R=\cO(\ep^2)$. Consider the Fourier series of the functions $ F$ and $R$:
\[
F(\vartheta)=\sum_{k=-\infty}^\infty \hat F_k\,\e^{\I
  k\vartheta}\,,\qquad R(\vartheta)=\sum_{k=-\infty}^\infty \hat R_k\,
\e^{\I k\vartheta}\,.
\]
Eq.~\eqref{eqF} then asserts that for any nonzero integer $k$ the
Fourier coefficients of $F$ and $R$ are related through the identity
\begin{equation}\label{hatF}
\hat F_k=\frac{\hat R_k}{\e^{\I k\om_\Pi}-1}\,.
\end{equation}
We can obviously take $\hat F_0=0$; moreover, $\hat R_0=0$ because it
is a necessary condition for the existence of the diffeomorphism
$\Theta$.

Since $\om_\Pi$ satisfies the Diophantine condition~\eqref{Dioph}, for large integer
values of $k$ we have the elementary inequality
\[
|\e^{\I k\om_\Pi}-1|>C|k|^{-\nu}\,,
\]
so that from Eq.~\eqref{hatF} the $H^m$ norm of $F$ can be estimated by
\[
\|F\|_{H^m}\leq C\bigg(\sum_{k=-\infty}^\infty(1+k^{2})^m |k|^{2\nu}|\hat
R_k|^2\bigg)^{\frac12} \leq C\|R\|_{H^{m+\nu}}\leq C_m\ep^2
\]
for any nonnegative  integer $m$. To derive the last inequality, which shows that $F=\cO(\ep^2)$, we
have used that $R=\cO(\ep^2)$. In view of Eqs.~\eqref{eqTe1}
and~\eqref{eqTe2}, this ensures that
\[
\Theta(\vartheta)=\vartheta-\frac{\ep\ka(0)}4\sin\Theta(\vartheta)+\cO(\ep^2)\,.
\]
In turn, this readily leads to the expression for the diffeomorphism $\Theta$ provided in the statement.
\end{proof}

We are ready to provide a closed formula for the normal torsion of the
Poincar\'e map, up to terms of order $\cO(\ep^3)$. The leading term only depends on the geometry of the curve (through
its curvature and torsion) and, as is to be expected, not on the section of $\es\times\DD$ we used to
define the Poincar\'e map:

\begin{theorem}\label{T.Herman}
Suppose that the rotation number $\om_\Pi$ satisfies a Diophantine
condition. Then the normal torsion of the Poincar\'e map of the local
Beltrami field $v$ on
the invariant circle $\pd\DD$  is
\[
\cN_\Pi=-\frac{5\pi\ep^2}8\int_0^\ell\ka(\al)^2\,\tau(\al)\, d\al +\cO(\ep^3)\,.
\]
\end{theorem}
\begin{proof}
By definition, the normal torsion is
\[
\cN_{\Pi}:=\int_0^{2\pi}\frac1{\Theta'(\vartheta+\om_\Pi)\ms \Theta'(\vartheta)}\frac{\pd_r\Pi_\te(1,\Theta(\vartheta))}{G_\La(1,\Theta(\vartheta))}\,d\vartheta\,,
\]
where $\Pi_\te(r,\te)$ is the angular component of the Poincar\'e map,
$\Theta(\vartheta)$ is the diffeomorphism defined in
Proposition~\ref{P.Theta} (considered as a map $\SS^1\to \SS^1$)
and the function $G_\La(r,\te)$ is given by Eq.~\eqref{defF}. 

We have already computed all the terms we need to evaluate the
integrand up to an $\cO(\ep^3)$ error. Indeed, from
Propositions~\ref{L.trajectories}, \ref{P.area}
and~\ref{P.Theta} and Theorem~\ref{T.rotation} it stems that, setting
$T:=\int_0^\ell\tau(\al)\, d\al$, 
\begin{align*}
\om_\Pi&=T+\cO(\ep^2)\,,\\
\Theta'(\vartheta)&=1-\frac{\ep \ka(0)}4\cos\vartheta+\cO(\ep^2)\,,\\
\Theta'(\vartheta+\om_\Pi)&=1-\frac{\ep \ka(0)}4\cos(\vartheta+T)+\cO(\ep^2)\,,\\
G_\La(1,\Theta(\vartheta))&= 1+\ep\ka(0)\cos\vartheta+\cO(\ep^2)\\
\frac{\pd\Pi^\te}{\pd r}(1,\Theta(\vartheta))&= \ep\frac{\pd
  \te^{(1)}}{\pd r_0}\Big(\ell;1,\vartheta-\frac{\ep \ka(0)}4\sin\vartheta\Big)+\ep^2\frac{\pd\te^{(2)}}{\pd r_0}(\ell;1,\vartheta)+\cO(\ep^3)\,,
\end{align*}
the functions $\te^{(j)}(s;r_0,\te_0)$ being in
Proposition~\ref{L.trajectories}. Plugging these expressions in the integral
for the normal torsion and using trigonometric identities, one arrives
at the expression
\begin{align*}
\cN_\Pi&=\int_0^{2\pi}\bigg[\frac{\ep\ka(0)}2\big(\sin(\vartheta+T)-\sin\vartheta\big)-\frac{5\ep^2}{16}\int_0^{2\pi}\ka(\al)^2\tau(\al)\,
d\al\\
&\qquad \qquad \qquad \qquad 
\qquad\qquad+\frac{5\ka^2\ep^2}{48}\sin
T\cos(2\vartheta+T)+\cO(\ep^3)\bigg]d\vartheta\\
&=-\frac{5\pi\ep^2}8 \int_0^{2\pi}\ka^2\tau\, d\al+\cO(\ep^3)\,,
\end{align*}
as claimed.
\end{proof}

\subsection{A KAM theorem for generic tubes}
\label{SS.KAM}

In this subsection we will use the previous results to prove a theorem
on the preservation of invariant tori for divergence-free vector
fields that are close to the local Beltrami field $v$ for small enough
$\ep$. As we shall see, the hypotheses of the KAM theorem will hold
true as long as the core curve $\ga$ of the tube satisfies certain
generic geometric conditions. Details on the validity
of these conditions are given below.

Since the local Beltrami field $v$ is analytic in a neighborhood of the
closure of the tube, which we identify with $\es\times\DD$ via the coordinates~$(\al,y)$, we can assume that $v$
is defined in some domain $\es\times\cD$, with $\cD$ a neighborhood of
the closed unit disk in the plane. To measure the smallness of a field, we will use the
norm $\|u\|_{C^k(\es\times\cD)}$, which we define in terms of its
components in the coordinates $(\al,y)$ as
\[
\|u\|_{C^k(\es\times\cD)}:=\|u_\al\|_{C^k(\es\times\cD)}+\|u_y\|_{C^k(\es\times\cD)}\,.
\]

We will sometimes find it convenient to refer to a
domain bounded by an invariant torus
as an {\em invariant tube} of the field. We say that a field $u$ is {\em orbitally
  conjugate to a rotation}\/ of frequency $\om$ on an invariant torus $\Si$ if there are
global coordinates $(\tilde\al,\tilde\te):\Si \to \SS^1\times\SS^1$ in
which the vector field is linear with frequency $\om$ up to a multiplicative factor, that is,
\[
u|_{\Si}=F(\tilde\al,\tilde\te)\,\big(\pd_{\tilde\al}+\om\,\pd_{\tilde\te}\big)\,,
\]
where $F:\SS^1\times\SS^1\to\RR$ is a nonvanishing
function.

In the following lemma we will show that, for a generic core curve
$\ga$, the rotation number of the Poincar\'e map of the local Beltrami
field is Diophantine and its normal torsion is nonzero. To make
precise what we understand by ``generic'', we will say that certain property holds for a {\em $C^m$-dense set of
  closed analytic curves} if, given any closed analytic curve $\ga$ in
$\RR^3$, one can deform it by a diffeomorphism $\Phi$ of $\RR^3$, with
$\|\Phi-\id\|_{C^m(\RR^3)}$ as small as one wishes, so that the curve
$\Phi(\ga)$ has the desired property. Notice that this does not imply
that the property holds for an open set of curves.

\begin{lemma}\label{L.generic}
Let $m$ be any positive integer.  The set of closed analytic curves for which the Poincar\'e map of
  the local Beltrami field $v$ has a Diophantine rotation number $\om_\Pi$ and nonzero normal torsion
  $\cN_\Pi$ on the invariant circle $\pd\DD$ is $C^m$-dense.
\end{lemma}
\begin{proof}
The result is not hard to prove using the expressions for the
rotation number and the normal torsion derived in
Theorems~\ref{T.rotation} and~\ref{T.Herman}:
\begin{align*}
\om_\Pi&=\int_0^\ell \tau\,d\al+\cO(\ep^2)\,,  \\
\cN_\Pi&=-\frac{5\pi\ep^2}8\int_0^\ell\ka^2\tau\, d\al +\cO(\ep^3)\,.
\end{align*}
A way of making things precise is the following. We will consider
deformations of the curve $\ga$, labeled by a parameter $\de$. More
concretely, let us denote by $(T(\al),N(\al),B(\al))$ the Frenet
trihedron of the curve $\ga$ at the point $\ga(\al)$ (one should not mistake the
binormal vector $B(\al)$ for the function $B$ that we introduced in
Eq.~\eqref{B}, which will not be used in this proof). Let $F(\al)$
be an analytic $\ell$-periodic function and consider the family of curves
$\Ga(\al,\de)$ in $\RR^3$ given by
\[
\Ga(\al,\de):=\ga(\al)+\de\, F(\al) B(\al)\,.
\]
If $\de$ is close to zero, $\ga_\de\equiv \Ga(\cdot,\de)$ is a closed analytic
curve. Notice that, for $\de\neq0$, $\al$ is no longer an arc-length
parametrization of $\ga_\de$ but, due to the properties of the
binormal field,
\[
|\Ga'|=1+\cO(\de^2)\,.
\]
Here and in what follows, we denote by a prime the derivatives with
respect to $\al$. In particular, the length of $\ga_\de$ is
$\ell+\cO(\de^2)$. We will label the geometric quantities associated
with the curve $\ga_\de$ with a subscript $\de$ (e.g.,
$\ka_\de$ and $\tau_\de$ for its curvature and torsion).

The results we have presented in this section carry over immediately
when one does not only consider the tube $\tube$ associated with the
curve $\ga$, but the family of tubes $\cT_\ep(\ga_\de)$. The
dependence of the various quantities on the small parameter $\de$ is smooth
and can be controlled easily. A tedious but
straightforward computation using the well-known formulas for the
curvature and torsion of a parametrized curve shows that
\begin{align*}
\ka_\de^2&=\ka^2+2\de\ka(2 F'\tau+F\tau')+\cO(\de^2)\,,\\
\tau_\de&=\tau-\de\bigg[\ka F'+\frac{d}{d\al}\bigg(\frac{
  F''-\tau^2 F}{\ka}\bigg)\bigg]+\cO(\de^2)\,.
\end{align*}
Since the length of $\ga_\de$ differs from $\ell$ by an $\cO(\de^2)$
term, one readily finds that the rotation number of the Poincar\'e map~$\Pi_\de$
associated with the harmonic field of the tube $\cT_\ep(\ga_\de)$ is
\begin{align*}
\om_{\Pi_\de}&=\int_0^\ell \tau_\de(\al)\,d\al+\cO(\ep^2+\de^2)\\
&=
\om_\Pi+\de\int_0^\ell  \ka'(\al)\, F(\al)\, d\al+\cO(\ep^2+\de^2)\,.
\end{align*}
Similarly, the dependence of the normal torsion on $\de$
can be shown to be
\begin{equation*}
\cN_{\Pi_\de}= \cN_\Pi-\frac{5\pi\de\ep^2}8\int_0^\ell
(2\ka'''+3\ka^2\ka'-6\ka\tau\tau'-6\tau^2\ka')F\, d\al+\cO(\de^2+\ep^3)\,.
\end{equation*}
We recall that, although the Poincar\'e map $\Pi_\de$ depends smoothly
on the parameter~$\de$, the terms $\cO(\ep^2+\de^2)$ and $\cO(\de^2+\ep^3)$ are
continuous, but possibly not differentiable, functions of~$\de$.

Perturbing the curve $\ga$ a little if necessary to ensure that the
functions $\ka'$ and 
\[
2\ka'''+3\ka^2\ka'-6\ka\tau\tau'-6\tau^2\ka'
\]
are not
identically zero, we deduce that the function $F$ can be chosen so
that the integrals
\[
\int_0^\ell  \ka'\, F\, d\al \qquad\text{and}\qquad 
\int_0^\ell
(2\ka'''+3\ka^2\ka'-6\ka\tau\tau'-6\tau^2\ka')\,F\, d\al
\]
are
nonzero. For small enough $\ep$ this ensures that, as $\de$ takes
values in a small enough interval $(-\de_0,\de_0)$, the values taken
by the continuous function $\om_{\Pi_\de}$ (resp.\ $\cN_{\Pi_\de}$)
cover an interval centered at $\om_\Pi$ (resp.\ $\cN_\Pi$) of radius
$C\de_0$ (resp.\ $C\ep^2\de_0$). Since Diophantine numbers have full Lebesgue measure, this
immediately implies that one can choose an arbitrarily small $\de$
such that the rotation number $\om_{\Pi_\de}$ satisfies a Diophantine
condition and the normal torsion $\cN_{\Pi_\de}$ is nonzero.
\end{proof}

We can now show that, for a generic core curve $\ga$ and small enough $\ep$, a suitably small
perturbation of the local Beltrami field still has an invariant torus
that is close to the original one:

\begin{theorem}\label{T.KAM}
  For any positive integer $m$ there is a $C^m$-dense set of closed
  analytic curves $\ga$ with the following KAM-type property: for any
  $\de>0$, there is another positive integer $k$ and some $\de'>0$
  such that any analytic divergence-free vector field $u$ whose difference with
  the local Beltrami field $v$ of the tube $\cT_\ep\equiv\cT_\ep(\ga)$
  is bounded by
\begin{equation}\label{hwdep}
\|u-v\|_{C^k({\es\times\cD})}<\de'
\end{equation}
possesses an invariant tube. Furthermore,
one can find a
diffeomorphism $\Psi$ of~$\RR^3$, with $\|\Psi-\id\|_{C^m(\RR^3)}<
\de$ and $\Psi-\id$ supported in a small neighborhood of the torus~$\pd\cT_\ep$, such that
$\Psi(\cT_\ep)$ is an invariant tube of $u$ and $u$ is orbitally
conjugate on the invariant torus $\pd\Psi(\cT_\ep)$  to a Diophantine rotation.
\end{theorem}

\begin{proof}
By Lemma~\ref{L.generic}, we can deform the curve that lies at the
core of the tube $\cT_\ep$ by a
diffeomorphism of $\RR^3$, arbitrarily close to the identity in the
$C^m$ norm, so that, if we consider the local Beltrami field (which we still denote by $v$) associated
to the deformed tube, its rotation number $\om_\Pi$ satisfies a Diophantine
  condition and its normal torsion $\cN_\Pi$ is nonzero. As before, we
  will identify this deformed tube with the domain $\es\times\DD$
  through adapted coordinates $(\al,y)$.

  Consider the Poincar\'e map $\Pi$ of the local Beltrami field $v$,
  which can be safely considered as a diffeomorphism $\Pi$ from $\cD$
  onto its image, with $\cD$ a neighborhood of the closed unit disk $\overline{\DD}$.  The Poincar\'e map of the field $u$,
  also defined on the section $\{\al=0\}$, is another diffeomorphism
  of $\cD$ onto its image that we will denote by $\tilde\Pi$. The
  vector fields $u$ and $v$ being close by~\eqref{hwdep}, it is apparent that
\begin{equation}\label{PiPdep}
\|\Pi-\tilde\Pi\|_{C^k(\cD)}<C\ms \de'
\end{equation}
as the Poincar\'e map is simply obtained by integrating the associated
vector field along a trajectory between two consecutive intersections
with the section $\{\al=0\}$. To avoid cumbersome notations related to the intersection of the
domains of auxiliary maps, wherever appropriate we will assume that the
Poincar\'e maps are defined in a domain slightly larger than $\cD$
without further mention. Obviously there is no loss of generality in
this assumption.

We have seen in Proposition~\ref{P.area} that the Poincar\'e map $\Pi$
preserves the  $2$-form $\La=G_\La(r,\te)\, r\,dr\wedge d\te$, where
\[
G_\La(r,\te):=Bv_\al\big|_{\al=0}=1+\cO(\ep)\,,
\]
and $v_\al$ denotes the $\al$-component of the local Beltrami field $v$.
Mimicking the proof of this proposition, we immediately obtain that the
Poincar\'e map $\tilde\Pi$ of the divergence-free field $u$ preserves  the
$2$-form $\tilde\La:=\tG(r,\te)\, r\, dr\wedge d\te$, with
\[
\tG(r,\te):=Bu_\al\big|_{\al=0}\,.
\]
Notice that $u_\al$ does not vanish in $\{0\}\times \cD$ because the difference $|u_\al-v_\al|$
is small and the $\al$-component of the local Beltrami field is close to $1$.

Our next goal is to relate the above invariant $2$-forms to apply Theorem~\ref{T.Llave}. More concretely, we will show that there is a $C^m$-small diffeomorphism
$\Phi$ such that
\begin{equation}\label{LahLa}
\La=\Phi^*\ms \tilde\La\,,
\end{equation}
where $\Phi^*$ is the pullback of the diffeomorphism.
This will be done using Moser's trick. We start by noticing that the difference between these
$2$-forms is obviously exact, as
\[
\La-\tilde\La=d\ms \Ga
\]
with the $1$-form $\Ga$ given by
\[
\Ga:=\bigg(\int_0^r \big[G_\La(\bar r,\te)-\tG(\bar r,\te)\big]\, \bar r\, d \bar r\bigg)\, d\te\,.
\]
Although we are making computations in polar coordinates, it is
readily seen that all the objects we are considering are well-defined
also at the origin, and therefore determine smooth forms in the whole
disk $\cD$.

Consider the non-autonomous vector field of class $C^\infty$
\[
Z_s:=\frac{\int_0^r [G_\La(\bar r,\te)-\tG(\bar r,\te)]\, \bar r\, d \bar r}{r[(1-s)\ms
  G_\La(r,\te)+s\ms \tG(r,\te)]}\, \pd_r\,,
\]
where $s$ will be the flow parameter. It should be noticed that, by
the assumptions on the vector fields, the denominator behaves as
$r+\cO(r^2)$ while the numerator is of order $\cO(r^2)$. The field $Z_s$ satisfies the $C^m$
bound
\begin{equation}\label{rollo2}
\|Z_s\|_{C^m(\cD)}< C\de'
\end{equation}
for all $s\in[0,1]$ as a consequence of the estimate~\eqref{hwdep}.

The time-$s$ flow of the non-autonomous field $Z_s$, which will be
denoted by $\phi_{s}$, is given by the solution to the initial value
problem
\[
\frac{\pd}{\pd s}\phi_{s}x=Z_s(\phi_{s}x)\,,\qquad \phi_{0}x=x\,.
\]
Consider the $s$-dependent $2$-form
\[
\La_s:=(1-s)\ms \La+s\ms\tilde\La\,.
\]
A simple computation shows that
\begin{equation}\label{rollo1}
\frac{\pd}{\pd
  s}\big(\phi_{s}^*\La_s\big)=\phi_{s}^*\bigg(\frac\pd{\pd
  s}\La_s+L_{Z_s}\La_s\bigg)=\phi_{s}^*\big( \tilde\La-\La+d\ms i_{Z_s}\La_s\big)=0\,,
\end{equation}
where $L_{Z_s}$ denotes the Lie derivative along $Z_s$ and the last equality follows immediately from the definition of
$\La_s$ and the fact that the interior product of $Z_s$ with $\La_s$
is
\[
i_{Z_s}\La_s=\Ga\,.
\]
Therefore, if we set $\Phi:=\phi_{1}$, we obtain Eq.~\eqref{LahLa}
from~\eqref{rollo1} and the definition of $\La_s$. Moreover,
\begin{equation}\label{CmTe}
\|\Phi-\id\|_{C^m(\cD)}<C\ms \de'
\end{equation}
because $\Phi$ is the time-$1$
flow of the vector field $Z_s$, whose $C^m$ norm is controlled by Eq.~\eqref{rollo2}.

Let us now consider the map
\[
\hat \Pi:=\Phi^{-1}\circ \tilde\Pi\circ \Phi\,,
\]
which is a diffeomorphism from a neighborhood of $\overline{\DD}$
(which we still take as $\cD$) onto its image. We shall next relate the
new map $\hat \Pi$ to the Poincar\'e map $\Pi$. By the definition of $\Phi$
and the relation between the invariant $2$-forms~\eqref{LahLa}, the
map $\hat \Pi$ preserves the same $2$-form $\La$ as the Poincar\'e
map~$\Pi$, and $\hat \Pi$ is close to $\Pi$ by~\eqref{CmTe}:
\begin{equation}\label{CmPi}
\big\|\Pi-\hat \Pi\big\|_{C^m(\cD)}< C\de'\,.
\end{equation}

We are now ready to apply Theorem~\ref{T.Llave} with the maps $\Pi$ and $\hat
\Pi$. Indeed, the previous arguments and the way we have deformed the
curve $\ga$ ensure that the following statements hold true:
\begin{enumerate}
\item The rotation number $\om_\Pi$ satisfies a
  Diophantine condition, so the map $\Pi|_{\pd\DD}$ is analytically conjugate to
  a rotation of frequency $\om_\Pi$~(cf.\ Proposition~\ref{P.Theta}).

\item The normal torsion $\cN_\Pi$ is nonzero.

\item The $C^k$ norm of $\Pi-\hat \Pi$ is at most $C\de'$ by
  Eq.~\eqref{CmPi}, with $\de'$ arbitrarily small.
\end{enumerate}
Hence Theorem~\ref{T.Llave} ensures that, for any given integer $m$, if the integer $k$ is large
enough and $\de'$ is sufficiently small  there is a diffeomorphism
$\hat\Psi$ of $\cD$ with $\hat\Psi-\id$ arbitrarily
small in the ${C^m(\cD)}$ norm and supported in a neighborhood of
$\pd\DD$ such that the curve $\hat\Psi(\pd\DD)$ is invariant under the
map $\hat \Pi$. Furthermore,
the restriction of $\hat \Pi$ to this invariant curve, $\hat
\Pi|_{\hat\Psi(\pd\DD)}$, also has rotation number $\om_\Pi$. 

The definition of $\hat \Pi$ then ensures that the curve
$\Phi\circ\hat\Psi(\pd\DD)$ is invariant under the Poincar\'e map
$\tilde\Pi$ of the field $u$. (As an aside, notice that this invariant curve
is close to $\pd\DD$ but, in principle, is not contained in the
closure of $\DD$, which is the reason why we are considering a
slightly larger disk $\cD$ throughout the proof.) 
It is standard that this is equivalent to saying that there
is an invariant torus of~$u$ whose intersection with the disk~$\{\al=0\}$ is precisely the aforementioned curve. Since the $C^m(\cD)$ norm of
the diffeomorphism $\Phi\circ\hat\Psi$ is arbitrarily small by
the properties of $\hat\Psi$ and Eq.~\eqref{CmTe}, one can take a diffeomorphism~$\bar\Psi$ of $\es\times\cD$
such that $\bar\Psi(\es\times\pd\DD)$ is an invariant torus of the field $u$ and
$\bar\Psi-\id$ is small in the $C^m$~norm and is supported in a
neighborhood of $\es\times\DD$. Indeed, this diffeomorphism can be
defined as follows. Take the solution to the system of ODEs
\begin{align*}
\frac{d r}{d\al}&= \frac{u_r(\al,r,\te)}{u_\al(\al,r,\te)}\,,\\
\frac{d\te}{d\al}&= \frac{u_\te(\al,r,\te)}{u_\al(\al,r,\te)}\,,
\end{align*}
with the trajectory $(r,\te)$ parametrized by the angle $\al$ and
depending on the initial conditions $(r_0,\te_0)$. Consider the
function $\vp_\al$ mapping the initial conditions
$(r_0,\te_0)$ to its time-$\al$ flow
$(r(\al;r_0,\te_0),\te(\al;r_0,\te_0))$. Then it is easy to
check that the diffeomorphism $\bar\Psi$ is given, in polar coordinates, by
\[
\bar\Psi(\al,r,\te):=\big(\al,\vp_{\al}\circ\Phi\circ\hat\Psi(r,\te)\big)\,.
\]
Actually, this formula simply asserts that the intersection of the
invariant torus of the field $u$ with each section $\{\al=\al_0\}$,
understood as a curve in the disk $\cD$, is the image under
$\vp_{\al_0}$ of the invariant curve $\Phi\circ\hat\Psi(\pd\DD)$ at $\{\al=0\}$.
Clearly the formula for $\bar\Psi$ and the estimates for the maps $\Phi$
and $\hat\Psi$ imply that
\[
\|\bar\Psi-\id\|_{C^m(\es\times\cD)}\leq C\|\Phi\circ\hat\Psi-\id\|_{C^m(\cD)}
\]
can be made arbitrarily small.

To complete the proof of the theorem, it suffices to recall that the
map $\hat\Pi|_{\hat\Psi(\pd\DD)}$ has rotation number $\om_\Pi$. This
implies that the Poincar\'e map $\tilde\Pi$ of the vector field~$u$, restricted to the invariant curve
$\Phi\circ\hat\Psi(\pd\DD)$, also has rotation number~$\om_\Pi$,
which trivially implies that the vector field $u$ itself is
orbitally conjugate to a rotation of frequency $\om_\Pi$ on the
invariant torus $\bar\Psi(\es\times\pd\DD)$. The existence of the diffeomorphism~$\Psi$
of~$\RR^3$ that appears in the statement of the theorem is then immediate.
\end{proof}

\begin{remark}
Ultimately, Theorem~\ref{T.KAM} is a
result on the preservation of invariant tori for small perturbations
of the harmonic field $h$, rather than of the local Beltrami field
$v$. Indeed, if we had considered a Beltrami field with small but
otherwise arbitrary parameter $\la$ (possibly $0$), we would have
found the same expressions for the rotation number and normal torsion
of its Poincar\'e map, the only change being that the error terms
would be $\cO(\ep^2+|\la|)$ and $\cO(\ep^3+|\la|)$, respectively.
\end{remark}

We will conclude this section with a result on the
persistence of another invariant set: we will show that, for an
open and dense set of core curves $\ga$ and small enough $\ep$, any
divergence-free vector field that is a small perturbation of the local
Beltrami field $v$ has an elliptic periodic trajectory close to the
core curve $\ga$. We recall that a periodic trajectory is {\em
  elliptic}\/ if the nontrivial eigenvalues of the associated
monodromy matrix have all unit modulus but are different from $1$.

\begin{proposition}\label{P.harmelliptic}
Suppose that the total torsion of the curve $\ga$ satisfies
\[
\int_0^\ell \tau(\al)\, d\al\neq n\pi
\]
for all integers $n$. Then for any $\de>0$,
there exists some $\de'>0$ such that any divergence-free vector field
$u$ in the tube $\cT_\ep$ which is close to the local Beltrami field~$v$ in the sense that
\[
\|u-v\|_{C^k({\es\times\DD})}<\de'\,.
\]
also has an elliptic periodic trajectory diffeomorphic to the curve
$\{y=0\}$. Moreover, the corresponding diffeomorphism $\Psi$ is bounded by
\[
\|\Psi-\id\|_{C^k(\es\times\DD)}<\de
\]
and is different from the identity only in a small neighborhood of the
curve $\{y=0\}$.
\end{proposition}

\begin{proof}
From the expressions~\eqref{h0}--\eqref{nabla} for the harmonic field~$h$, the estimates
for the function~$\psi$ proved in Theorem~\ref{T.psi0} and the
connection between the local Beltrami field and $h$
(see Eq.~\eqref{vhep3}), we infer
that the difference $v-\tilde h_0$ can be estimated as
\begin{equation*}
\|v-\tilde h_0\|_{C^k(\es\times\DD)}< C_k\ep\,,
\end{equation*}
with the vector field $\tilde h_0$ defined as
\[
\tilde h_0:=\pd_\al+\tau(\al)\,(y_1\,\pd_2-y_2\,\pd_1)\,.
\]
It is clear that $(\al(s),y(s))=(s,0)$ is an $\ell$-periodic trajectory of the
field $\tilde h_0$. Setting 
\[
T:=\int_0^\ell\tau(\al)\, d\al\,,
\]
an easy computation shows that the monodromy
matrix of this trajectory for the field $\tilde h_0$ is
\[
\left(\begin{matrix}
  1 & 0 &0\\
  0 & \cos T &-\sin T\\
  \phantom{-} 0 & \sin T & \phantom{-}\cos T
\end{matrix}\right)\,.
\]
The nontrivial eigenvalues of this matrix are $\e^{\pm\I T}$, and hence
different from $\pm1$ by the hypotheses of the theorem, thus showing
that $\{y=0\}$ is an elliptic trajectory of $\tilde
h_0$. 

It is then standard that any divergence-free field $u$ that is close enough to
$\tilde h_0$ (say, $\|u-\tilde h_0\|_{C^k(\es\times\DD)}<\de_1)$) has
an elliptic periodic trajectory given by the image of $\{y=0\}$ under
a diffeomorphism $\Psi$ with
$\|\Psi-\id\|_{C^k(\es\times\DD)}<\de$. There is no loss of generality
in assuming that $\Psi-\id$ is supported in a small neighborhood of
$\{y=0\}$. Since
\[
\|u-\tilde h_0\|_{C^k(\es\times\DD)}< \|u- v\|_{C^k(\es\times\DD)}+ \|
v-\tilde h_0\|_{C^k(\es\times\DD)}<\de'+ C_k\ep\,,
\]
the theorem then follows by
taking $\ep$ and $\de'$ small enough.
\end{proof}

\section{Approximation by Beltrami fields with decay}
\label{S.approx}

In this section we prove a result that allows us to approximate a
field $v$ that satisfies the Beltrami equation 
\[
\curl v=\la v
\]
on a neighborhood of a compact set
$S$, by a global Beltrami field $u$, which satisfies
\[
\curl u=\la u
\] 
in the whole space $\RR^3$ and falls off at infinity as
$1/|x|$. Throughout we will assume that the complement $\RR^3\minus S$
is a connected set. It is not hard to see that this condition is necessary.

It will be more convenient for us to work with an auxiliary elliptic equation
instead of considering the Beltrami equation directly. To this end, let us denote by
\[
G(x):=\frac{\cos\la |x|}{4\pi|x|}
\]
the Green's function of the operator $\De+\la^2$ in $\RR^3$, which
satisfies the distributional equation
\[
\De G+\la^2 G=-\de_0
\]
with $\de_0$ the Dirac measure supported at $0$. We will use the
notation $\BB_R$ for the ball in $\RR^3$ centered at the origin and of
radius $R$.

The following lemma, which shows how to ``sweep'' the singularities of
the Green's function, will be used in the demonstration of the global
approximation theorem. Its proof is based on a duality argument and the
Hahn--Banach theorem.

\begin{lemma}\label{L.approx}
Take $R>0$ and consider a domain $U\subset
\RR^3\minus \BB_{2R}$ and a compact set $S\subset \BB_R$ whose
complement $\RR^3\minus S$ is connected. Let us consider the vector
field
\[
v(x):=\sum_{m=1}^M\rho_m\, G(x-x_m)\,,
\]
where $\{x_m\}_{m=1}^M$ is a finite set of points in $\BB_R\minus S$ and $\rho_m\in\RR^3$
are constant vectors. Then, for any $\de>0$, there is a finite set of points
$\{z_j\}_{j=1}^J$ in the domain $U$ and constant vectors $c_j\in\RR^3$ such that the
finite linear combination
\begin{equation}\label{eqmwx}
w(x):=\sum_{j=1}^J c_j\, G(x-z_j)
\end{equation}
approximates the field $v$ uniformly in $S$ as
\begin{equation*}
\|v-w\|_{C^0(S)}<\de\,.
\end{equation*}
\end{lemma}
\begin{proof}
Consider the space $\cU$ of all vector fields that are linear
combinations of the form~\eqref{eqmwx}, where the points $z_j$ belong
to the set $U$ and the coefficients $c_j\in\RR^3$ are constant
vectors. Restricting these fields to the set $S$, $\cU$ can
be regarded as a subspace of the Banach space $C^0(S,\RR^3)$ of continuous
vector fields on $S$.

By the Riesz--Markov theorem, the dual of $C^0(S,\RR^3)$ is  the space
$\cM(S,\RR^3)$ of the finite vector-valued Borel measures on $\RR^3$ whose support
is contained in the set~$S$. Let us take any measure
$\mu\in\cM(S,\RR^3)$ such that $\int_{\RR^3} f\cdot d\mu=0$ for all
$f\in \cU$. Let us now define a field $F\in L^1\loc(\RR^3,\RR^3)$ as
\[
F(x):=\int_{\RR^3} G(x-\tilde x)\,d\mu(\tilde x)\,,
\]
so that $F$ satisfies the equation 
\[
\De F+\la^2F=-\mu\,.
\]
Notice that $F$ is identically zero on the open set $U$ by the definition of the
measure~$\mu$, that $\RR^3\minus S$ is connected and that $F$
satisfies the elliptic
equation
\[
\De F+\la^2 F=0
\]
in $\RR^3\minus S$. Hence the unique continuation
theorem ensures that the field $F$ vanishes on the complement of $S$. It
then follows that the measure $\mu$ also annihilates any field of the
form $\rho_m\,G(x-x_m)$
because, as the points $x_m$ do not belong to $S$,
\[
0=F(x_m)\cdot \rho_m=\int_{\RR^3} G(x-x_m)\,\rho_m\cdot d\mu(x)\,.
\]
Therefore 
\[
\int_{\RR^3}v\cdot d\mu=0\,,
\]
which implies that $v$ can be
uniformly approximated on~$S$ by elements of the subspace $\cU$ as a
consequence of the Hahn--Banach theorem. The lemma then follows.
\end{proof}

As an intermediate step before proving the global approximation result
for the Beltrami equation, we will establish the following proposition on the
approximation of solutions to the elliptic equation $\De v=-\la^2 v$
by solutions defined in a large ball. Throughout, we will say that a
differential equation holds in a closed set if it holds in a
neighborhood of this set.

\begin{proposition}\label{P.approx}
Let $v$ be a vector field which satisfies the equation
\begin{equation}\label{eqmv}
\De v=-\la^2 v
\end{equation}
in a compact subset $S$ of $\RR^3$. Assume that its complement
$\RR^3\minus S$ is connected and that $S$ is contained in the ball
$\BB_R$. Then for any $\de>0$ and any positive integer $k$ there is a vector field $w$ satisfying the equation
\[
\De w=-\la^2 w
\]
in  $\BB_R$ that approximates the field $v$ in $S$ as
\begin{equation}\label{eqmvw}
\|v-w\|_{C^k(S)}<\de\,.
\end{equation}
Here $\de$ is any fixed positive constant.
\end{proposition}
\begin{proof}
  By hypothesis, there is an open subset $\Om\supset S$ such that the
  field $v$ satisfies the equation~\eqref{eqmv} in $\Om$. We can
  assume that $\Om$ is contained in the ball $\BB_R$. Let us take a
  smooth function $\chi:\RR^3\to\RR$ equal to $1$ in a closed set $S'\subset\Om$
  whose interior contains $S$ and identically zero outside $\Om$. Defining
 a smooth extension $\tilde v$ of the field $v$ to $\RR^n$ by
  setting $\tilde v:=\chi v$, we obviously have
\begin{equation}\label{intbv}
\tilde v(x)=\int_{\RR^3} G(x-\tilde x)\, \rho(\tilde x)\, d\tilde x
\end{equation}
with $\rho:=-\De\tilde v-\la^2\tilde v$. 

The vector field $\rho$ is necessarily supported in
$\Om\minus S'$. Therefore, an easy continuity argument ensures that one
can approximate the integral~\eqref{intbv} in the compact set~$S'$ by a finite
Riemann sum of the form
\[
\hv(x):=\sum_{m=1}^M \rho_m\, G(x-x_m)
\]
so that, for any constant $\de'>0$,
\[
\|\tilde v-\hv\|_{C^0(S')}<\de'\,.
\]
Here $\rho_m$ are constant vectors in $\RR^3$ and $x_m$ are points that lie in
$\Om\minus S'$. 

Let us take a domain $U\subset \RR^3\minus
\BB_{2R}$. Lemma~\ref{L.approx} asserts that there is a vector field
of the form
\[
w(x):=\sum_{j=1}^J c_j\, G(x-z_j)
\]
such that
\[
\|\hat v-w\|_{C^0(S')}<\de'\,,
\]
where $\{z_j\}_{j=1}^J$ is a finite set of points in $U$ and $c_j\in\RR^3$ are
constant vectors. Therefore, 
\begin{equation}\label{eqmvwde}
\|v-w\|_{C^0(S')}<2\de'\,.
\end{equation}

To complete the proof of the proposition, notice that the field $v$ satisfies 
\[
\De v+\la^2 v=0
\]
in the set $S'$ (whose interior contains $S$) and $w$ satisfies the same equation in
the ball $\BB_{2R}$. By standard elliptic estimates, it follows that
the $C^0$ approximation~\eqref{eqmvwde} can be promoted to the $C^k$
bound
\[
\|v-w\|_{C^k(S)}<C_k\de'\,.
\]
Choosing $\de'$ small enough, the result follows.
\end{proof}

We are now ready to prove the global approximation theorem for the
Beltrami equation with solutions that decay at infinity. To construct
these solutions, we will truncate a suitable series representation for the
fields in a large ball obtained using Proposition~\ref{P.approx} and
act on them using a convenient differential operator.

\begin{theorem}\label{T.approx}
Let $v$ be a vector field that satisfies the Beltrami equation
\[
\curl v=\la v
\]
in a compact set $S\subset\RR^3$, where $\la$ is a nonzero constant and
the complement $\RR^3\minus S$ is connected. Then
there is a global Beltrami field $u$, satisfying the equation
\[
\curl u=\la u
\]
in $\RR^3$, which falls off at infinity as $|D^j u(x)|<C_j/|x|$ and
approximates the field $v$ in the $C^k$ norm as
\[
\|u-v\|_{C^k(S)}<\de\,.
\]
Here $\de$ is any positive constant.
\end{theorem}
\begin{proof}
Let us assume that the compact set $S$ is contained in the ball $\BB_{R/2}$. As the Beltrami field $v$ satisfies the equation
\[
\De v+\la^2 v=0
\]
in $S$, by Proposition~\ref{P.approx} there is a field $w$ satisfying 
\begin{equation}\label{eqwmi}
\De w+\la^2 w=0
\end{equation}
in the ball $\BB_R$ and such that
\begin{equation}\label{mivwC}
\|v-w\|_{C^{k+2}(S)}<\de'\,.
\end{equation}

Let us take spherical coordinates $(r,\te,\vp)$ in the ball
$\BB_R$. Writing the field $w$ as a series of spherical harmonics and
using the equation~\eqref{eqwmi} we immediately obtain that $w$ can be
written as a series
\[
w=\sum_{l=0}^\infty\sum_{m=-l}^l c_{lm}\, j_l(\la r)\, Y_{lm}(\te,\vp)\,.
\]
Here $j_l$ is the spherical Bessel function, $Y_{lm}$ are the
spherical harmonics and $c_{lm}\in\RR^3$ are constant
vectors. Therefore, given any $\de'>0$ there is an integer $L$
such that the finite sum
\[
\hu:=\sum_{l=0}^L\sum_{m=-l}^l  c_{lm}\, j_l(\la r)\, Y_{lm}(\te,\vp)
\]
approximates the field $w$ in $L^2$ sense:
\begin{equation}\label{huwL2}
\|\hu-w\|_{L^2(\BB_R)}<\de'\,.
\end{equation}
By the properties of spherical Bessel functions, the vector field $\hu$
satisfies the equation 
\begin{equation}\label{eqmhu}
\De\hu+\la^2\hu=0
\end{equation}
in $\RR^3$ and falls off at infinity as $|D^j\hu(x)|<C_j/|x|$.

In view of Eqs.~\eqref{eqwmi} and~\eqref{eqmhu}, standard elliptic
estimates allow us to pass from the $L^2$ bound~\eqref{huwL2} to
$C^{k+2}$ estimate
\[
\|\hu-w\|_{C^{k+2}(\BB_{R/2})}<C_k\de'\,.
\]
From this inequality and the bound~\eqref{mivwC} we infer
\begin{equation}\label{boundmi}
\|\hu-v\|_{C^{k+2}(S)}<C\de'\,.
\end{equation}

Let us now set
\[
u:=\frac{\curl\curl \hu+\la\curl \hu}{2\la^2}\,.
\]
A simple computation shows that the vector field thus defined
satisfies the Beltrami equation
\[
\curl u=\la u
\]
in $\RR^3$ and falls off as $|D^ju(x)|<C_j/|x|$ by the properties of $\hu$. Moreover,
\begin{align*}
\|u-v\|_{C^k(S)}&=\bigg\| \frac{\curl\curl \hu+\la\curl \hu}{2\la^2}-
v\bigg\|_{C^k(S)}\\
&=\bigg\| \frac{(\curl+\la)\curl
  (\hu-v)}{2\la^2}\bigg\|_{C^k(S)} \\
&\leq C\|\hu-v\|_{C^{k+2}(S)}\\
&< C\de'\,,
\end{align*}
as we wanted to prove.
\end{proof}

\begin{remark}
  The fall-off at infinity of the global Beltrami field $u$ is
  obtained from the truncation of the explicit series representation for the auxiliary field
  $w$. This is the reason why Theorem~\ref{T.approx} does not work in
  arbitrary open Riemannian $3$-manifolds, unlike the approximation
  theorem we used in~\cite{Annals}. Notice that the latter theorem
  does not yield any control at infinity whatsoever for the global
  Beltrami fields.
\end{remark}

\section{Proof of the main theorem}
\label{S.main}

We are now ready to give the proof of Theorem~\ref{T.main}. Let us
begin by considering one of the curves (say, $\ga_1$) in the statement of the
theorem. By perturbing this curve with a
$C^m$-small diffeomorphism if necessary, we can assume that
$\ga_1$ is an analytic curve whose curvature $\ka$ does not vanish
anywhere~\cite[p.\ 184]{Bruce}, and that its total torsion satisfies
\[
\int_0^\ell \tau(\al)\, d\al\neq \pi k
\]
for all integers $k$. As before, $\ell$ denotes the length of the
curve $\ga_1$. It is worth emphasizing that, as these conditions are
open, they are obviously preserved if we deform the curve $\ga_1$
with a diffeomorphism that is close enough to the identity in the
$C^m$ norm ($m\geq3$).

Let us consider the tube $\cT_\ep(\ga_1)$ of core curve $\ga_1$ and thickness~$\ep$, and the corresponding harmonic field $h$, given by the
expression~\eqref{h} in the coordinates~$(\al,y)$ adapted to the
tube. Throughout we will assume that $\ep$ is small enough. By
Theorem~\ref{T.estimBeltrami}, we can consider the associated local Beltrami field $v_1$, which is
given by the only solution to the Beltrami equation with parameter $\la=\ep^3$,
\[
\curl v_1=\ep^3v_1\,,
\]
in the tube $\cT_\ep(\ga_1)$ that is tangent to the boundary and whose
harmonic part is the field $h$. (Notice that in this section $v_1$ will
{\em not}\/ stand for the component of a vector field in the direction
of the coordinate $y_1$.) Since the boundary of the tube is analytic, it is well known~\cite{Mo58}
that the local Beltrami field $v_1$ is analytic in the closure of a
small neighborhood $\Om_1$ of~$\overline{\cT_\ep(\ga_1)}$. 

Since the total torsion of the curve is not an integral multiple of
$\pi$, Theorem~\ref{T.KAM} and Proposition~\ref{P.harmelliptic} ensure that, given any $\de>0$, we can
deform the curve $\ga_1$ by a diffeomorphism of $\RR^3$ arbitrarily close to the
identity in the $C^m$ norm so that any divergence-free vector field
$u$ in $\Om_1$ with 
\[
\|u-v_1\|_{C^k(\Om_1)}<\de'
\]
has: 
\begin{itemize}
\item An invariant tube given by $\Psi_1[\cT_\ep(\ga_1)]$.
\item An elliptic periodic trajectory given by $\tilde\Psi_1(\ga_1)$.
\end{itemize}
Moreover, on the invariant torus $\pd\Psi_1[\cT_\ep(\ga_1)]$, the field
$u$ is orbitally conjugate to a Diophantine rotation, and therefore ergodic. Here $\Psi_1$ and $\tilde\Psi_1$ are diffeomorphisms of $\RR^3$ with
\[
\|\Psi_1-\id\|_{C^m(\RR^3)} + \|\tilde \Psi_1-\id\|_{C^m(\RR^3)}<\de
\]
and such that the differences $\Psi_1-\id$ and $\tilde \Psi_1-\id$ are
supported on small neighborhoods $U_1,\tilde U_1$ of $\pd\cT_\ep(\ga_1)$ and
$\ga_1$, respectively. The constants $k$ and $\de'$ depend on $m$, $\de$
and on the geometry of the curve $\ga_1$.

We can apply the same argument for each curve $\ga_i$ ($1\leq i\leq
N$), thereby obtaining (for small enough $\ep$) a collection of local
Beltrami fields $v_i$, satisfying the equation
\[
\curl v_i=\ep^3 v_i
\]
in the closure of a neighborhood $\Om_i$ of the closed tube
$\overline{\cT_\ep(\ga_i)}$ and such that any  divergence-free vector field
$u$ in $\Om_i$ with $\|u-v_i\|_{C^k(\Om_i)}<\de'$ has an
invariant torus, where the field $u$ is ergodic, and an elliptic periodic trajectory. Furthermore, they are
respectively given by $\pd\Psi_i[\cT_\ep(\ga_i)]$ and
$\tilde\Psi_i(\ga_i)$, where $\Psi_i,\tilde\Psi_i$ are $C^m$-small
diffeomorphisms of $\RR^3$ with $\Psi_i-\id$ and $\tilde \Psi_i-\id$
supported in small neighborhoods $U_i,\tilde U_i$ of $\pd\cT_\ep(\ga_i)$ and $\ga_i$, in
each case. We can assume that the complement 
\[
\RR^3\minus(\Om_1\cup \cdots\cup
\Om_N)
\]
is connected and that the sets $\overline{\Om_i}$ are pairwise disjoint.

Let us define a vector field $v$ in $\overline{\Om_1}\cup \cdots\cup 
\overline{\Om_N}$ by setting
it equal to the local Beltrami field $v_i$ in each set
$\overline{\Om_i}$. By Theorem~\ref{T.approx}, there is a Beltrami
field $u$, which satisfies the equation
\[
\curl u=\ep^3u
\]
in $\RR^3$, that falls off at infinity as $|D^ju(x)|<C_j/|x|$  and approximates the field $v$ as
\[
\|u-v\|_{C^k(\Om_1\cup \cdots\cup \Om_N)}<\de'\,.
\]
Therefore, if we
define the diffeomorphism $\Phi$ of $\RR^3$ as
\[
\Phi(x):=\begin{cases}
\Psi_i(x)& \text{if } x\in U_i\,,\\
\tilde\Psi_i(x)& \text{if } x\in \tilde U_i\,,\\
x & \text{otherwise}\,,
\end{cases}
\]
it follows that, for each $i$, $\Phi[\cT_\ep(\ga_i)]$ is a vortex tube
of the Beltrami field $u$, and that $\Phi(\ga_i)$ is an elliptic
periodic trajectory. Besides, the Beltrami field $u$ is ergodic (and
orbitally conjugate to a Diophantine rotation) on each invariant torus
$\pd \Phi[\cT_\ep(\ga_i)]$.

The field $u$ being orbitally conjugate to a Diophantine rotation on
each invariant torus $\pd \Phi[\cT_\ep(\ga_i)]$, it
follows~\cite{FK09} that this invariant torus is accumulated by a
Cantor-like set of invariant tori with positive Lebesgue measure. On
these invariant tori, the field is also orbitally conjugate to
Diophantine rotations. The corresponding set of Diophantine
frequencies is Cantor-like because the normal torsion $\cN_\Pi$ is
nonzero. Therefore, if we consider the trajectories of the field $u$
between two of these invariant tori (lying on the same vortex tube
$\Phi[\cT_\ep(\ga_i)]$), Angenent's dichotomy~\cite{Angenent} asserts that either there is a horseshoe-type
invariant set between them or there are invariant tori where the field
is conjugate to a rotation of rational frequency. In both cases, there
are infinitely many periodic trajectories between these tori. The
theorem then follows.

\begin{remark}
It is worth giving some additional
details about what we understand by a thin tube. What we have proved is that for any set of
smooth periodic curves $\ga_1,\dots,\ga_N$, $\de>0$ and $m\in\NN$, there is a
positive constant $\ep_0$, depending on $\de$, $m$ and the geometry of
the curves, such that the statement of Theorem~\ref{T.main} holds true
for any thickness $\ep<\ep_0$, the diffeomorphism $\Phi$ that maps the tubes
$\cT_\ep(\ga_i)$ into vortex tubes being bounded by
$\|\Phi-\id\|_{C^m(\RR^3)}<\de$. It should be noticed that both
Theorem~\ref{T.main} and its proof hold verbatim if one takes tubes
$\cT_{\ep_1}(\ga_1),\dots, \cT_{\ep_N}(\ga_N)$ of different thickness,
as long as $\ep_1,\dots,\ep_N$ are small enough (that is, smaller than
the above constant $\ep_0$).
\end{remark}

\begin{remark}\label{R.locallyfinite}
  Contrary to what happened in the main theorem of~\cite{Annals}, the
  proof of Theorem~\ref{T.main} does not work in general (even if we drop the
  requirement that the field~$u$ decays at infinity, replacing
  Theorem~\ref{T.approx} by~\cite[Theorem 3.6]{Annals}) if we
  substitute the finite set of curves $\{\ga_i\}_{i=1}^N$ by an
  infinite set $\{\ga_i\}_{i=1}^\infty$ that is locally finite. The
  reason is that the ``maximal thickness'' $\ep_0(\ga_i)$
  associated to each curve~$\ga_i$ individually does not need to be
  bounded away from zero: since all the vector fields~$v_i$, defined in a neighborhood of
  the tube $\overline{\cT_{\ep_i}(\ga_i)}$ (whose thickness can vary
  from tube to tube), must satisfy the Beltrami
  equation
\[
\curl v_i=\la v_i
\]
with the {\em same}\/ constant $\la$ for all $i$. Since this $\la$ must be of order $\cO[\ep_0(\ga_i)^3]$
for all $i$, it is clear that the construction breaks down if the
infimum of the positive quantities $\{\ep_0(\ga_i)\}_{i=1}^\infty$ is $0$. On the other hand,
if this infimum is positive, we can apply the approximation theorem
in~\cite[Theorem 3.6]{Annals} to construct a global Beltrami field
with all the properties listed in the statement of
Theorem~\ref{T.main} with the exception that its growth at infinity is
not controlled. 
\end{remark}

\section{A comment about the Navier--Stokes equation}
\label{S.remarks}

To conclude, we will present an easy application of
Theorem~\ref{T.main} to the existence of (time-dependent) solutions to
the Navier--Stokes equation that have a prescribed set of stationary (possibly
knotted and linked) vortex tubes.

For this, let us take the global Beltrami field $u$ that we considered
in Section~\ref{S.main}. That is, $u$ satisfies the equation
\[
\curl u=\ep^3u
\]
in $\RR^3$, falls off at infinity as $C/|x|$ and has a set of thin
invariant tubes given by $\{\Phi[\cT_\ep(\ga_i)]\}_{i=1}^N$, where the
diffeomorphism $\Phi$ is close to the identity and $\ep$ is a small
constant. Then the analytic time-dependent field
\[
w(x,t):=u(x)\,\e^{-\nu\ep^6t}
\]
is a solution of the Navier--Stokes equation
\[
\pd_t w+(w\cdot\nabla)w =\nu\De w-\nabla P\,,\qquad \Div w=0
\]
in $\RR^3$ with pressure $P(x,t)=c-\frac12|w(x,t)|^2$. As the vortex
lines of $w$ (which are the trajectories of the vorticity $\curl
w(x,t)$ for fixed $t$) coincide with those of $u$ at all times, up to
a reparametrization, it follows that $w$ is a solution to the
Navier--Stokes equation with the desired properties.

\section*{Acknowledgments}

This work is supported in part by the Spanish MINECO under
grants~FIS2011-22566 (A.E.) and the ICMAT Severo
Ochoa grant SEV-2011-0087 (A.E.\ and D.P.S.). The authors are
respectively supported by the Ram\'on y Cajal program (A.E.) and an
ERC starting grant 335079 (D.P.S.).

\end{document}